\documentclass[11pt]{amsart}
\usepackage{amscd,amsmath,amssymb,amsfonts,verbatim}
\usepackage[cmtip, all]{xy}
\usepackage{MnSymbol}
\usepackage{comment}

\newtheorem{thm}{Theorem}[section]
\newtheorem{prop}[thm]{Proposition}
\newtheorem{lem}[thm]{Lemma}
\newtheorem{cor}[thm]{Corollary}

\theoremstyle{definition}
\newtheorem{ques}[thm]{Question}

\newtheorem{defn}[thm]{Definition}
\theoremstyle{remark}
\newtheorem{remk}[thm]{Remark}
\newtheorem{remks}[thm]{Remarks}

\newtheorem{exm}[thm]{Example}
\newtheorem{exms}[thm]{Examples}
\newtheorem{notat}[thm]{Notation}
\numberwithin{equation}{section}

{\hfill$\square$\end{defn}}
{\hfill$\square$\end{remk}}
{\hfill$\square$\end{remks}}
{\hfill$\square$\end{exm}}
{\hfill$\square$\end{exms}}
{\hfill$\square$\end{notat}}

\newcommand{\thmref}{Theorem~\ref}
\newcommand{\propref}{Proposition~\ref}
\newcommand{\corref}{Corollary~\ref}

\newcommand{\lemref}{Lemma~\ref}

\newcommand{\refone}{{\ensuremath 1}}
\newcommand{\reftwo}{{\ensuremath 2}}

\newcommand{\sC}{{\mathcal C}}
\newcommand{\sD}{{\mathcal D}}

\newcommand{\sF}{{\mathcal F}}

\newcommand{\sH}{{\mathcal H}}
\newcommand{\sI}{{\mathcal I}}

\newcommand{\sK}{{\mathcal K}}
\newcommand{\sL}{{\mathcal L}}

\newcommand{\sO}{{\mathcal O}}

\newcommand{\sR}{{\mathcal R}}

\newcommand{\sX}{{\mathcal X}}
\newcommand{\sY}{{\mathcal Y}}
\newcommand{\sZ}{{\mathcal Z}}
\newcommand{\A}{{\mathbb A}}

\newcommand{\F}{{\mathbb F}}
\newcommand{\G}{{\mathbb G}}

\newcommand{\N}{{\mathbb N}}

\newcommand{\Q}{{\mathbb Q}}

\newcommand{\T}{{\mathbb T}}

\newcommand{\W}{{\mathbb W}}

\newcommand{\Z}{{\mathbb Z}}

\newcommand{\fm}{{\mathfrak m}}

\newcommand{\fp}{{\mathfrak p}}

\newcommand{\fl}{{\mathfrak l}}
\newcommand{\ff}{{\mathfrak f}}

\newcommand{\aj}{{\rm AJ}}
\newcommand{\lw}{{\rm lw}}

\newcommand{\Ker}{{\rm Ker}}

\newcommand{\CH}{{\rm CH}}

\newcommand{\surj}{\twoheadrightarrow}
\newcommand{\inj}{\hookrightarrow}
\newcommand{\red}{{\rm red}}

\newcommand{\Pic}{{\rm Pic}}
\newcommand{\Div}{{\rm Div}}
\newcommand{\Hom}{{\rm Hom}}

\newcommand{\Spec}{{\rm Spec \,}}
\newcommand{\sing}{{\rm sing}}
\newcommand{\Char}{{\rm char}}

\newcommand{\Tr}{{\rm Tr}}

\newcommand{\divf}{{\rm div}}

\newcommand{\id}{{\operatorname{id}}}

\newcommand{\pfd}{{\operatorname{\mathbf{Pfd}}}} 
\newcommand{\Sch}{{\operatorname{\mathbf{Sch}}}}

\newcommand{\Top}{{\mathbf{Top}}}

\newcommand{\<}{\langle}
\renewcommand{\>}{\rangle}

\newcommand{\Sm}{{\mathbf{Sm}}}

\newcommand{\can}{{\operatorname{\rm can}}}

\newcommand{\Ab}{{\mathbf{Ab}}}

\newcommand{\cyc}{{\operatorname{\rm cyc}}}

\newcommand{\Sym}{{\operatorname{\rm Sym}}}

\newcommand{\et}{{\text{\'et}}}

\newcommand{\ds}{{/\kern-3pt/}}

\newcommand{\res}{{\operatorname{res}}}

\renewcommand{\log}{{\operatorname{log}}}

\newcommand{\Tor}{{\operatorname{Tor}}}

\newcommand{\Br}{{\operatorname{Br}}}
\newcommand{\sm}{{\operatorname{sm}}}
\newcommand{\Proj}{{\operatorname{Proj}}}

\newcommand{\lci}{{\rm l.c.i.\!}}

\newcommand{\Picc}{{\mathbf{Pic}}}

\newcommand{\tr}{{\operatorname{tr}}}

\newcommand{\un}{\underline}
\newcommand{\ov}{\overline}

\renewcommand{\dim}{\text{\rm dim}}

\newcommand{\tuborg}{\left\{\begin{array}{ll}}
\newcommand{\sluttuborg}{\end{array}\right.}

\newcommand{\zar}{{\rm zar}}
\newcommand{\nis}{{\rm nis}}

\newcommand{\irr}{{\rm Irr}}
\newcommand{\reg}{{\rm reg}}

\newcommand{\tor}{{\rm tor}}

\newcommand{\dlog}{{\rm dlog}}
\newcommand{\ft}{{\rm ft}}
\newcommand{\cf}{{\rm cf}}
\newcommand{\pf}{{\rm pf}}
\newcommand{\inv}{{\rm inv}}
\newcommand{\alb}{{\rm alb}}
\newcommand{\Cores}{{\rm Cores}}
\newcommand{\Ress}{{\rm Res}}
\newcommand{\dt}{{\rm dt}}
\newcommand{\divv}{{\rm div}}

\newcommand{\Sw}{{\rm sw}}

\newcommand{\wt}{\widetilde}
\newcommand{\wh}{\widehat}

\newcommand{\coker}{{\rm Coker}}
\newcommand{\Fil}{{\rm fil}}

\newcommand{\Tab}{{\mathbf {Tab}}}

\newcommand{\etl}{{\acute{e}t}}

\newcounter{elno}

\newcounter{elno-abc}   

\newcounter{elno-abc-prime}

\begin{document}
\title{Duality for cohomology of split tori on curves over local fields}
\author{Amalendu Krishna, Jitendra Rathore, Samiron Sadhukhan}
\address{Department of Mathematics, Indian Institute of Science,  
Bangalore, 560012, India.}
\email{amalenduk@iisc.ac.in}
\address{School of Mathematics, Tata Institute of Fundamental Research, Homi Bhabha
  Road, Mumbai-400005, India.}
\email{jitendra@math.tifr.res.in}
\address{School of Mathematics, Tata Institute of Fundamental Research, Homi Bhabha
  Road, Mumbai-400005, India.}
\email{samiron@math.tifr.res.in}




\maketitle

\begin{quote}\emph{Abstract.}
  We prove duality theorems for the {\'e}tale cohomology of logarithmic Hodge-Witt
  sheaves and split tori on smooth curves over a local field of positive
  characteristic. As an application, we obtain a description of the Brauer group of the
  function fields of curves over local fields in terms of the characters of the
  idele groups. We also show that the classical Brauer-Manin pairing between the Brauer
  and Picard groups of smooth projective curves over local fields has analogues for
  arbitrary smooth curves, smooth projective curves with modulus and singular projective
  curves over such fields. 
  
\end{quote}
\setcounter{tocdepth}{1}
\tableofcontents

\section{Introduction}\label{sec:Intro}
The duality theorems of Tate (see \cite{Tate-1}, \cite{Tate-2}) and Lichtenbaum
\cite{Lichtenbaum} for the {\'e}tale cohomology of split tori over $p$-adic fields and
smooth projective curves over such fields are fundamental results in
arithmetic geometry. In recent years, these duality theorems have
been extended by Scheiderer and van Hamel \cite{Sh-Hamel}
(see also \cite{Harari-Szamuely} and \cite{Yamazaki})
to more general settings of smooth affine curves over $p$-adic fields.
These generalizations have found many applications,
especially in the study of local-global principles for cohomology of commutative group
schemes over the function fields of curves over $p$-adic fields.

The analogue of Tate's duality over local fields of
positive characteristics was proven by Milne \cite{Milne-Duality} while the analogue of
Lichtenbaum duality over such fields was proven
by Saito \cite{Saito-Invent} (see also \cite{Milne-Duality}).
However, the analogues of the results of Scheiderer and van Hamel for smooth affine
curves and their function fields
over local fields of positive characteristics are currently unknown.
This paper is an attempt to fill this gap.

To obtain these generalizations, we prove a new duality theorem for the logarithmic
Hodge-Witt cohomology on smooth projective curves over a local field of positive
characteristic. We introduce Brauer group with modulus, and extend the classical
Brauer-Manin pairing to the setting of 0-cycles and Brauer groups with modulus over
an arbitrary local field.
The duality theorem for the  logarithmic Hodge-Witt cohomology on smooth projective
varieties over finite fields was established long ago by Milne \cite{Milne-Zeta}. Over
a local field, a partial result was obtained by Kato-Saito \cite{Kato-Saito-Ann}.
Below, we describe the main results of this paper.

\subsection{Duality for cohomology of $\G_m$}\label{sec:MR**}
We fix a local field (i.e., a complete discrete valuation field with finite residue
field) $k$ of characteristic $p > 0$.
Let $X$ be a geometrically connected smooth projective
curve over $k$ and let $j \colon X^o \inj X$ be a dense open immersion.
Let $\iota \colon D \inj X$ be the inclusion of an effective Cartier divisor
whose support is the complement of $X^o$. We let $\G_{m}$ denote the rank one split 
torus over $k$. We let $H^q_{cc}(X^o, \G_m) = {\varprojlim}_n \
H^q_\et(X, \G_{m,(X,nD)})$, where $\G_{m,(X,nD)} = \Ker(\G_{m,X} \surj \iota_*(\G_{m,nD}))$.
This is an enriched version of the usual cohomology with compact support
$H^q_{\et, c}(X^o, \G_m)$ of $X^o$ in positive characteristic. One easily checks that
the canonical map $H^q_{\et, c}(X^o, \G_m) \to H^q_{\et}(X, \G_m)$ factors through
$H^q_{cc}(X^o, \G_m)$.

In characteristic zero, the duality theorem for $\G_m$
on $X^o$ (e.g., see \cite{Sh-Hamel}) is essentially insensitive to the topology of the
underlying {\'e}tale cohomology groups because it holds true if we simply endow these
groups with discrete topology and pass to their profinite completions. This makes
the proof of duality relatively simpler. In contrast,
endowing each cohomology group with correct topology is a challenging part of the proof
of the duality theorem in positive characteristic. We briefly explain the topologies
that we use and refer to the body of the paper for details.

Recall that $\Pic(X)$ is the group of
$k$-points of the Picard scheme $\Picc(X)$. Since the latter is a locally of finite type
$k$-scheme, $\Pic(X)$ is endowed with the adic topology induced by the
valuation topology of $k$. Since $\Pic(X^o)$ is a quotient of $\Pic(X)$, it is equipped
with the quotient topology. As a key step in our proofs, we show in this
paper that the relative Picard group
$\Pic(X|D)$ is also equipped with the adic topology such that the canonical
map $\Pic(X|D) \surj \Pic(X)$ is a topological quotient.
We endow $H^1_{cc}(X^o, \G_m)$ with
the inverse limit topology using a canonical isomorphism
$\Pic(X|D) \cong H^1_\et(X, \G_{m,(X,D)})$.
All other cohomology groups will be endowed with the discrete topology.

For a Hausdorff topological abelian group $G$ with the topology $\tau$, we let
$G^{\pf}$ denote the profinite $\tau$-completion of $G$. That is, ${G}^\pf$ is
the inverse limit ${\varprojlim}_U \ {G}/{U}$ with the inverse limit topology,
where $U$ runs through $\tau$-open subgroups of finite index in $G$. It is clear that
${G}^{\pf}$ is a profinite abelian group. We refer to \S~\ref{sec:Duality-X} for the
definitions of non-degenerate and perfect pairings of topological abelian groups.
In this paper, we prove the following duality theorem for {\'e}tale cohomology.

\begin{thm}\label{thm:Main-0}
  For every integer $q \neq 0$, there is a bilinear pairing
  \[
    H^q(X^o, \G_m) \times H^{3-q}_{cc}(X^o, \G_m) \to {\Q}/{\Z}
  \]
  which induces perfect pairings of topological abelian groups
  \[
    H^1(X^o, \G_{m})^{\pf} \times H^2_{cc}(X^o, \G_{m}) \to {\Q}/{\Z},
    \]
    \[
      H^2(X^o, \G_{m}) \times H^1_{cc}(X^o, \G_{m})^{\pf} \to {\Q}/{\Z}
      \]
      and
      \[
      H^3(X^o, \G_{m}) \times H^0_{cc}(X^o, \G_{m})^{\pf} \to {\Q}/{\Z}.  
      \]
      \end{thm}

For smooth curves over $p$-adic fields, this theorem is
      due to Lichtenbaum \cite{Lichtenbaum} in the proper case, and to
      Scheiderer-van Hamel \cite{Sh-Hamel} (see also \cite{Harari-Szamuely}) in the open
      case. In positive characteristic, the $q =1$ case of the theorem is
      due to Saito \cite{Saito-Invent}  (see also \cite{Milne-Duality}) in the proper
      case. 

      \begin{remk}\label{remk:q=0}
        One can construct the pairing of \thmref{thm:Main-0} for
      $q = 0$ as well but it will not be perfect (not even non-degenerate).
    \end{remk}

      \subsection{Duality for cohomology of $\G_{m,K}$}\label{sec:MR***}
      Let $X$ be as above and let $K$ denote the function field of $X$.
      We let $\wh{I}(X) = {\underset{x \in X_{(0)}}{\prod'}} \wh{K}^\times_x$ denote the
      restricted product with respect to the subgroups $(\wh{\sO_{X,x}})^\times$,
      where $\wh{K}_x$ is the quotient field of $\wh{\sO_{X,x}}$.
      The idele class group $C(K)$ is the cokernel of the canonical inclusion
      $K^\times \inj \wh{I}(X)$.
      This group is endowed with the inverse limit of the adic topologies of the
      relative Picard groups of $X$ (cf. \S~\ref{sec:Dual-K}). We let $\Br(K)$ have the
      discrete topology. As an application of \thmref{thm:Main-0}, we prove the
      following.

     \begin{thm}\label{thm:Main-10} 
There is a perfect pairing of topological abelian groups 
 \[
   \Br(K) \times C(K)^{\pf} \to {\Q}/{\Z}.
 \]
\end{thm}

This result provides an explicit description of $\Br(K)$ in terms of the
characters of the idele class group. In fact, we prove a duality theorem
for all {\'e}tale cohomology groups of $\G_{m,K}$ in this paper (cf. \thmref{thm:Main-6}).
When $K$ is the quotient field of an excellent normal 2-dimensional complete local domain
with finite residue field, an analogue of \thmref{thm:Main-10} was shown by
Saito \cite{Saito-Invent}.

\subsection{Duality for logarithmic Hodge-Witt cohomology}
      \label{sec:DHW}
      The proof of \thmref{thm:Main-0} is based on the duality theorem for the $p$-adic
      {\'e}tale motivic cohomology groups of smooth projective curves over $k$, which
      we now describe. We fix an integer $n \ge 1$.
      Let $X$ be as above and let $W_n\Omega^\bullet_{X,\log}$ denote the
      logarithmic Hodge-Witt complex on $X_\et$ {\`a} la Bloch-Deligne-Illusie.

      For smooth projective schemes (of arbitrary dimensions) over a finite field, Milne
      showed that the logarithmic Hodge-Witt cohomology (i.e., the
      {\'e}tale cohomology of $W_n\Omega^\bullet_{X, \log}$) groups satisfy 
      Poincar{\'e} duality. However, the analogous result over local fields (of
      positive characteristics) is presently unknown. The following result settles this
      problem for curves.

      To state the duality theorem, we need to recall that over finite fields,
      Milne's duality yields a perfect pairing of finite groups. However, the
     logarithmic Hodge-Witt cohomology groups over local fields are generally
     not finite. To offset this problem, we need to equip these groups with
     suitable non-discrete topologies. In fact, dealing with the infinitude of
     cohomology groups and choosing correct topology on them are perhaps the major part
     of proving a duality theorem over local fields. For the topologies that we
     endow these groups with, the reader is referred to \S~\ref{sec:Top-X}. The following
     theorem was proven by Kato-Saito \cite[Prop.~4]{Kato-Saito-Ann}
     when $j = 0$ by a different method.

\begin{thm}\label{thm:Main-5}
        Let $X$ be as in \thmref{thm:Main-0} and let $q, j$ be any integers. Then the cup
        product on the {\'e}tale cohomology of the logarithmic Hodge-Witt sheaves
        induces a perfect pairing of topological abelian groups
        \[
          H^q(X, W_n\Omega^j_{X,\log}) \times H^{2-q}(X, W_n\Omega^{2-j}_{X,\log})
          \to {\Q}/{\Z}.
        \]
      \end{thm}

\begin{remk}\label{remk:Dual-gen-sm}
        We actually prove this theorem under weaker assumptions on $X$. Namely, we 
        need to assume that $X$ is regular but do not require it to be smooth over $k$.
        We work under the weaker hypothesis that $X$ is generically smooth over $k$.
       The validity of
        the theorem under generic smoothness is useful for some concrete applications.
        For instance, we hope that it will allow one to prove the duality theorem for the
        $p$-adic {\'e}tale motivic cohomology groups with compact support for smooth
        but non-proper curves. Such a result will have applications to the
        class field theory of smooth open curves over local fields. This is the topic of
        \cite{KM}.
      \end{remk}

\vskip .2cm

      \subsection{Brauer-Manin pairing for modulus pairs}\label{sec:Appln}
      Let $\iota \colon D \inj X$ be the inclusion of an effective Cartier divisor
      as above. We shall refer to the pair $(X,D)$ as a (1-dimensional) modulus pair.
      Let $\CH_0(X|D)$ be the Chow group of 0-cycles for the modulus pair $(X,D)$.
      We shall show in this paper that $\CH_0(X|D)$ is naturally endowed with 
      an adic topology such that the canonical map
      $\CH_0(X|D) \surj \CH_0(X)$ is a topological quotient.
      In order to extend the classical Brauer-Manin pairing to the modulus setting, we
      introduce the Brauer group of the modulus pair $(X,D)$ which we denote by
      $\Br(X|D)$ (see Definition~\ref{defn:BGM-0}).
      This is a certain subgroup of $\Br(X^o)$ which is
      defined in terms of Kato's filtration on the {\'e}tale motivic cohomology group
      $H^2(K, {\Q}/{\Z}(1))$, where $K$ is the function field of $X$. We endow
      $\Br(X|D)$ with the discrete topology. As a key step for proving
      \thmref{thm:Main-0}, we establish the following new Brauer-Manin pairing.

      \begin{thm}\label{thm:Main-1}
        There is a continuous pairing of topological abelian groups
        \[
          \Br(X|D) \times \CH_0(X|D) \to {\Q}/{\Z}
        \]
        such that the induced pairing
        \[
          \Br(X|D) \times \CH_0(X|D)^{\pf} \to {\Q}/{\Z}
        \]
        is perfect.
        \end{thm}
      
        The case $D = \emptyset$ of this theorem is due to Lichtenbaum \cite{Lichtenbaum}
        in characteristic zero and Saito \cite{Saito-Invent}
        (see also \cite[Thm.~III.7.8]{Milne-Duality}) in positive characteristic.
        In this paper, we provide an independent and short proof of the harder part of
        Saito's result as an
        immediate corollary of \thmref{thm:Main-5} (see \corref{cor:LS-thm}).

  \subsection{Brauer-Manin pairing for singular curves}
        \label{sec:Appln-0}
        Recall that the Brauer group of a singular variety $Y$ does not satisfy
 the Brauer injectivity property, i.e., the canonical map $\Br(Y) \to \Br(K)$ fails
 to be injective in general, for $K$ the function field of $Y$.
 This defect is analogous to the
 one the classical Chow group (one defined in \cite{Fulton}) $\CH_0(Y)$ suffers from. 
 In this paper, we introduce a refined Brauer group of $Y$ which we denote by
 $\Br^{\lw}(Y)$. We refer to this as the Levine-Weibel Brauer group.
 The refined Brauer group satisfies some nice properties, including the Brauer
 injectivity property.

 Let $\CH^{\lw}_0(Y)$ denote the Levine-Weibel Chow group of
 $Y$. One may recall that this coincides with $\Pic(Y)$ if $\dim(Y) = 1$.
 We show in this paper that $\CH^{\lw}_0(Y)$ is naturally endowed with an
 adic topology such that the pull-back map $\CH^{\lw}_0(Y) \to \CH_0(Y_n)$
 is continuous if $Y$ is a geometrically integral
 projective curve over $k$ with smooth normalization $Y_n$.
 We endow $\Br^{\lw}(Y)$ with the discrete topology and prove the following extension of
 the Brauer-Manin pairing to singular curves. 

\begin{thm}\label{thm:Main-3}
  Let $Y$ be a geometrically integral projective curve over a local field $k$ whose
  normalization is smooth over $k$. Then there is a continuous pairing of topological
  abelian groups
        \[
          \Br^{\lw}(Y) \times \CH^{\lw}_0(Y) \to {\Q}/{\Z}
        \]
        such that the induced pairing
        \[
          \Br^{\lw}(Y) \times \CH^{\lw}_0(Y)^{\pf} \to {\Q}/{\Z}
        \]
        is perfect.
        \end{thm}

\subsection{A brief outline of proofs}\label{sec:Outline}    
The main ingredients for proving \thmref{thm:Main-0} are
Theorems~\ref{thm:Main-5} and ~\ref{thm:Main-1}. A key step in the
proof of Theorem~\ref{thm:Main-5} is a duality theorem of Zhao \cite{Zhao}.
In \S~\ref{sec:Prelim}, we recall this duality theorem. In \S~\ref{sec:Duality},
we equip the logarithmic Hodge-Witt cohomology of a smooth projective curve $X$ over
a local field $k$ of positive characteristic with a suitable topology
using some deductions. We construct
a trace homomorphism for the top logarithmic Hodge-Witt cohomology of $X$ in
\S~\ref{sec:HWD} and complete the proof of \thmref{thm:Main-5}.
A trace map in this case was also constructed by Kato-Saito
\cite[\S~3, Prop.~4]{Kato-Saito-Ann}. But our trace homomorphism is, a priori, different.
It has an advantage that one can easily show its compatibility with the
trace homomorphisms for closed points under the Gysin homomorphisms.
This compatibility is crucial in our proofs.
In the next three sections, we equip the relative Picard group
of modulus pairs $(X,D)$ with the adic topology. We also prove some key properties of the
Albanese map for $X$ and the Pontryagin dual of the relative Picard group.

We introduce the Brauer group of modulus pairs in \S~\ref{sec:BM} and establish the
Brauer-Manin pairing between the Brauer and Chow groups of a modulus pair in
\S~\ref{sec:BM-mod**}. One of the key steps in the proof of  ~\thmref{thm:Main-1}
is a duality theorem for the cohomology of $\G_m$ on $X$ and its compatibility with
Kato's duality for 2-dimensional local fields obtained by completing the function
field of $X$ at closed points. This is done in \S~\ref{sec:LGD}.
In \S~\ref{sec:CBMP}, we complete another step for proving \thmref{thm:Main-1},
namely, we show that the Brauer-Manin paring for a modulus pair $(X,D)$ is
continuous. We complete the proof of ~\thmref{thm:Main-1} in \S~\ref{sec:PBMP}, and
then apply it to complete the proofs of Theorems~\ref{thm:Main-0} and
\thmref{thm:Main-10}.
In \S~\ref{sec:BMS}, we introduce a refined version of the Brauer group of singular
varieties and prove \thmref{thm:Main-3} as an application of \thmref{thm:Main-1}.

\subsection{Notations}\label{sec:Notn}
We shall work over a field $k$ of characteristic $p \ge 0$ throughout this
paper. We shall let $k_s$ (resp. $\ov{k}$) denote a fixed separable (resp. algebraic)
closure of $k$. We let $\Sch_k$ denote the category of 
separated Noetherian $k$-schemes and $\Sm_k$ the category of smooth (in particular,
finite type) $k$-schemes. A $k$-scheme will mean an object of $\Sch_k$ and a finite type
$k$-scheme will mean an object of $\Sch^{\ft}_k$.
The product $X \times_{\Spec(k)} Y$ in $\Sch_k$ will be written as
$X \times Y$. We let $X^{(q)}$ (resp. $X_{(q)}$) denote the set of points on $X$ having
codimension (resp. dimension) $q$. We shall let $X_\sing$ (resp. $X_\reg$) denote the
singular locus (resp. regular locus) of $X$ with the reduced closed subscheme
structure. We let $\sZ_0(X)$ denote the free abelian group of 0-cycles on $X$.

We let $\Sch_{k/\zar}$ (resp. $\Sch_{k/\nis}$, resp. $\Sch_{k/ \etl})$
denote the Zariski (resp. Nisnevich, resp. {\'e}tale) site of $\Sch_{k}$.
Unless we mention the topology specifically,
all cohomology groups in this paper will be considered with respect to the
{\'e}tale topology.
We shall let $cd_p(X)$ denote the {\'e}tale $p$-cohomological dimension of $X \in \Sch_k$
if $\Char(k) = p > 0$. We shall let $cd(X)$ denote the {\'e}tale cohomological
dimension of torsion sheaves on $X$. We shall let $G_k$ denote the abelian
absolute Galois group of $k$.

For an abelian group $A$, we shall write $\Tor^1_{\Z}(A, {\Z}/n)$ as
$_nA$ and $A/{nA}$ as $A/n$. We shall let $A\{p'\}$ denote the
subgroup of elements of $A$ which are torsion of order prime to $p$.
We let $A\{p\}$ denote the subgroup of elements of $A$ which are torsion of 
order some power of $p$. The tensor product $A \otimes_{\Z} B$ will be written as
$A \otimes B$. We let $\Ab$ denote the category of abelian groups and
$\Tab$ denote the category of topological abelian groups with continuous homomorphisms.
For $A, B \in \Tab$, we shall let $\Hom_\cf(A,B) = (\Hom_{\Tab}(A,B))_\tor$.
Unless a specific topology is mentioned, we shall assume all finite abelian groups to
be endowed with the discrete topology. In this paper, we shall use the general
notation $\{A_n\}$ for pro-abelian or ind-abelian group indexed by
$\N = \{1, 2, \ldots\}$. However, we shall also use specific notations
$``{{\underset{n}\varinjlim}}" A_n$ for an ind-abelian group and
$``{{\underset{n}\varprojlim}}" A_n$ for a pro-abelian
group if we need to make a distinction between them.

\section{Recollection of Zhao's duality}
\label{sec:Prelim}
In this section, we recall Brauer group, Milnor $K$-theory, Hodge-Witt sheaves and
the duality theorem of Zhao \cite{Zhao} for the
logarithmic Hodge-Witt cohomology. This will be the key ingredient in the proof of
\thmref{thm:Main-5}. We also recall the
Pontryagin duality for a special class of locally compact Hausdorff topological abelian
groups.

\subsection{Pontryagin duality for torsion-by-profinite groups}
\label{sec:PDuality}
Let $G$ be a locally compact Hausdorff topological abelian group.
We shall say that $G$ is `torsion-by-profinite' if there is an exact sequence
\begin{equation}\label{eqn:Prof-tor}
  0 \to G_\pf \to G \to G_\dt \to 0,
\end{equation}
where $G_\pf$  is an open and profinite subgroup of $G$ and $G_\dt$
is a torsion group with the quotient topology (necessarily discrete).
We shall call ~\eqref{eqn:Prof-tor} a torsion-by-profinite presentation of $G$.
Note that such a presentation of $G$ is not unique.

Recall that $G$ admits its Pontryagin dual $G^\star := \Hom_{\Tab}(G, \T)$,
  where $\T$ is the circle group such that the evaluation map $G \to (G^*)^*$ is
  an isomorphism of locally compact topological abelian groups
  (e.g., see \cite[\S~2.9]{Pro-fin} or \cite[Thm.~4.32]{Folland})
  if $G^\star$ is endowed with the compact-open topology.
One knows that $G^\star \cong \Hom_{\Tab}(G, {\Q}/{\Z})$ if $G$ is torsion-by-profinite
  group (e.g., see \cite[Lem.~2.9.2]{Pro-fin}), where
  ${\Q}/{\Z} = \T_{\tor}$ is endowed with the discrete topology.
 Furthermore, one has an isomorphism of topological
  abelian groups $\Hom_{\Tab}(G^\star, {\Q}/{\Z}) \xrightarrow{\cong} 
  \Hom_{\Tab}(G^\star, \T)$. In other words, letting $\pfd$ denote the 
  category of torsion-by-profinite topological abelian groups with continuous
  homomorphisms, we have the following.

 \begin{lem}\label{lem:Topological-0}
   The category $\pfd$ is closed under taking Pontryagin dual such that
   $G \in \pfd$ is profinite if and only if $G^\star$ is a discrete torsion group.
   Moreover, the canonical map $\Hom_{\Tab}(G, {\Q}/{\Z}) \to G^\star$ is an
   isomorphism of topological abelian groups for every $G \in \pfd$.
   
  \end{lem}

In this paper, we shall always consider ${\Z}$ and ${\Q}/{\Z}$ as topological
  abelian groups with discrete topologies. For any topological
  abelian group $G$, we shall denote $\Hom_{\Tab}(G, {\Q}/{\Z})$ by $G^\star$.
 For $G \in \Ab$, we shall let $G^\vee = \Hom_\Ab(G, {\Q}/{\Z})$.

\begin{lem}\label{lem:Closed-subgroup}
  If $G \in \pfd$ and $G' \subseteq G$ is a closed subgroup, then $G' \in \pfd$.
\end{lem}
\begin{proof}
  This is elementary and we skip the proof.
  \end{proof}

\begin{lem}\label{lem:Quotient-group}
  Let $\beta \colon G \surj G''$ be a continuous surjective homomorphism in $\pfd$ such
  that $G''_{\pf} \subseteq \beta(G_\pf)$ with respect to some torsion-by-profinite
  presentations of $G$ and $G''$ as in ~\eqref{eqn:Prof-tor}. Then
  the induced map $\beta' \colon {G}/{\Ker(\beta)} \to G''$ is an isomorphism of
  topological abelian groups if $ {G}/{\Ker(\beta)}$ is endowed with the quotient
  topology.
\end{lem}
\begin{proof}
  If $G''$ is profinite, the claim is an easy application of the fact that a continuous
  bijective homomorphism from a compact Hausdorff topological abelian group to a
  Hausdorff topological abelian group is a topological isomorphism. For the general
  case, we let
  $H = \beta^{-1}(G''_\pf)$. Then $H \subset G$ is an open subgroup and $G/H$ is discrete.
  Furthermore, there is a surjective continuous homomorphism
  $\beta \colon H \surj G''_\pf$ such that $(G_\pf \cap H) \surj G''_\pf$ by our
  assumption. \lemref{lem:Closed-subgroup} implies that $\beta \colon H \to
  G''_\pf$ is a continuous surjective homomorphism from a torsion-by-profinite group to a
  profinite group. In particular, this is a topological quotient. Since
  $G''_\pf$ is open in $G''$, one deduces that $\beta$ is an open map on $G$.
\end{proof}

\begin{lem}\label{lem:Dual-ex}
  Let
  \[
    0 \to G' \xrightarrow{\alpha} G \xrightarrow{\beta} G'' \to 0
\]
be a short exact sequence in $\pfd$. Assume that with respect to some
torsion-by-profinite presentations of $G$ and $G''$ as in ~\eqref{eqn:Prof-tor},
one has ${G''}_{\pf} \subseteq \beta(G_\pf)$. Then the sequence
\[
  0 \to (G'')^\star \xrightarrow{\beta^\star} G^\star \xrightarrow{\alpha^\star} (G')^\star
  \to 0
\]
is exact.
\end{lem}
\begin{proof}
The sequence
  \[
    0 \to (G'')^\star \xrightarrow{\beta^\star} G^\star \xrightarrow{\alpha^\star}
    (G')^\star
  \]
  is exact by \lemref{lem:Quotient-group} and $\alpha^\star$ is surjective by
  \lemref{lem:Topological-0} and \cite[Cor.~4.42]{Folland}.
  \end{proof}

Recall from \S~\ref{sec:Intro} that for a topological abelian group $G$, the
profinite completion of $G$ (with respect to its topology) is the profinite topological
abelian group $G^\pf = {\varprojlim}_U {G}/U$ with the inverse limit topology, where
the limit is taken over all open subgroups of finite index in $G$.
The following result is elementary whose proof is left to the reader.

\begin{lem}\label{lem:Prof-dual}
  Let $G$ be a topological abelian group. Then the
  canonical map $(G^\pf)^\star \to G^\star$ induces an isomorphism
  $(G^\pf)^\star \xrightarrow{\cong} (G^\star)_\tor$.
\end{lem}

Suppose that $\{G_i\}_{i \in I}$ is a pro-object in $\Tab$ such that $I$ is cofiltered.
We let $G = {\varprojlim}_i G_i$ and endow it with the inverse limit
topology. Let $\pi_i \colon G \to G_i$ be the projection map.
If the topology of each $G_i$ is generated by open subgroups, then the same holds for
$G$ as well. More precisely, $G$ admits a fundamental system of neighborhoods of
the identity of the form $\pi^{-1}_i(U_i)$, where $U_i \subset G_i$ is an open subgroup.
This happens, for instance, when each $G_i$ is a Hausdorff, totally disconnected
and locally compact topological abelian group. If $k$ is a local field and
$X$ is a commutative group scheme over $k$, then $X(k)$ is endowed with
a topology of this kind (cf. \S~\ref{sec:Pic-loc}). We shall use the following
result in this paper.

\begin{lem}\label{lem:dual-surj}
  Let $\{G_i\}_{i \in I}$ be above such that the topology of each $G_i$ is generated by
  open subgroups. Then the canonical map ${\varinjlim}_i G^\star_i \to G^\star$ is
  surjective. This map is bijective if $\pi_i \colon G \to G_i$ is surjective for each
  $i \in I$.
\end{lem}
\begin{proof}
  Let $\chi \colon G \to {\Q}/{\Z}$ be a continuous character and let $H = \Ker(\chi)$.
  Then our hypothesis implies that there exists $i \in I$ and an open subgroup
  $U_i \subset G_i$ such that $V_i:= \pi^{-1}_i(U_i) \subset H$. It follows that
  $\chi$ factors through $\chi' \colon G/{V_i} \to {\Q}/{\Z}$.
  We thus get a diagram
  \begin{equation}\label{eqn:dual-surj-0}
    \xymatrix@C1pc{
      G_i \ar@{->>}[r]^-{\alpha_i} & {G_i}/{U_i} \ar@{.>}[dr] & \\
      G \ar@{->>}[r] \ar[u]^-{\pi_i} & G/{V_i} \ar@{^{(}->}[u] \ar[r]^-{\chi'} &
      {\Q}/{\Z},}
   \end{equation}
   whose left square is commutative. Since $G/{V_i}$ and ${G_i}/{U_i}$ are discrete,
   we get a continuous character $\phi' \colon {G_i}/{U_i} \to {\Q}/{\Z}$ such that
   the left triangle commutes. Letting $\phi = \phi' \circ \alpha_i$, we see that
   $\phi$ is continuous and $\phi \circ \pi_i = \chi$. The second part of the lemma
   is clear.
\end{proof}

\subsection{Milnor $K$-theory and logarithmic
  Hodge-Witt sheaves}\label{sec:MKHW}
For a commutative ring $A$, we let $K_*(A)$ denote the Quillen $K$-theory of
finitely generated projective $A$-modules. We let $\wh{K}^M_*(A)$ denote the graded
commutative ring defined as the quotient of the tensor algebra
$T_*(A^\times)$ by the two-sided ideal generated by the homogeneous elements
$a_1 \otimes a_2$ such that $a_1 + a_2 = 1$. We let ${K}^M_*(A)$ denote the improved
Milnor $K$-theory of $A$ {\`a} la Gabber-Kerz \cite{Kerz-JAG}.
There are natural multiplicative homomorphisms $\wh{K}^M_r(A) \xrightarrow{\alpha_A}
{K}^M_r(A) \xrightarrow{\beta_A} K_r(A)$ such that $\alpha_A$ is surjective (resp. an
isomorphism) if $A$ is a local ring (resp. a field). Furthermore, $\alpha_A$ is an
isomorphism if $A$ is local with infinite residue field and $\beta_A$ is an isomorphism
for $r \le 2$ if $A$ is local (see \cite[Prop.~10]{Kerz-JAG}).

For a scheme $X$, we let
$\sK^M_{r,X}$ (resp. $\wh{\sK}^M_{r,X}$, resp. $\sK_{r,X}$) denote the sheaf
(in Zariski, Nisnevich or {\'e}tale topology) on $X$ associated to the presheaf
$U \mapsto K^M_r(\sO(U))$ (resp. $\wh{K}^M_r(\sO(U))$, resp.
$K_r(\sO(U))$). For an ideal $I \subset A$ in a local ring, we let $\wh{K}^M_r(A,I) =
\Ker(\wh{K}^M_r(A) \surj \wh{K}^M_r(A/I))$. Given a closed immersion $\iota \colon
Y \inj X$, we let $\wh{\sK}^M_{r,(X,Y)} = \Ker(\wh{\sK}^M_{r,X} \surj
\iota_*(\wh{\sK}^M_{r, Y}))$. The group ${K}^M_r(A,I)$ and the sheaf ${\sK}^M_{r,(X,Y)}$
are defined in a similar manner.

Let $k$ be a field of exponential characteristic $p > 1$.
Recall from \cite{Illusie} that for a $k$-scheme $X$, $\{W_m\Omega^\bullet_X\}_{m \ge 1}$
denotes the pro-complex of de Rham-Witt (Nisnevich) sheaves on $X$.
This is a pro-complex of sheaves of differential
graded algebras with the structure map $R$ and the
differential $d$. Let $[-]_m \colon \sO_X \to W_m\sO_X$ be the multiplicative
Teichm{\"u}ller homomorphism. Recall that the pro-complex 
$\{W_m\Omega^\bullet_X\}_{m \ge 1}$ is equipped with the
Frobenius homomorphism of graded algebras $F \colon W_m\Omega^r_X \to 
W_{m-1}\Omega^r_X$ and the additive Verschiebung homomorphism 
$V \colon W_m\Omega^r_X \to W_{m+1}\Omega^r_X$.
We let $Z W_m\Omega^r_X = \Ker(d \colon W_m\Omega^r_X \to
W_m\Omega^{r+1}_X)$ and $B W_m\Omega^r_X = 
{\rm Image}(d \colon W_m\Omega^{r-1}_X \to
W_m\Omega^{r}_X)$. We write $\sH^r(W_m\Omega^\bullet_X) = 
{Z W_m\Omega^r_X}/{B W_m\Omega^r_X}$.

Recall from \cite{Illusie} that $W_m\Omega^r_{X, \log}$ is the (Zariski, Nisnevich or
{\'e}tale) subsheaf of $W_m\Omega^r_X$ which is the image of the
map $\dlog \colon {\sK^M_{r,X}}/{p^m} \to W_m\Omega^r_X$, locally given by
$\dlog(\{a_1, \ldots , a_r\}) = \dlog[a_1]_m \wedge \cdots \wedge \dlog[a_r]_m$.
It is easily seen that this map exists (e.g., see \cite[Rem.~1.6]{Luders-Morrow}).
Moreover, it is an isomorphism in any of the above topologies if $X$ is regular
(e.g., see \cite[Lem.~2.3]{Gupta-Krishna-Duality}). Equivalently, there is a short exact
  sequence of (Zariski, Nisnevich or {\'e}tale) sheaves
\begin{equation}\label{eqn:GL*}
  0 \to \sK^M_{r,X} \xrightarrow{p^n} \sK^M_{r,X} \to
  W_n\Omega^r_{X, \log} \to 0.
\end{equation}
The multiplicative structure of ${\sK^M_{*,X}}/{p^m}$ and $W_m\Omega^\bullet_X$
together with the multplicativity of $\dlog$ induces a graded commutative ring structure
on $W_m\Omega^\bullet_{X, \log}$.

We end this subsection by recalling the Brauer group. For any Noetherian scheme $X$,
recall that $\Br(X)$ denotes the cohomological Brauer group $H^2(X, \sO^\times_X)$.
A result of Gabber (e.g., see \cite[Thm.~4.2.1]{CTS}) says that $\Br(X)$ coincides with
the Azumaya Brauer group $\Br'(X)$ if $X$ is quasi-projective over an affine scheme.
Since all schemes in this paper will satisfy this condition, we shall make no
distinction between the Azumaya and cohomological Brauer groups. Brauer group of
curves over local fields will be one the main objects of study in this paper.

\subsection{The set-up}\label{sec:setup}
The set-up for our duality theorem will be the following.
We fix a equicharacteristic complete discrete valuation ring (cdvr) $R$ with finite
residue field $\ff$. We let $\fm$ denote the maximal ideal of $R$ and $k$
denote the quotient field of $R$. It is well known that $R$ is canonically isomorphic
to the formal power series ring $\F_q[[\pi]]$, and $k = \F_q((\pi))$, where $q = p^n$ for
some prime number $p \ge 2$ and some positive integer $n \ge 1$. We shall fix this
isomorphism throughout our discussion. We let $S = \Spec(R)$ and let $\eta$ (resp.
$s$) denote the generic (resp. closed) point of $S$.

We let $\sX$ be a connected Noetherian regular scheme with a flat and projective
morphism $\F \colon \sX \to S$ of relative dimension $d \le 1$.
We let $X$ denote the generic
fiber and $\sX_s$ the (scheme-theoretic) closed fiber of $\F$. We shall assume that
$Y := (\sX_s)_\red$ is a simple normal crossing divisor on $\sX$. A morphism $\F$
satisfying these properties will be called a semi-stable model for the $k$-scheme $X$.
We let $u \colon X \inj \sX$ and $\iota \colon Y \inj \sX$ be the inclusions.
We let $f \colon X \to \Spec(k)$ and $g \colon \sX_s \to \Spec(\ff)$ denote the
structure maps. We shall assume that $X$ is geometrically connected and
generically smooth over $k$. Recall that the generic smoothness of a morphism of schemes
$f \colon Z \to W$ means that $f$ is smooth in a neighborhood of each generic point of
$Z$. All results in the remainder of \S~\ref{sec:Prelim} (and also in the next two
sections) will be proven under this set-up. We let $K$ denote the function field of
$\sX$.

The following elementary observation will be useful in this paper. We skip its
proof.

\begin{lem}\label{lem:Geom-int}
 $\sX_s$ (equivalently, $Y$) is geometrically connected over $\ff$.
\end{lem}

\subsection{Zhao's duality theorem}\label{sec:ZD}
We now recall the $d =1$ case of the duality theorem of Zhao
(see \cite[Cor.~1.4.10, Thm.~3.1.1]{Zhao}) which will be
a key ingredient in the proof of \thmref{thm:Main-5}.
We fix an integer $n \ge 1$.

\begin{thm}\label{thm:Zhao*}
  \begin{enumerate}
  \item
  There is a canonical isomorphism of ${\Z}/{p^n}$-modules
  \[
   {\Tr}_{\sX} \colon H^{3}_Y(\sX, W_n\Omega^{2}_{\sX, \log}) \xrightarrow{\cong}
    {\Z}/{p^n}.\]
\item
  The graded commutative ring structure of $H^*(\sX, W_n\Omega^\bullet_{\sX, \log})$
  gives rise to a pairing
  \[
    H^i(\sX, W_n\Omega^j_{\sX, \log}) \times
    H^{3-i}_{Y}(\sX, W_n \Omega^{2-j}_{\sX, \log}) \to {\Z}/{p^n}
  \]
  such that the induced map
  \[
   {\Phi}^{ij}_\sX \colon H^i(\sX, W_n\Omega^j_{\sX, \log}) \to
    \Hom_{\Ab}(H^{3-i}_{Y}(\sX, W_n \Omega^{2-j}_{\sX, \log}), {\Z}/{p^n})
  \]
  is an isomorphism.
  \item
The map
  \[
    {\Psi}^{ij}_{\sX} \colon H^{i}_{Y}(\sX, W_n \Omega^{j}_{\sX, \log}) \to
 \Hom_{\Tab}(H^{3-i}(\sX, W_n \Omega^{2-j}_{\sX, \log}), {\Z}/{p^n})   
\]
is a topological isomorphism if $H^{i}_{Y}(\sX, W_n \Omega^{j}_{\sX, \log})$ is endowed
with the discrete topology and $H^i(\sX, W_n\Omega^j_{\sX, \log})$ is endowed with the
profinite topology by virtue of $\Phi^{ij}_{\sX}$ and the Pontryagin duality.
\end{enumerate}
\end{thm}
\begin{proof}
  This result was proven by Zhao in \cite{Zhao} except that he additionally assumed
  that $\sX_s$ is reduced and $X \in \Sm_k$. But these
  assumptions are unnecessary as we shall explain by outlining Zhao's proof.

Recall that $\iota \colon Y \inj \sX$ is the inclusion of the reduced special fiber,
  which is assumed to be a simple normal crossing divisor in $\sX$. 
  The proof of (1) goes as follows. Using a duality theorem of Sato for normal
  crossing projective schemes over $\ff$, the
  proof of (1) is reduced to showing a purity isomorphism (i.e., Corollary~1.4.9 of op.
  cit.). The proof of this corollary has the following ingredients.
  The first is Proposition~1.2.4 of op.
  cit. for $Y$, which is a result of Sato that holds for any normal crossing scheme. 
  The second is Corollary~1.3.14 of op. cit., which is a result of Shiho that holds
  for any closed immersion of regular $\F_p$-schemes. The third is Proposition~1.4.6
  of op. cit., which is a result of Moser that holds for any finite type $\ff$-scheme.
  The final step is Theorem~1.4.4 of op. cit., whose proof also requires
  only the above three ingredients.

To prove (2), we let $\iota' \colon \sX_s \inj \sX$ and $\iota'' \colon Y \inj \sX_s$
  be the inclusions. We then have a cup product pairing (see \S~3.1 of op. cit.)
  \[
    H^i(Y, \iota^*(W_n\Omega^j_{\sX, \log})) \times H^{3-i}(Y, R\iota^{!}
    (W_n\Omega^{2-j}_{\sX, \log})) \to H^{3}(Y, R\iota^{!}
    (W_n\Omega^{2}_{\sX, \log})).
  \]
On the other hand, there are canonical maps
  \begin{equation}\label{eqn:PBC}
    \begin{array}{lll}
 H^i(\sX, W_n\Omega^j_{\sX, \log}) & \xrightarrow{\iota'^*} &
      H^i(\sX_s, \iota'^*(W_n\Omega^j_{\sX, \log})) \\
      & \xrightarrow{\iota''^*} &
    H^i(\sX_s, \iota''_* \circ \iota''^* \circ \iota'^*(W_n\Omega^j_{\sX, \log}))
      \\
      & \cong & H^i(\sX, \iota_* \circ \iota^*(W_n\Omega^j_{\sX, \log})) \\
& \cong & H^i(Y, \iota^*(W_n\Omega^j_{\sX, \log})).
          \end{array}
        \end{equation}

The first arrow is an isomorphism by the proper base change theorem for the
{\'e}tale cohomology of torsion sheaves (e.g., see \cite[Cor.~VI.2.7]{Milne-etale}).
The second arrow is the canonical counit of adjunction map which is
an isomorphism since the closed immersion $\iota''$ has empty complement and
$\iota''_*$ is exact. Since $H^{i}(Y, R\iota^{!} (W_n\Omega^{j}_{\sX, \log}))
\cong H^{i}_{Y}(\sX, W_n\Omega^{j}_{\sX, \log})$, the above pairing can be written as
\begin{equation}\label{eqn:Zhao*-1}
  H^i(\sX, W_n\Omega^j_{\sX, \log}) \times 
H^{3-i}_{Y}(\sX, W_n\Omega^{2-j}_{\sX, \log})
  \to H^{3}_{Y}(\sX, W_n\Omega^{2}_{\sX, \log}) 
\xrightarrow{{\Tr}_{\sX}} {\Z}/{p^n}.
\end{equation}

The proof of the perfectness of ~\eqref{eqn:Zhao*-1} 
is easily reduced to the case when $n = 1$ using the exact sequence
  (see \cite[Prop.~2.8, 2.12]{Shiho})
  \begin{equation}\label{eqn:Zhao*-0}
    0 \to W_{n-m}\Omega^j_{\sX, \log} \xrightarrow{\un{p}^m} W_{n}\Omega^j_{\sX, \log}
    \xrightarrow{R^{n-m}} W_{m}\Omega^j_{\sX, \log} \to 0
  \end{equation}
  and compatibility of ~\eqref{eqn:Zhao*-1} with respect to maps
  $H^i(\sX, W_{n}\Omega^j_{\sX, \log}) \xrightarrow{\un{p}}
  H^i(\sX, W_{n+1}\Omega^j_{\sX, \log})$ and $H^{3-i}_{Y}(\sX, W_{n+1}\Omega^{2-j}_{\sX, \log})
  \xrightarrow{R} H^{3-i}_{Y}(\sX, W_n\Omega^{2-j}_{\sX, \log})$. Here,
$\un{p}^m$ is induced by the multiplication map ${\sK^M_{j,\sX}}/{p^{n-m}}
  \xrightarrow{p^m} {\sK^M_{j,\sX}}/{p^{n}}$ via the isomorphism $\dlog$.

The proof of the case $n=1$ is a direct consequence of
  Cor.~2.5.2 (which is the coherent duality) and Prop.~3.1.2 (which proves the
  compatibility between the trace maps of Sato and that of the coherent duality)
  of op. cit..
  The latter result is a statement only about $S$. The proof of the former result
  does not depend on the special fiber at all, but it requires us to know that
  the locally free sheaf $\Omega^1_\sX$ has rank two. However, this can be checked
  by restricting to the open subscheme of $\sX$ where $X$ is smooth over $k$.
  The generic smoothness of $X$ over $k$ suffices for this purpose.
  \end{proof}

\section{The topology of logarithmic Hodge-Witt cohomology}
\label{sec:Duality}
As we mentioned in \S~\ref{sec:Intro}, we need to endow the logarithmic Hodge-Witt
cohomology of complete curves over local fields with suitable topology in order to  
extend Milne's duality for logarithmic Hodge-Witt
cohomology over finite fields to local fields. This is the goal of the present
section. We shall continue to work under the set-up of \S~\ref{sec:setup}.

\subsection{Some finiteness results}\label{sec:Vanishing}
In this subsection, we prove some technical lemmas regarding certain
logarithmic Hodge-Witt cohomology groups as preparation for proving \thmref{thm:Main-0}.
We let $n \ge 1$ be any integer. We begin with the following vanishing statement.
\begin{lem}\label{lem:Higher-diff}
 We have $W_n\Omega^j_{\sX, \log} = 0$ for $n \ge 1$ and $j \ge 3$.
  Furthermore, $H^i(\sX, W_n\Omega^j_{\sX, \log}) = 0$ for $i \ge 3$ and
  $j \ge 0$. 
\end{lem}
\begin{proof}        
To prove the first claim, we can assume, using ~\eqref{eqn:Zhao*-0}, that $n = 1$.
By the exact sequence (where $\ov{\id}$ is the canonical quotient map and $\ov{F}$ is
induced by $F$, see \cite[\S~5.2]{Gupta-Krishna-Duality})
\begin{equation}\label{eqn:dlog-2-*}
0 \to W_{n}\Omega^j_{\sX, \log} \to
W_{n}\Omega^j_\sX \xrightarrow{\ov{\id} - \ov{F}}  
\frac{W_{n}\Omega^j_\sX}{dV^{n-1}\Omega^{j-1}_\sX} \to 0,
\end{equation}
it suffices to show that $\Omega^j_{\sX} = 0$ for $j \ge 3$.
Since $\sX$ is connected and $\Omega^j_\sX$ is locally free
(e.g., see \cite[Prop.~2.1.2]{Zhao}), it suffices to show that $\Omega^j_\sX$ is
generically zero. We can thus replace $\sX$ by the largest affine open subscheme 
$X_\sm \subset X$ such that $X_\sm$ is smooth over $k$. Since $X_\sm \neq
\emptyset$ by our assumption,  the claim is therefore reduced to showing that
$\Omega^j_U = 0$ for $j \ge 3$ if $U$ is any smooth curve over $k$.

Now, we look at the exact sequence (e.g., see \cite[Thm.~25.1]{Matsumura})
\begin{equation}\label{eqn:Higher-diff-0}
    0 \to \Omega^1_k \otimes_k \sO_U \xrightarrow{d} \Omega^1_U \to \Omega^1_{U/k} \to 0.
  \end{equation}
 Taking the higher exterior powers, we get $\Omega^j_{U} \cong \Omega^{j-1}_k \otimes_k
\Omega^1_{U/k}$ for $j \ge 2$. Hence, it suffices to show that $\Omega^j_{k} = 0$ for
$j \ge 2$. But this is well known because $R = \F_q[[t]]$ and a direct computation shows
that $R \xrightarrow{\cong} \Omega^1_R$, which implies that $\Omega^j_{R} = 0$ for
$j \ge 2$. Note that this also shows that $\Omega^2_X \cong \Omega^1_{X/k}$.

To prove the second claim, it suffices to show using
 ~\eqref{eqn:PBC} that $H^i(Y, \iota^*(W_n\Omega^j_{\sX, \log})) = 0$ for
 $i \ge 3$. But this is clear since $cd_p(Y) \le 2$
 (e.g., see \cite[Chap.~VI, Rem.~1.5(b)]{Milne-etale}).
\end{proof}

\begin{lem}\label{lem:Trace}
  The boundary map of the localization sequence
  \[
    \partial \colon H^2(X, W_n\Omega^2_{X, \log}) \to
    H^3_{Y}(\sX, W_n\Omega^2_{\sX, \log})
  \]
  is an isomorphism.
\end{lem}
\begin{proof}
The localization sequence in question is the exact sequence
  \begin{equation}\label{eqn:Trace-0}
    H^2(\sX, W_n\Omega^2_{\sX, \log}) \to  H^2(X, W_n\Omega^2_{X, \log}) \to
    H^3_{Y}(\sX, W_n\Omega^2_{\sX, \log}) \to H^3(\sX, W_n\Omega^2_{\sX, \log}).
  \end{equation}
  It suffices therefore to show that $H^i(\sX, W_n\Omega^2_{\sX, \log}) = 0$ for
  $i \ge 2$. By ~\eqref{eqn:PBC}, we have
  $H^i(\sX, W_n\Omega^2_{\sX, \log}) \cong H^i(Y, \iota^*(W_n\Omega^2_{\sX, \log}))$.
  But the latter group is zero for $i \ge 3$ since $cd_p(Y) = 2$, as we observed before. 

To prove the vanishing of $H^2(\sX, W_n\Omega^2_{\sX, \log})$, it suffices to show,
using \thmref{thm:Zhao*}, that $H^1_{Y}(\sX, {\Z}/{p^n}) = 0$.
To that end, we look at the exact sequence
  \begin{equation}\label{eqn:Trace-1}
    0 \to H^0(\sX, {\Z}/{p^n}) \xrightarrow{u^*} H^0(X, {\Z}/{p^n}) \to
    H^1_{Y}(\sX, {\Z}/{p^n}) \to H^1(\sX, {\Z}/{p^n}) \xrightarrow{u^*}
    H^1(X, {\Z}/{p^n}).
    \end{equation}
Since $\sX$ is integral and normal, the map $u^*$ on the left is an isomorphism
and the map $u^*$ on the right is injective. It follows that the middle group is zero.
\end{proof}

\begin{lem}\label{lem:H1-fin}   
  The group $H^i_{Y}(\sX, W_n\Omega^j_{\sX, \log})$ is finite unless
  $i \in \{2,3\}$ and $j \in \{0,1\}$.
\end{lem}
  \begin{proof}
In view of \lemref{lem:Higher-diff}, we can assume $j \le 2$. We have seen in the
    proof of \lemref{lem:Trace} that $H^i_{Y}(\sX, W_n\Omega^0_{\sX, \log}) = 0$ for
    $i \le 1$. We have $H^0_{Y}(\sX, W_n\Omega^j_{\sX, \log}) = 0$ for every $j \ge 0$
by the Gersten injectivity $W_n\Omega^j_{\sX, \log} \inj u_*(W_n\Omega^j_{X, \log})$
(see \cite[Thm.~4.1]{Shiho}). The finitude claim for
$H^i_{Y}(\sX, W_n\Omega^2_{\sX, \log})$ follows from \thmref{thm:Zhao*} since
$H^i(\sX, {\Z}/{p^n}) \cong H^i(Y, {\Z}/{p^n})$ (see ~\eqref{eqn:PBC})
and the latter group is finite by \cite[Thm.~1.2.2]{Sato}.
We are left with showing the finitude of $H^1_{Y}(\sX, W_n\Omega^1_{\sX, \log})$.

We look at the exact (localization) sequence
\begin{equation}\label{eqn:H1-fin-0}
  H^0(\sX, W_n\Omega^1_{\sX, \log}) \xrightarrow{u^*_0}
  H^0(X, W_n\Omega^1_{X, \log}) \to H^1_{Y}(\sX, W_n\Omega^1_{\sX, \log}) \hspace*{3cm}
\end{equation}
\[
  \hspace*{7cm} \to
  H^1(\sX, W_n\Omega^1_{\sX, \log}) \xrightarrow{u^*_1}
    H^1(X, W_n\Omega^1_{X, \log}).
 \]
It suffices to show that $\coker(u^*_0)$ and $\Ker(u^*_1)$ are finite.
To prove the finiteness of $\coker(u^*_0)$, note that ~\eqref{eqn:GL*} yields
the commutative diagram of exact sequences
\begin{equation}\label{eqn:H1-fin-1}
  \xymatrix@C.8pc{
  0 \ar[r] & {H^0(\sX, \sO^\times_{\sX})}/{p^n} \ar[r] \ar[d]_-{u^*} &
  H^0(\sX, W_n\Omega^1_{\sX, \log}) \ar[r] \ar[d]^-{u^*_0} & _{p^n} \Pic(\sX)
  \ar[r] \ar[d]^-{u^*} & 0 \\
 0 \ar[r] & {H^0(X, \sO^\times_{X})}/{p^n} \ar[r] &
 H^0(X, W_n\Omega^1_{X, \log}) \ar[r] & _{p^n} \Pic(X) \ar[r] & 0.}
\end{equation}

We next note that there are exact sequences
\begin{equation}\label{eqn:H1-fin-2}
  0 \to H^0(\sX, \sO^\times_{\sX}) \xrightarrow{u^*} H^0(X, \sO^\times_{X})
\xrightarrow{v_k} \Z \to 0
\end{equation}
and
\begin{equation}\label{eqn:H1-fin-3}
  0 \to \Z \xrightarrow{\iota'_*} {\underset{Z \in \irr(Y)}\bigoplus} \Z
  \xrightarrow{\cyc} \Pic(\sX) \xrightarrow{u^*} \Pic(X) \to 0,
\end{equation}
where the first sequence follows from the geometric integrality condition
(cf. \lemref{lem:Geom-int}). The second sequence is well known (e.g., see
\cite{Shafarevich}) if we let $\iota'_*(1) = [\sX_s]$, the class of the Weil divisor
associated to the special fiber $\sX_s$ and $\cyc$ denotes the cycle class map.
It easily follows from these exact sequences that the kernels and cokernels
of the left and right vertical arrows in ~\eqref{eqn:H1-fin-1} are finite. Hence, so are
the kernel and cokernel of the middle vertical arrow (the assertion about the kernel
will be used only in the next lemma).

To prove the finitude of $\Ker(u^*_1)$, we look at the commutative diagram
with exact rows (following ~\eqref{eqn:GL*}) 
\begin{equation}\label{eqn:H1-fin-4}
  \xymatrix@C.8pc{
0 \ar[r] & {\Pic(\sX)}/{p^n} \ar[r] \ar@{->>}[d]_-{u^*} & H^1(\sX, W_n\Omega^1_{\sX, \log}) 
\ar[r] \ar[d]^-{u^*_1} & _{p^n} \Br(\sX) \ar[r] \ar[d]^-{u^*} & 0 \\
0 \ar[r] & {\Pic(X)}/{p^n} \ar[r] & H^1(X, W_n\Omega^1_{X, \log}) 
\ar[r] & _{p^n} \Br(X) \ar[r] & 0.}
\end{equation}
It follows from \cite[Thm.~5.6.1(v), 10.3.1(ii)]{CTS} that
\begin{equation}\label{eqn:H1-fin-5}
  \Br(\sX) = \Br(\sX_s) = \Br(Y) = 0.
  \end{equation}
 On the other hand, it follows from ~\eqref{eqn:H1-fin-3} that the kernel of the
  left vertical arrow in ~\eqref{eqn:H1-fin-4} is finite. We conclude that
  $\Ker(u^*_1)$ is finite. We have thus shown that $H^1_{Y}(\sX, W_n\Omega^1_{\sX, \log})$
  is finite. This concludes the proof.
\end{proof}

\begin{remk}\label{remk:H1-fin-8}
  It can be shown that each of the groups $H^i_{Y}(\sX, W_n\Omega^j_{\sX, \log})$ is
  infinite in general if $i \in \{2,3\}$ and $j \in \{0,1\}$.
\end{remk}

\begin{lem}\label{lem:Fin-ker}
For all $i,j \ge 0$, the kernel of $u^* \colon
    H^i(\sX, W_n\Omega^j_{\sX, \log}) \to H^i(X, W_n\Omega^j_{X, \log})$ is finite.
  \end{lem}
  \begin{proof}
    Combine \thmref{thm:Zhao*} and \lemref{lem:H1-fin}.
    \end{proof}
    
    \subsection{Logarithmic Hodge-Witt cohomology as topological
      groups}\label{sec:Top-X}
    We shall now endow the logarithmic Hodge-Witt cohomology of $X$ with
    torsion-by-profinite topology. 
For $i, j \ge 0$, we consider the localization sequence
\begin{equation}\label{eqn:Main-ex}
  \cdots \to H^i_{Y}(\sX, W_n\Omega^j_{\sX, \log}) \xrightarrow{\iota_*}
  H^i(\sX, W_n\Omega^j_{\sX, \log}) \xrightarrow{u^*}
  H^i(X, W_n\Omega^j_{X, \log}) \hspace*{1cm}
\end{equation}
\[
  \hspace*{9cm}
  \xrightarrow{\partial} H^{i+1}_{Y}(\sX, W_n\Omega^j_{\sX, \log}) \to \cdots .
\]
By \lemref{lem:Fin-ker}, it breaks into exact sequences
\begin{equation}\label{eqn:Main-ex-0}
H^i_{Y}(\sX, W_n\Omega^j_{\sX, \log}) \xrightarrow{\iota_*}
H^i(\sX, W_n\Omega^j_{\sX, \log}) \xrightarrow{u^*} F_{ij} \to 0;
\end{equation}
\begin{equation}\label{eqn:Main-ex-1}
  0 \to F_{ij} \to H^i(X, W_n\Omega^j_{X, \log}) \xrightarrow{\partial}
  H^{i+1}_{Y}(\sX, W_n\Omega^j_{\sX, \log})
  \end{equation}
  such that $\Ker(u^*)$ is finite. 

Recall from \thmref{thm:Zhao*} that each $H^i_{Y}(\sX, W_n\Omega^j_{\sX, \log})$ is
a discrete torsion group of exponent $p^n$ and $H^i(\sX, W_n\Omega^j_{\sX, \log})$
is a profinite topological abelian group. It follows that ${\rm Image}(\iota_*)$ is
a finite closed subgroup of $H^i(\sX, W_n\Omega^j_{\sX, \log})$. In particular, its 
subspace topology induced from $H^i(\sX, W_n\Omega^j_{\sX, \log})$ as well as
its quotient topology induced from $H^i_{Y}(\sX, W_n\Omega^j_{\sX, \log})$ is discrete.

For $i, j \ge 0$, we endow $F_{ij}$ with the quotient topology via
~\eqref{eqn:Main-ex-0}. Then $F_{ij}$ becomes a quotient of the profinite group
$H^i(\sX, W_n\Omega^j_{\sX, \log})$ by a closed subgroup.
It follows that $F_{ij}$ is a profinite abelian group and ~\eqref{eqn:Main-ex-0}
is an exact sequence of topological abelian groups.

We endow $H^i(X, W_n\Omega^j_{X, \log})$ with the unique topology for which
$F_{ij}$ is an open subgroup of $H^i(X, W_n\Omega^j_{X, \log})$ and
${H^i(X, W_n\Omega^j_{X, \log})}/{F_{ij}} = \Ker(\iota_*)$ is a discrete
quotient of $H^i(X, W_n\Omega^j_{X, \log})$.
This endows $H^i(X, W_n\Omega^j_{X, \log})$ with the structure of a topological
abelian group such that ~\eqref{eqn:Main-ex-1} is an exact sequence of topological
abelian groups. If we combine this with everything we have shown in
\S~\ref{sec:ZD} and \S~\ref{sec:Vanishing}, we get the following.

\begin{prop}\label{prop:Top-Coh-X}
  For each $i, j \ge 0$, the group $H^i(X, W_n\Omega^j_{X, \log})$ is
  equipped with the structure of a torsion-by-profinite topological abelian group
  satisfying the following.
  \begin{enumerate}
  \item
   ~\eqref{eqn:Main-ex} is an exact sequence in $\pfd$ for
   every $j \ge 0$.
   \item
    $H^i(X, W_n\Omega^j_{X, \log}) = 0$ unless $0 \le i, j \le 2$.
  \item
    $H^i(X, W_n\Omega^j_{X, \log})$ is profinite unless $i \in \{1,2\}$ and
  $j \in \{0,1\}$.
\item
  $H^2(X, W_n\Omega^0_{X, \log})$ is a discrete torsion group.
  \item
    $H^1(X, W_n\Omega^0_{X, \log})$ is an infinite discrete torsion group.
  \item
    $H^1(X, W_n\Omega^1_{X, \log})$ is neither a profinite nor a discrete torsion group.
  \end{enumerate}
  \end{prop}
  \begin{proof}
    We only need to explain (4) and (6). But the former follows from
    \cite[Thm.~1.2.2]{Sato} and the isomorphism $H^2(\sX, {\Z}/{p^n}) \cong
    H^2(Y, {\Z}/{p^n})$ while the latter follows from the perfectness of
    the Brauer-Manin pairing for $X$ (e.g., see \cite[\S~9]{Saito-Invent}).
\end{proof}

Using \lemref{lem:Topological-0}, \propref{prop:Top-Coh-X} and considering the
  Pontryagin duals associated to
  ~\eqref{eqn:Main-ex}, we get a chain complex in $\pfd$:
  \begin{equation}\label{eqn:Exact-Dual-0}
  \cdots \to H^{i+1}_{Y}(\sX, W_n\Omega^j_{\sX, \log})^\star \xrightarrow{\partial^\star}
  H^i(X, W_n\Omega^j_{X, \log})^\star \xrightarrow{(u^*)^\star}
  H^i(\sX, W_n\Omega^j_{\sX, \log})^\star \hspace*{1cm}
\end{equation}
\[
  \hspace*{9cm}
  \xrightarrow{(\iota_*)^\star} H^{i}_{Y}(\sX, W_n\Omega^j_{\sX, \log})^\star \to
  \cdots .
\]

\begin{lem}\label{lem:Exact-Dual}
  The sequence ~\eqref{eqn:Exact-Dual-0} is exact.
\end{lem}
\begin{proof}
 We look at the exact sequences
  \begin{equation}\label{eqn:Exact-Dual-1}
0 \to \Ker(\iota_*) \to H^{i}_{Y}(\sX, W_n\Omega^j_{\sX, \log})
\to \Ker(u^*) \to 0;
\end{equation}
\begin{equation}\label{eqn:Exact-Dual-2}
  0 \to \Ker(u^*) \to  H^{i}(\sX, W_n\Omega^j_{\sX, \log}) \to F_{ij} \to 0;
\end{equation}
\begin{equation}\label{eqn:Exact-Dual-3}
0 \to F_{ij} \to H^{i}(X, W_n\Omega^j_{X, \log}) \to \Ker(\iota_*) \to 0.
\end{equation}

It follows from \propref{prop:Top-Coh-X} and the definition of the topologies of the
various groups in ~\eqref{eqn:Main-ex} that the above three are exact
sequences in $\pfd$ and they satisfy the hypothesis of \lemref{lem:Dual-ex}.
We can therefore apply this lemma to get short exact sequences
\begin{equation}\label{eqn:Exact-Dual-4}
0 \to  \Ker(u^*)^\star \to H^{i}_{Y}(\sX, W_n\Omega^j_{\sX, \log})^\star \to  
\Ker(\iota_*)^\star \to 0;
\end{equation}
\begin{equation}\label{eqn:Exact-Dual-5}
  0 \to \Ker(\iota_*)^\star \to H^{i}(X, W_n\Omega^j_{X, \log})^\star \to
  (F_{ij})^\star \to 0;
\end{equation}
\begin{equation}\label{eqn:Exact-Dual-6}
  0 \to (F_{ij})^\star \to H^{i}(\sX, W_n\Omega^j_{\sX, \log})^\star \to 
  \Ker(u^*)^\star \to 0.
  \end{equation}
The desired result follows by piecing these exact sequences together.
\end{proof}

\section{Duality for logarithmic Hodge-Witt
  cohomology}\label{sec:HWD}
We continue to work under the set-up of \S~\ref{sec:setup}.
In this section, we shall prove the duality theorem for the logarithmic Hodge-Witt
cohomology of the proper curve $X$ over $k$. We begin by constructing the trace map.
For a scheme $Z$, we let $\sD_\et(Z)$ denote the derived category of {\'e}tale
sheaves of abelian groups on $Z$. We let $\sD_\et(Z, {\Z}/m)$ denote the
derived category of {\'e}tale sheaves of ${\Z}/m$-modules on $Z$ with
$m \in \sO^\times(Z)$.
For a local ring $A$, we shall let $A^h$ (resp.
$A^{sh}$, resp. $\wh{A}$) denote the Henselization (resp. strict Henselization,
resp. completion) of $A$ with respect to its maximal ideal. For a closed point
$x \in X$, we shall let $K_x$ (resp. $K^{sh}_x$, resp. $\wh{K}_x$) denote the
quotient field of the Henselization $\sO^h_{X,x}$ (resp. strict Henselization
$\sO^{sh}_{X,x}$, resp. completion $\wh{\sO_{X,x}}$) of $\sO_{X,x}$.
We fix an integer $n \ge 1$.

\subsection{The trace maps}\label{sec:TRACE}
We let $j \colon \Spec(K) \inj X$ be the inclusion of the generic point of $X$ and let
$\iota_x \colon \Spec(k(x)) \inj X$ be the inclusion of a closed point.
By the Gersten resolution of the logarithmic Hodge-Witt sheaves, one has an exact
sequence of {\'e}tale sheaves
\begin{equation}\label{eqn:Gersten-0}
  0 \to W_n\Omega^2_{X, \log} \to j_*(W_n\Omega^2_{K,\log}) \xrightarrow{\res}
  {\underset{x \in X_{(0)}}
    \bigoplus} (\iota_x)_*(W_n\Omega^1_{k(x),\log}) \to 0,
\end{equation}
where $\res$ is the sum of residue maps in Kato's complex
(see \cite[Thm.~5.2]{Shiho}).
It follows that there is an exact triangle
\begin{equation}\label{eqn:Gersten-1}
  0 \to \left({\underset{x \in X_{(0)}} \bigoplus}
    (\iota_x)_*(W_n\Omega^1_{k(x),\log})\right)
  [-1] \to W_n\Omega^2_{X, \log} \to  j_*(W_n\Omega^2_{K,\log}) \to 0
\end{equation}
in $\sD_\et(X)$. We let  $\nu^2_{n,X}$ denote the complex
$\left(j_*(W_n\Omega^2_{K,\log}) \xrightarrow{\res} {\underset{x \in X_{(0)}} \bigoplus}
  (\iota_x)_*(W_n\Omega^1_{k(x),\log})\right)$.

For every $x \in X_{(0)}$, we have the canonical trace map
$\epsilon_x \colon (f_x)_*(W_n\Omega^1_{k(x),\log}) \to W_n\Omega^1_{k,\log}$, induced by the
norm $N_x \colon (f_x)_*(\sO^\times_{k(x)}) \to \sO^\times_k$ via the dlog map,
where $f_x \colon \Spec(k(x)) \to \Spec(k)$ is the structure map.
We let $\epsilon_X := \sum_x \epsilon_x$. By the reciprocity law for Milnor
$K$-theory (e.g., see  \cite[Lem.~4]{Kato-83}), it is clear that the composite map
\[
 g_*(W_n\Omega^2_{K,\log}) \xrightarrow{\res}   
{\underset{x \in X_{(0)}} \bigoplus}
(f_x)_*(W_n\Omega^1_{k(x),\log}) \xrightarrow{\epsilon_X} W_n\Omega^1_{k,\log}
\]
is zero if we let  $g = f \circ j$, where recall that $f \colon X \to \Spec(k)$ is the
structure map. In other words, we have a morphism of complexes 
\begin{equation}\label{eqn:Gersten-2}
  f_* \colon  f_*(\nu^2_{n,X}) \to  W_n\Omega^1_{k,\log}[-1]
  \end{equation}
  whose composition with $(f_x)_*(W_n\Omega^1_{k(x),\log})[-1] \to f_*(\nu^2_{n,X})$
  is $\epsilon_x[-1]$ for every $x \in X_{(0)}$.
We now prove some key lemmas.

\begin{lem}\label{lem:Gersten-3}
  One has $R^qj_*(W_n\Omega^2_{K, \log}) = 0$ for $q \ge 1$.
\end{lem}
\begin{proof}
The stalk of $R^qj_*(W_n\Omega^2_{K, \log})$ at a closed point $x \in X$ is
  $H^q(K^{sh}_x, W_n\Omega^2_{K^{sh}_x, \log})$. To show that this cohomology is zero,
  we can assume $q = 1$ because $cd_p(K^{sh}_x) \le 1$.
  By \cite[Lem.~21]{Kato-Invitation}, we can pass to the completion $\wh{K^{sh}_x}$
  of $K^{sh}_x$.
We let $A_x$ denote the completion of $\sO^{sh}_{X,x},  \ F_x = \wh{K^{sh}_x}$ and
  consider the exact sequence
\begin{equation}\label{eqn:Gersten-3-0}
  H^1(A_x, W_n\Omega^2_{A_x, \log}) \to
  H^1(F_x, W_n\Omega^2_{F_x, \log}) \to
  H^{2}_{\ov{x}}(A_x, W_n\Omega^2_{A_x, \log}),
\end{equation}
where $\ov{x}$ is the closed point of $\Spec(A_x)$. By the proper base change theorem,
we have $H^1(A_x, W_n\Omega^2_{A_x, \log}) 
\cong  H^1(k(\ov{x}), \iota^*(W_n\Omega^2_{A_x, \log}))$, where
$\iota \colon \Spec(k(\ov{x})) \inj \Spec(A_x)$ is the inclusion of the closed
point. But the latter group vanishes since $k(\ov{x})$ is
separably closed (e.g., see \cite[Thm.~VI.1.1]{Milne-etale}). It remains to show that
$H^{2}_{\ov{x}}(A_x, W_n\Omega^2_{A_x, \log}) = 0$. 

Since $k(\ov{x})$ is separably closed, we have the well known equality
$[k(\ov{x}): k(\ov{x})^p] = p$ (e.g., see \cite[\S~3]{Kato-Saito-Ann}).
Since $F_x \cong k(\ov{x})((T))$, it is an elementary exercise 
that $[F_x: (F_x)^p] = p^2$ (e.g., see the proof of \cite[Lem.~1.4.5]{Zhao}).
We therefore conclude from \cite[Thm.~3.2]{Shiho}
that there is a purity isomorphism $\iota_* \colon
H^{1}(k(\ov{x}), W_n\Omega^1_{k(\ov{x}), \log}) \xrightarrow{\cong}
H^{2}_{\ov{x}}(A_x, W_n\Omega^2_{A_x, \log})$. This proves our claim because
$H^{1}(k(\ov{x}), W_n\Omega^1_{k(\ov{x}), \log})$ is clearly zero.
To finish the proof of the lemma, we note that the stalk of $R^qj_*(W_n\Omega^2_{K, \log})$
at the generic point of $X$ is $H^q(K_s, W_n\Omega^2_{K_s, \log})$, which is 
zero if $q \ge 1$.
\end{proof}

\begin{lem}\label{lem:Gersten-4}
   The canonical map $f_*(\nu^2_{n,X}) \to Rf_*(\nu^2_{n,X})$ is an isomorphism in
   $\sD_\et(k)$.
  \end{lem}
  \begin{proof}
In view of \lemref{lem:Gersten-3},
    it suffices to show that $R^qf_*(j_*(W_n \Omega^2_{K, \log})) = 0$
and $R^qf_*((\iota_x)_*(W_n \Omega^1_{k(x), \log})) = 0$ for all $q \ge 1$ and
$x \in X_{(0)}$. Using the exactness of $(\iota_x)_*$ for $x \in X_{(0)}$ and
\lemref{lem:Gersten-3} for the generic point of $X$, it suffices to show that
$R^qg_*(W_n \Omega^2_{K, \log}) = 0$ and $R^q(f_x)_*(W_n \Omega^1_{k(x), \log}) = 0$
for all $q \ge 1$ and $x \in X_{(0)}$.
This latter claim for $x \in X_{(0)}$ is obvious since
$f_x$ is a finite morphism.

On the other hand, we have
$R^qg_*(W_n \Omega^2_{K, \log}) \cong H^q(K', W_n\Omega^2_{K', \log})$,
where $K'$ is the product of the function fields of the
connected components of $X_s := X \times_{\Spec(k)} \Spec(k_s)$. We are now done because
$H^q(K', W_n\Omega^2_{K', \log})$ is clearly zero if $q \ge 2$, and its
vanishing for $q = 1$ was shown by Kato-Saito \cite{Kato-Saito-Ann}
(see the proof of their Lemma~1(2) on p.~252). We remark that even though the
cited lemma of Kato-Saito assumes that $X$ is smooth over $k$, the proof of
the vanishing of $H^1(K', W_n\Omega^2_{K', \log})$ only requires $X$ to be generically
smooth over $k$. This concludes the proof. 
\end{proof}

\begin{lem}\label{lem:Gersten-5}
There exists a canonical morphism
  ${\tr}_X \colon  Rf_*(W_n\Omega^2_{X, \log}) \to
    W_n\Omega^1_{k, \log}[-1]$ in $\sD_\et(k)$ whose composition with the canonical
    morphisms
    \[
      (f_x)_*(W_n\Omega^1_{k(x), \log}[-1]) \to f_*(W_n\Omega^2_{X, \log}) \to
      Rf_*(W_n\Omega^2_{X, \log})
    \]
    is $\epsilon_x[-1]$ for every $x \in X_{(0)}$.
  \end{lem}
  \begin{proof}
    Combine the isomorphism $W_n\Omega^2_{X, \log} \xrightarrow{\cong} \nu^2_{n,X}$
    in $\sD_\et(X)$ with \lemref{lem:Gersten-4} and the morphism $f_*$ in
    ~\eqref{eqn:Gersten-2}.
\end{proof}

Recall that there is a canonical isomorphism $\inv_k \colon \Br(k) \xrightarrow{\cong}
{\Q}/{\Z}$. If ${k'}/k$ is a finite field extension of degree $m$ and
$v \colon \Spec(k') \to \Spec(k)$ is the projection, then the norm map
$N_{{k'}/k} \colon v_*(\sO^\times_{k'}) \to \sO^\times_k$ between the {\'e}tale sheaves
induces the pull-back and push-forward maps
$v^* \colon \Br(k) \to \Br(k')$ and $v_* \colon \Br(k') \to \Br(k)$
such that $\inv_{k'} \circ v^* = m (\inv_{k})$ on ${\Q}/{\Z}$ and
  $\inv_k \circ v_* = \inv_{k'}$.

Identifying $_{p^n} \Br(k')$ with $H^1(k', W_n\Omega^1_{k', \log})$,
we get a commutative diagram
\begin{equation}\label{eqn:Gersten-6}
  \xymatrix@C1pc{
    H^1(k', W_n\Omega^1_{k', \log}) \ar[d]_-{v_*} \ar[dr]^-{\inv_{k'}} & \\
    H^1(k, W_n\Omega^1_{k, \log}) \ar[r]_-{\inv_k} & {\Z}/{p^n}}
\end{equation}
in which the horizontal and the diagonal arrows are isomorphisms.
It follows that $v_*$ is an isomorphism such that $\inv_k \circ v_* =
\inv_{k'}$. We shall therefore identify $v_*$ with the identity map of
${\Z}/{p^n}$. Similarly, we shall identify $v_* \colon \Br(k') \to \Br(k)$ with
the identity map of ${\Q}/{\Z}$ throughout this paper.
In the sequel, we shall write $v_* = \epsilon_x$ if
$k' = k(x)$ for some closed point $x$ on a curve over $k$.

\begin{prop}\label{prop:Gersten-9}
There exists a unique homomorphism $\Tr_X \colon H^2(X, W_n\Omega^2_{X, \log}) \to
  {\Z}/{p^n}$ such that $\Tr_X \circ (\iota_x)_*$ is the identity map for every
  $x \in X_{(0)}$. Furthermore, $\Tr_X$ is bijective.
  \end{prop}
  \begin{proof}
By ~\eqref{eqn:Gersten-0} and \lemref{lem:Gersten-3}, we have an exact sequence
 \begin{equation}\label{eqn:Gersten-9-0}
  H^1(K, W_n\Omega^2_{K, \log}) \xrightarrow{\res}
      {\underset{x \in X_{(0)}}\bigoplus} H^1(k(x), W_n\Omega^1_{k(x), \log})
      \to H^2(X, W_n\Omega^2_{X, \log}) \to 0.
    \end{equation}
This implies that $H^2(X, W_n\Omega^2_{X, \log})$ is generated by
the images of $(\iota_x)_*(1)$ as $x$ varies in $X_{(0)}$. The uniqueness assertion
follows immediately from this. The existence follows directly from \lemref{lem:Gersten-5}
    and ~\eqref{eqn:Gersten-6}. The claim that $\Tr_X$ is bijective follows from
    the fact that its source and target are finite of equal
    cardinality by \thmref{thm:Zhao*} and \lemref{lem:Trace} while
    $\Tr_X \circ (\iota_x)_*$ is bijective.
  \end{proof}

\begin{remk}\label{remk:Gersten-9-1}
  The assertion that $H^2_\et(X, W_n\Omega^2_{X, \log})$ is isomorphic to ${\Z}/{p^n}$
  is not new (at least when $X$ is smooth over $k$)
  as it was already shown by Kato-Saito in \cite[Prop.~4]{Kato-Saito-Ann}.
  It also follows from \thmref{thm:Zhao*} and \lemref{lem:Trace}.
  However, the above proposition provides an explicit isomorphism which has the
  advantage that it commutes with the trace maps for all closed points of $X$.
This is unclear in the construction of \cite{Kato-Saito-Ann}. We shall
  need this explicit nature of $\Tr_X$ later in the paper in a critical way.
  Another advantage is that it does not require $X$ to be smooth
  everywhere. Although this flexibility is not important in this paper,
  it will be vital in the study of duality and class field theory for open smooth
  curves over $k$ which do not admit smooth compactifications.
  We were unable to prove if $\Tr_X$ coincides with the isomorphism of
  Kato-Saito, whose construction, we believe, is quite intricate.
\end{remk}

Given $x \in X_{(0)}$, we now look at the canonical maps of {\'e}tale sheaves
\[
  (\iota_x)_*(W_n\Omega^1_{k(x), \log})
  \to (\iota_x)_*R(\iota_x)^{!}(W_n\Omega^2_{X, \log}) \to W_n\Omega^2_{X, \log},
\]
on $X$, where $[t_x]$ is the Teichm{\"u}ller image of the chosen uniformizer
$t_x \in \sO_{X,x}$ and the first arrow is obtained by taking cup product with
$\dlog([t_x])$. This arrow is an isomorphism
in $\sD_\et(X)$ by \cite[Cor.~1.3.14]{Zhao}. The same is also true if we replace $X$ by
$\Spec(\sO^h_{X,x})$.

Furthermore, the diagram of the induced maps on cohomology 
\begin{equation}\label{eqn:Gersten-7}
  \xymatrix@C1pc{
    H^1(k(x), W_n\Omega^1_{k(x), \log}) \ar@/^2pc/[rr]^-{(\iota_x)_*}
    \ar[r]^-{\cup [t_x]}_-{\cong}
    \ar[d]_-{\cup [t_x]} \ar[dr] &
    H^2_{x}(X, W_n\Omega^2_{X, \log}) \ar[d]^-{\cong} \ar[r] &
    H^2(X, W_n\Omega^2_{X, \log}) \ar[d] \\
    H^1(K_x, W_n\Omega^2_{K_x, \log}) \ar[r]^-{\partial_x}   &
    H^2_{x}(\sO^h_{X,x}, W_n\Omega^2_{\sO^h_{X,x}, \log}) \ar[r] & 
    H^2(\sO^h_{X,x}, W_n\Omega^2_{\sO^h_{X,x}, \log})}
\end{equation}
is commutative, where the left vertical arrow $\cup [t_x]$ is the isomorphism of
\cite[\S~3.2, Lem.~3]{Kato-80}. It follows from the proper base change theorem that
the group on the bottom right corner is zero. Since the
other two groups in the bottom row are isomorphic to ${\Z}/{p^n}$
(e.g., see \cite[Prop.~3.1]{Kato-80}), it follows that $\partial_x$ is
bijective.

Using \propref{prop:Gersten-9} and ~\eqref{eqn:Gersten-7}, we get a commutative
diagram
\begin{equation}\label{eqn:Gersten-8}
  \xymatrix@C1pc{
    & H^1(k(x), W_n\Omega^1_{k(x), \log}) \ar[r]^-{(\iota_x)_*}
    \ar[dl]_-{h_x}^-{\cong} \ar[d]^-{\epsilon_x}_-{\cong} &  
    H^2(X, W_n\Omega^2_{X, \log}) \ar[dl]_-{\tr_X} \ar[d]^-{\Tr_X} \\
    H^1(K_x, W_n\Omega^2_{K_x, \log})  \ar@{.>}[r] &
    H^1(k, W_n\Omega^1_{k, \log}) \ar[r]^-{\inv_k} & {\Z}/{p^n},}
\end{equation}
where $h_x := \cup [t_x]$ is an isomorphism.

Letting $\Tr_x := \inv_k \circ \epsilon_x \circ (h_x)^{-1}$, we get the
following.

\begin{cor}\label{cor:Gersten-10}
 For every $x \in X_{(0)}$, there is a commutative diagram
  \begin{equation}\label{eqn:Gersten-10-0}
  \xymatrix@C1pc{
H^1(k(x), W_n\Omega^1_{k(x), \log}) \ar[r]^-{(\iota_x)_*}
    \ar[d]_-{h_x} & H^2(X, W_n\Omega^2_{X, \log}) \ar[d]^-{\Tr_X} \\  
    H^1(K_x, W_n\Omega^2_{K_x, \log}) \ar[r]^-{\Tr_x} & {\Z}/{p^n}}
\end{equation}
in which all arrows are isomorphisms.
\end{cor}

\subsection{Duality theorem for $H^*(X, W_n\Omega^\bullet_{X,\log})$}
\label{sec:Duality-X}
Recall that a pairing between locally compact Hausdorff topological abelian groups
$\lambda \colon A \times B \to {\Q}/{\Z}$ is called continuous if $\lambda$ is
continuous with respect to the product topology of $A \times B$.
Equivalently, either (and hence both) of the maps $A \to B^\vee$ and
$B \to A^\vee$ is continuous and factors through the
continuous dual (e.g., see \cite[Prop.~A.14]{Hatcher}).
$\lambda$ is called non degenerate on the left (resp. right) if the
induced map $A \to B^\vee$ (resp. $B \to A^\vee$) is injective.
One says that $\lambda$ is perfect (in particular, continuous) if it induces isomorphisms
$A \xrightarrow{\cong} B^\star$ and $B \xrightarrow{\cong} A^\star$ of topological
abelian groups.

We now return to the set-up of \S~\ref{sec:setup}.
Recall from \propref{prop:Gersten-9} that there is a canonical isomorphism
  of ${\Z}/{p^n}$-modules
  \begin{equation}\label{eqn:KS-iso}
    \Tr_X \colon H^{2}(X, W_n\Omega^2_{X, \log}) \xrightarrow{\cong}
    {\Z}/{p^n}.
    \end{equation}
    Using \lemref{lem:Trace}, we get a  unique isomorphism 
$\Tr_{\sX} \colon H^{3}_{Y}(\sX, W_n\Omega^2_{\sX, \log}) \xrightarrow{\cong}  {\Z}/{p^n}$
such that $\Tr_{\sX} \circ \partial = \Tr_X$ (cf. \thmref{thm:Zhao*}). 
The cup product pairing
\begin{equation}\label{eqn:KS-iso-0}
H^i(\sX, W_n\Omega^j_{\sX, \log}) \times H^{3-i}_{Y}(\sX, W_n\Omega^{2-j}_{\sX, \log}) 
\to H^{3}_{Y}(\sX, W_n\Omega^{2}_{\sX, \log}) \xrightarrow{\Tr_\sX} {\Z}/{p^n}
\end{equation}
gives rise to maps 
\[
H^i(\sX, W_n\Omega^j_{\sX, \log}) \xrightarrow{\Phi^{ij}_\sX}
H^{3-i}_{Y}(\sX, W_n\Omega^{2-j}_{\sX, \log})^\star; \ 
H^i_{Y}(\sX, W_n\Omega^j_{\sX, \log}) \xrightarrow{\Psi^{ij}_\sX}
H^{3-i}(\sX, W_n\Omega^{2-j}_{\sX, \log})^\star
\]
in $\pfd$. Furthermore, it follows from \thmref{thm:Zhao*} that these maps are
topological isomorphisms. In other words, ~\eqref{eqn:KS-iso-0} is a perfect pairing
in $\pfd$.

We shall now prove our duality theorem for the logarithmic Hodge-Witt cohomology
of $X$. This will prove \thmref{thm:Main-5}. This duality recovers as well as
generalizes \cite[Prop.~4]{Kato-Saito-Ann}. Let the notations and hypotheses be as in
\S~\ref{sec:setup}. In particular, $X$ is a geometrically connected, regular, and
generically smooth curve over the local field $k$ of exponential characteristic
$p > 1$. $\sX$ is a projective and flat regular semi-stable model of $X$ over
$\Spec(\sO_k)$.

\begin{thm}\label{thm:Duality-Main}
  For every pair of integers $i,j \ge 0$, the cup product on
  logarithmic Hodge-Witt cohomology induces a perfect pairing of topological
  abelian groups
  \begin{equation}\label{eqn:Duality-Main-0}
H^{i}(X, W_n\Omega^j_{X, \log}) \times  H^{2-i}(X, W_n\Omega^{2-j}_{X, \log}) 
\xrightarrow{\cup} H^{2}(X, W_n\Omega^2_{X, \log}) \xrightarrow{\Tr_X}
{\Z}/{p^n}.
\end{equation}
\end{thm}
\begin{proof}
 The existence of the bilinear pairing is clear. We shall prove its perfectness in
  several steps. In order to save space, we
  shall write all cohomology groups by suppressing the underlying schemes. We shall
  also use the short hand $\sF^j_{\sX}$ for $W_n\Omega^j_{\sX, \log}$ and $\sF^j_{X}$ for
  $W_n\Omega^j_{X, \log}$.

We now  fix $i, j \ge 0$ and we look at the diagram
 \begin{equation}\label{eqn:Duality-Main-2}  
   \xymatrix@C.8pc{
     H^{i}_{Y}(\sF^j_{\sX}) \ar[r] \ar[d]_-{\Psi^{ij}_{\sX}} &
     H^{i}(\sF^j_{\sX}) \ar[r] \ar[d]^-{\Phi^{ij}_{\sX}} \ar@{}[dr] | {\reftwo} & 
H^{i}(\sF^j_{X}) \ar[r] \ar@{.>}[d]  & 
H^{i+1}_{Y}(\sF^j_{\sX}) \ar[r] \ar[d]^-{\Psi^{i+1 j}_{\sX}} &
H^{i+1}(\sF^j_{\sX}) \ar[d]^-{\Phi^{i+1 j}_{\sX}} \\
H^{3-i}(\sF^{2-j}_{\sX})^\star \ar[r] \ar[d] & 
H^{3-i}_{Y}(\sF^{2-j}_{\sX})^\star \ar[r] \ar[d] & 
H^{2-i}(\sF^{2-j}_{X})^\star \ar[r] \ar[d] \ar@{}[dr] | {\refone} & 
H^{2-i}(\sF^{2-j}_{\sX})^\star \ar[r] \ar[d] & 
H^{2-i}_{Y}(\sF^{2-j}_{\sX})^\star \ar[d] \\
H^{3-i}(\sF^{2-j}_{\sX})^\vee \ar[r] & 
H^{3-i}_{Y}(\sF^{2-j}_{\sX})^\vee \ar[r] & 
H^{2-i}(\sF^{2-j}_{X})^\vee \ar[r]  & 
H^{2-i}(\sF^{2-j}_{\sX})^\vee \ar[r] & 
H^{2-i}_{Y}(\sF^{2-j}_{\sX})^\vee,}
\end{equation}
where the vertical arrows on the lower floor are the canonical inclusions. 
The top and the bottom rows are clearly exact while the exactness of the middle row is
shown in \lemref{lem:Exact-Dual}. All squares on the lower floor clearly commute. We
let the composite vertical arrow in the middle be the  map $(-1)^i \wt{\Phi}^{ij}_X$,
where $\wt{\Phi}^{ij}_X \colon H^{i}(\sF^{j}_{X}) \to H^{2-i}(\sF^{2-j}_{X})^\vee$
is induced by ~\eqref{eqn:Duality-Main-0}.
We now show that the composition of each of the top and the following
bottom squares in ~\eqref{eqn:Duality-Main-2} commutes. We now show that the
composition of each of the squares on the upper floor with the
one below it on the lower floor in ~\eqref{eqn:Duality-Main-2} commutes.

The composite squares on the left and right corners commute because they are induced by
the canonical commutative diagram
\begin{equation}\label{eqn:Duality-Main-3}
  \xymatrix@C1pc{
    \iota_* R\iota^{!} \sF^j_{\sX} \hspace*{.5cm} \times \hspace*{.3cm}  \sF^{j'}_{\sX}
    \ar[r]^-{\cup}
\ar@<-7ex>[d]_{\iota_*} & \iota_* R\iota^{!} \sF^{j+j'}_{\sX} \ar@{=}[d] \\
\hspace*{.6cm} \sF^j_{\sX} \hspace*{.5cm} \times \hspace*{.5cm}
\iota_* R\iota^{!}\sF^{j'}_{\sX}  
\ar[r]^-{\cup} \ar@<-7ex>[u]_-{\iota_*} &  \iota_* R\iota^{!} \sF^{j+j'}_{\sX}.}
\end{equation}

For the two middle composite squares, note that if we apply the
classical formula relating the cup product and the boundary maps in
sheaf cohomology (e.g., see \cite[Thm.~7.1]{Bredon} or \cite[Lem.~3.2]{Swan})
to the cohomology sequences associated to the (derived) tensor product
of $\sF^j_{\sX}$ with the exact triangle
\[
  \iota_* R\iota^{!} \sF^{2-j}_{\sX} \to \sF^{2-j}_{\sX} \to Ru_* \sF^{2-j}_X
\]
for $0 \le j \le 2$, we get that the diagram
\begin{equation}\label{eqn:Duality-Main-4}
  \xymatrix@C1pc{
H^{i}(\sF^j_{\sX}) \times H^{3-i}_{Y}(\sF^{2-j}_{\sX}) \ar[r]^-{\cup}
\ar@<-7ex>[d]_{u^*} &  H^{3}_{Y}(\sF^{2}_{\sX}) \\
H^{i}(\sF^j_{X}) \times H^{2-i}(\sF^{2-j}_{X}) \ar@<-7ex>[u]_-{\partial}
\ar[r]^-{\cup} &  H^{2}(\sF^{2}_{X}) \ar[u]_-{\partial}}
\end{equation}
is commutative up to multiplication by $(-1)^i$. Since $\Tr_{\sX} \circ \partial = \Tr_X$,
it follows that the two middle composite squares commute.

Since all vertical arrows on the lower floor of ~\eqref{eqn:Duality-Main-2} are injective
and all lower squares are commutative, it follows that the squares on the left and
the right corners on the upper floor are commutative. Next, an easy diagram chase shows
that the square labeled (1) is Cartesian. This implies that the image of the
composite middle vertical arrow lies in the subgroup $H^{2-i}(\sF^{2-j}_{X})^\star$.
Furthermore, if we let $\Phi^{ij}_X \colon H^{i}(\sF^j_{X}) \to
H^{2-i}(\sF^{2-j}_{X})^\star$ denote the induced map, then
all squares on the upper floor commute. Since all vertical arrows on this floor,
except possibly the middle one, are isomorphisms, it follows that $\Phi^{ij}_X$ is also an
isomorphism.

To finish the proof of the theorem, it remains to show that $\Phi^{ij}_X$ is a
continuous and open homomorphism. To show this, recall that $F_{ij}$
is a quotient of $H^{i}(\sF^j_{\sX})$ and is an open subgroup of $H^{i}(\sF^j_{X})$ such
that ${H^{i}(\sF^j_{X})}/{F_{ij}}$ is discrete. Hence, it suffices to show in
the square labeled (2) in ~\eqref{eqn:Duality-Main-2} that the composite map
$H^{i}(\sF^j_{\sX}) \to H^{3-i}_{Y}(\sF^{2-j}_{\sX})^\star \to H^{2-i}(\sF^{2-j}_{X})^\star$
is continuous and open. Let us call this composite map $\theta$.
Now, we have already seen above that $\Phi^{ij}_{\sX}$ is a
topological isomorphism. On the other hand,
$H^{3-i}_{Y}(\sF^{2-j}_{\sX})^\star \to H^{2-i}(\sF^{2-j}_{X})^\star$ is clearly continuous
because the Pontryagin dual is an auto-functor in the category of locally compact
Hausdorff topological abelian groups. This proves that $\theta$ is continuous.

To show that $\theta$ is open, it suffices to show that the map
$H^{i+1}_{Y}(\sF^{j}_{\sX})^\star \to H^{i}(\sF^{j}_{X})^\star$ is open for any $i, j \ge 0$.
From ~\eqref{eqn:Exact-Dual-4} and ~\eqref{eqn:Exact-Dual-5}, we see that this
map has a factorization
\begin{equation}\label{eqn:Duality-Main-5}
  H^{i+1}_{Y}(\sF^{j}_{\sX})^\star \surj {\rm Im}(\partial^\star) = \Ker(\iota_*)^\star \inj
  H^{i}(\sF^{j}_{X})^\star.
  \end{equation}
  The first arrow in this factorization is clearly open because it is a quotient map
  in $\pfd$ by \lemref{lem:Quotient-group}. On the other hand, $(F_{ij})^\star$ is
  discrete and $H^{i}(\sF^{j}_{X})^\star \surj (F_{ij})^\star$ is a continuous surjective
  homomorphism. It follows that its kernel ${\rm Im}(\partial^\star) =
  \Ker(\iota_*)^\star$ is open in
$H^{i}(\sF^{j}_{X})^\star$. We have thus shown that the composite map in
~\eqref{eqn:Duality-Main-5} is open. This concludes the proof.
\end{proof}

For integers $m, n \ge 1$, we have a diagram of bilinear parings
\begin{equation}\label{eqn:Limit-duality}
  \xymatrix@C.8pc{
{\sK^M_{i, X}}/{p^m} \times  {\sK^M_{j, X}}/{p^m} \ar[r]^-{\cup}
\ar@<-7ex>[d]_{p^n} &  {\sK^M_{i+j, X}}/{p^m} \ar[d]^-{p^n} \\
{\sK^M_{i, X}}/{p^{m+n}} \times  {\sK^M_{j, X}}/{p^{m+n}} \ar@<-7ex>[u]_-{\can}
\ar[r]^-{\cup} &  {\sK^M_{i+j, X}}/{p^{m+n}},}
\end{equation}   
where `$\can$' indicates the canonical surjection. It easily follows from the
bilinearity of the product in Milnor $K$-theory that this diagram is commutative.
Passing to the cohomology and using ~\eqref{eqn:GL*}, we get a bilinear pairing
\begin{equation}\label{eqn:Limit-duality-0}
``{{\underset{n}\varinjlim}}" H^{i}(X, W_n\Omega^j_{X, \log}) \times
``{{\underset{n}\varprojlim}}" H^{2-i}(X, W_n\Omega^{2-j}_{X, \log}) 
\to ``{{\underset{n}\varinjlim}}" {\Z}/{p^n}
\end{equation}
between ind-abelian and pro-abelian groups.

We endow ${\varinjlim}_n \ H^{i}(X, W_n\Omega^j_{X, \log})$
with the direct limit (i.e., the weak) topology and
${\varprojlim}_n \ H^{i}(X, W_n\Omega^j_{X, \log})$ the inverse limit
topology. Taking the limits in ~\eqref{eqn:Limit-duality-0} and using
\lemref{lem:dual-surj}, we get the following.

\begin{cor}\label{cor:Limit-duality-1}
  There is a pairing of topological abelian groups
  \[
    {\varinjlim}_n \ H^{i}(X, W_n\Omega^j_{X, \log}) \times
    {\varprojlim}_n \ H^{2-i}(X, W_n\Omega^{2-j}_{X, \log}) 
    \to {\Q_p}/{\Z_p}
  \]
  which induces an isomorphism of abelian groups
  \[
    {\varprojlim}_n \ H^{i}(X, W_n\Omega^{j}_{X, \log}) \to
    \left({\varinjlim}_n \ H^{2-i}(X, W_n\Omega^{2-j}_{X, \log})\right)^\star,
  \]
a continuous epimorphism of topological abelian groups
  \[
    {\varinjlim}_n \ H^{i}(X, W_n\Omega^j_{X, \log}) \to
   \left({\varprojlim}_n \
      H^{2-i}(X, W_n\Omega^{2-j}_{X, \log})\right)^\star.
  \]
  \end{cor}

\section{The relative Picard scheme}\label{sec:Defn}
After the duality theorem for the logarithmic Hodge-Witt cohomology, the next key
ingredients in the proof of \thmref{thm:Main-0} and the Brauer-Manin pairing for
modulus pairs are the representability and certain topological properties of the relative
Picard group. These results are of independent interest and we shall establish
them in the next two sections using \cite{Kleiman} and \cite{Serre-AGCF}.

We fix an arbitrary field $k$ of exponential characteristic $p \ge 1$. We let
$X$ be a smooth projective geometrically connected curve over $k$ and let $D \subset X$
be an effective Cartier divisor whose support is a non-empty
finite closed subset $S = \{x_1, \ldots , x_r\}$ of $X$. We shall use $\fm$ as another
notation for $D$ and call it the modulus divisor, following the terminology of
\cite{Serre-AGCF}. We let $\iota \colon D \inj X$ be the inclusion.
We shall call $(X,D)$ a modulus pair (of dimension one).
We write $D = \stackrel{r}{\underset{i = 1}
  \sum} n_i [x_i]$ as a Weil divisor and let $\deg(D) = \stackrel{r}{\underset{i = 1}
  \sum} n_i [k(x_i): k]$ denote the degree of $D$. Note that $\deg(D) =1$ if and
only if $D = \Spec(k(x))$, where $x \in X(k)$.
We let $X^o = X \setminus S$. 
We let $A = \sO_{X,S}$ and $I \subset A$ the ideal defining $D$ in $\Spec(A)$.
We write $A_D = A/I$. We let $K$ denote the function field of $X$.
For any $k$-scheme $Z$, we let $\sZ_0(Z)$ denote the free abelian group on $Z_{(0)}$.
We shall let $\sO(Z)$ (resp.
$\sO^\times(Z)$) denote the group $H^0(Z, \sO_Z)$ (resp. $H^0(Z, \sO^\times_Z)$).

\subsection{Relative Picard and 0-cycles with modulus}
\label{sec:RPM}
Recall that for an integral quasi-projective scheme $Y$ over $k$ and an
effective Cartier divisor $E \subset Y$, the Chow group of 0-cycles with modulus
$\CH_0(Y|E)$ is defined to be the quotient of $\sZ_0(Y \setminus E)$ by the subgroup
$\sR_0(Y|E)$ generated by the 0-cycles $\nu_*(\divf(f))$, where $\nu \colon C_n \to Y$
is the canonical morphism from the normalization of an integral curve $C \subset Y$
not contained in $E$ 
and $f \in \Ker(\sO^\times_{C_n, \nu^{-1}(E)} \surj \sO^\times(\nu^*(E)))$.
The relative Picard group $\Pic(Y|E)$ (usually written as
$\Pic(Y,E)$ in the literature) is the set of isomorphism classes
of pairs $(\sL, u \colon \sO_E \xrightarrow{\cong} \sL|_E)$, where
$\sL$ is an invertible sheaf on $Y$ and $\sL|_E := \sL \otimes_{\sO_Y} \sO_E$.
One says that $(\sL, u) \cong (\sL', u')$ if there is an isomorphism
$w \colon \sL \xrightarrow{\cong} \sL'$ such that $u' = w|_E \circ u$.
$\Pic(Y|E)$ is an abelian group under the tensor product of invertible sheaves and
their trivializations along $E$ whose identity element is $(\sO_Y, \id)$.
There is a canonical homomorphism $\Pic(Y|E) \to \Pic(Y)$ which is surjective if
$\dim(Y) = 1$.

We now specialize to the one-dimensional modulus pair $(X,D)$ over $k$ that we fixed
above. Let $\Pic^0(X|D)$ denote the subgroup of pairs $(\sL, u)$ such that
$\deg(\sL) = 0$. We let $\CH_0(X|D)_0$ be the kernel of the
composite map $\deg \colon \CH_0(X|D) \surj \CH_0(X) \xrightarrow{\deg} \Z$.
There is a degree preserving cycle class map $\cyc_{X|D} \colon \CH_0(X|D) \to \Pic(X|D)$
(e.g., see \cite[Lem.~3.1]{Krishna-ANT}) such that $\cyc_{X|D}([Q]) =
(\sO_X(Q), u(1) = f^{-1})$, where $Q \in X^o_{(0)}$ and $(f) \in  \sO_{X, S \cup \{Q\}}$
is the ideal of $Q$.

\begin{lem}\label{lem:CCM-iso}
    The homomorphism $\cyc_{X|D}$ is bijective.
  \end{lem}
  \begin{proof}
    This is a combination of \cite[Lemmas~2.1, 2.3, 2.4]{SV-Invent}. 
\end{proof}

From the above definitions, it is clear that there is an exact sequence
\begin{equation}\label{eqn:CCM-0}
      0 \to \sO^\times(X) \xrightarrow{\iota^*} \sO^\times(D) 
\xrightarrow{\partial_{X|D}} \Pic(X|D) \xrightarrow{\vartheta_{X|D}} \Pic(X) \to 0,
\end{equation}
where $\iota^*$ is the canonical inclusion and
$\partial_{X|D}(u) = (\sO_X, u \colon \sO_D \xrightarrow{\cong} \sO_X|_D)$.
Note that this exact sequence holds even if $X$ is only integral.
In \S~\ref{sec:Dual-mod}, we shall study the exactness of the sequence obtained by
taking the Pontryagin dual of ~\eqref{eqn:CCM-0}.
Here, we note the following consequence.

\begin{lem}\label{lem:deg-1}
The canonical map $\Pic(X|D) \to \Pic(X)$ is an isomorphism if $\deg(D) = 1$,
\end{lem}

\subsection{The singular curve attached to $(X,D)$}\label{sec:Sing-curve}
Assume now that $\deg(D) \ge 2$. In order to prove
the representability of $\Pic(X|D)$, the strategy is to replace it with the ordinary
Picard group of a singular curve (without modulus). 
The following result from \cite[Chap.~IV, \S~1, no. 3, 4]{Serre-AGCF} helps us in
achieving this.

\begin{prop}\label{prop:Sing}
There exists a unique geometrically integral $k$-scheme $X_\fm$
together with a finite morphism $\psi_\fm \colon X \to X_\fm$ such that the
following hold.
\begin{enumerate}
  \item
    $\psi_\fm(S)$ is the unique singular point $P_0 \in X_\fm$ which is $k$-rational.
\item
$\psi_\fm \colon X^o \xrightarrow{\cong} X_\fm \setminus \{P_0\}$. 
\item
  If we let $I_\fm \subset \sO_{X_\fm, P_0}$ denote the maximal ideal, then
the pull-back map $\psi^*_\fm \colon \sO_{X_\fm, P_0} \to \sO_{X,S}$
  induces a bijection $I_\fm \xrightarrow{\cong} I$.
  \end{enumerate}
\end{prop}
\begin{proof}
Since we do not assume any condition on $k$ while the construction of $X_\fm$ given
in \cite{Serre-AGCF} assumes $k$ to be algebraically closed, we sketch the proof of
this proposition to explain that this additional hypothesis is not necessary.
  
Since $S$ has an affine neighborhood in $X$, one easily observes that it suffices
to show the existence of $X_\fm$ when $X = \Spec(R)$ is affine. We let
$B = k + I = \{a + b| a \in k, b \in I\} \subset A$. Then $B$ is a $k$-subalgebra
of $A$ and there is a commutative diagram of short sequences of $B$-modules:
\begin{equation}\label{eqn:Sing-0}
  \xymatrix@C.8pc{
    & & 0 \ar[d] & 0 \ar[d] & \\
    0 \ar[r] & I \ar[r] \ar@{=}[d] & B \ar[r] \ar[d] & k \ar[r] \ar[d] & 0 \\
    0 \ar[r] & I \ar[r] & A \ar[r] \ar[d] & A_D \ar[r] \ar[d] & 0 \\
    & & A/B \ar[r]^-{\cong} \ar[d] & {A_D}/k \ar[d] & \\
  & & 0 & 0. &}
\end{equation}

Since $A_D$ is finite over $k$, it follows from the above diagram that
  there is a finite-dimensional $k$-vector space $V \subset A$ such that
  $A = V + I$. It follows that there is a surjective map of
  $B$-modules $(k \oplus V) \otimes_k B \surj A$. In particular, $A$ is
  a finite $B$-module. An application of the Eakin-Nagata theorem
  (e.g., see \cite[Thm.~3.7]{Matsumura}) shows that $B$ is Noetherian.
  Using \cite[Thm.~5.10]{AM}, we conclude furthermore that $B$ is a 1-dimensional
  local ring with maximal ideal $I$.

  We now let $L$ be the fraction field of $B$ and
  take the tensor product of the middle row in ~\eqref{eqn:Sing-0}
  with $L$ over $B$ to get a commutative diagram of short exact sequences
\begin{equation}\label{eqn:Sing-1}
  \xymatrix@C.8pc{
   0 \ar[r] & B \ar[r] \ar[d] & A \ar[r] \ar[d] & A/B \ar[d] \ar[r] & 0 \\
    0 \ar[r] & L \ar[r] & A \otimes_B L \ar[r] & ({A_D}/k) \otimes_B L \ar[r] & 0.}
  \end{equation}

Since ${A_D}/k$ is a torsion $B$-module, we get $({A_D}/k) \otimes_B L =0$.
  Since $A \otimes_B L$ is a localization of $A$, the map $A \to A \otimes_B L$
  is injective. It follows that $A \subset L$. This implies that
  the inclusion $L \inj K$ is a bijection. Putting everything together, we conclude
  that the conditions (1) to (3) of the proposition hold if we replace $X$
  by $\Spec(A)$. We now let $B'$ be the intersection (in $K$)
  of the localizations of $R$
at all maximal ideals except those supported on $S$ and let 
$R' = B \cap B'$. It is then easy to check that $X_\fm := \Spec(R')$ is the
desired singular curve with the conductor ideal $I_\fm := I$.
Since $R'_{\ov{k}} \inj R_{\ov{k}}$ and the latter ring is an integral domain,
it follows that $X_\fm$ is geometrically integral. Note also that $P$ has to be
a singular point of $X_\fm$ since $\deg(k(P)) = 1$ while $\deg(D) \ge 2$, by our
assumption.
\end{proof}

Let $X_\fm$ be the singular curve with the unique singular point $P_0$
associated to the modulus pair $(X,D)$ as obtained in \propref{prop:Sing}.
We shall call this `the $\fm$-contraction of $X$' in the sequel.
We shall identify $X_\fm \setminus \{P_0\}$ with $X^o$ via $\psi_\fm$ throughout our
discussion.

\begin{lem}\label{lem:Base-extn}
  Let ${k'}/k$ be any field extension and let $\fm' := D_{k'} \subset X_{k'}$. Then
  the $\fm'$-contraction of $X_{k'}$ is canonically isomorphic to $(X_\fm)_{k'}$.
\end{lem}
\begin{proof}
  This is a direct consequence of the construction of $X_\fm$ using its property
  that $P_0 \in X_\fm(k)$.
  \end{proof}

 Recall that the Levine-Weibel Chow group of 0-cycles $\CH^{\lw}_0(Y)$ of an integral
  curve $Y$ over $k$ is the quotient of $\sZ_0(Y_\reg)$ by the subgroup generated by the
  divisors of those rational functions on $Y$ which lie in $\sO^\times_{Y, Y_\sing}$.
It is well known that assigning each closed point of $Y_\reg$ its
class in $\Pic(Y)$ induces a canonical cycle class isomorphism
$\cyc_{Y} \colon \CH^{\lw}_0(Y) \xrightarrow{\cong} \Pic(Y)$
(e.g., see \cite[Lem.~3.12]{Binda-Krishna-Comp}).
The following is a key lemma.

\begin{lem}\label{lem:Pic-PB}
  Let ${k'}/k$ be any field extension.  
  Then the identity map of $\sZ_0(X^o)$ induces via $\psi^*_\fm$,
  a degree preserving isomorphism
  \[
    \psi^*_\fm \colon \Pic((X_\fm)_{k'}) \xrightarrow{\cong} \Pic(X_{k'}|D_{k'})
  \]
  whose composition with $\Pic(X_{k'}|D_{k'}) \surj \Pic(X_{k'})$ is the usual
  pull-back between the Picard groups of $(X_\fm)_{k'}$ and $X_{k'}$.
\end{lem}
\begin{proof}
Since the base field $k$ is arbitrary, we can assume $k' = k$ by
  virtue of \lemref{lem:Base-extn}.
  To construct $\psi^*_\fm$, we can replace $\Pic(X|D)$ and $\Pic(X_\fm)$ with
  $\CH_0(X|D)$ and $\CH^{\lw}_0(X_\fm)$, respectively, via the cycle class isomorphisms.
  We first show that $\psi^*_\fm(\divf(f)) = 0$ if $f \in \sO^\times_{X_\fm, P_0}$.
  To that end,  we let $g$ denote the residue class of $f$ in
  $k(P_0)^\times = k^\times$ and let $\tilde{f} = fg^{-1}$. Then we see that
  $\tilde{f} \in \sO^\times_{X_\fm, P_0}$ is
  such that $\divf(f) = \divf(\tilde{f})$ in $\sZ_0(X^o)$.
  Since $\tilde{f} \in K_1(A, I)$, we also see that
  $\psi^*_\fm(\divf(\tilde{f})) = 0$ in $\CH_0(X|D)$. This proves the desired
  factorization of the pull-back map between the Picard groups. It is clear that
  $\psi^*_\fm$ is degree preserving (e.g., see \cite[Tag~0AYU, Lem.~33.44.4]{SP}).

To complete the proof of the lemma, we look at the commutative diagram of
  short exact sequences
  \begin{equation}\label{eqn:Sing-2}
  \xymatrix@C1pc{
    0 \ar[r] & {A_D^\times}/{k^\times} \ar[r] \ar@{=}[d] & \Pic(X_\fm)
    \ar[r]^-{\psi^*_\fm} \ar[d]^-{\psi^*_\fm} & \Pic(X) \ar[r] \ar@{=}[d] & 0 \\
0 \ar[r] & {A^\times_D}/{k^\times} \ar[r] & \Pic(X, D)
\ar[r] & \Pic(X) \ar[r] & 0.}
\end{equation}
This shows that the middle vertical arrow is an isomorphism.
\end{proof}

It follows from \lemref{lem:Pic-PB} that there is commutative diagram
\begin{equation}\label{eqn:Sing-3}
  \xymatrix@C1pc{
    \CH^{\lw}_0(X_\fm) \ar[r]^-{\cyc_{X_\fm}} \ar[d]_-{\psi^*_\fm} &
      \Pic(X_\fm)   \ar[d]^-{\psi^*_\fm} \\
      \CH_0(X|D) \ar[r]^-{\cyc_{X|D}} & \Pic(X|D),}
  \end{equation}
  in which all arrows are isomorphisms and degree preserving. We can now prove the
  existence of the relative Picard variety of the modulus pair $(X,D)$.

\begin{thm}\label{thm:Pic-rep}
There exists a group scheme $\Picc(X|D)$
    over $k$ such that $\Picc(X|D)(k') \cong \Pic(X_{k'}| D_{k'})$ for every field
    extension ${k'}/{k}$. The identity component $\Picc^0(X| D)$ of $\Picc(X|D)$ 
    is a smooth and irreducible quasi-projective group scheme over $k$ of dimension
    equal to $\dim_k H^1(X_\fm, \sO_{X_\fm})$ such that
  $\Picc^0(X| D)(k') \cong \Pic^0(X_{k'}| D_{k'})$. If $X^o(k) \neq \emptyset$, then
  the connected components of
  $\Picc(X|D)$ are parameterized by $\Z$ and each component is isomorphic
  to $\Picc^0(X| D)$. We have $\Picc(X|D)_{k'} \cong \Picc(X_{k'}| D_{k'})$ for every
  field extension ${k'}/{k}$. If $f \colon (X',D') \to (X,D)$ is a strict morphism of
  modulus pairs (i.e., $f^*(D) = D'$), then there is a canonical homomorphism of group
  schemes $f^* \colon \Picc(X|D) \to \Picc(X'|D')$.
\end{thm}
\begin{proof}
If $\deg(D) = 1$, then $D = \Spec(k(x))$ for some $x \in X(k)$.
  In this case, we have $\Pic(X_{k'}|D_{k'}) \xrightarrow{\cong} \Pic(X_{k'})$ for
  every field extension ${k'}/{k}$ by \lemref{lem:deg-1}. The theorem therefore
  follows directly from its known case $D = \emptyset$. We now assume $\deg(D) \ge 2$
  and let $\psi_\fm \colon X \to X_\fm$ be the $\fm$-contraction of $X$ (cf.
  \propref{prop:Sing}).
Since $X_\fm$ is geometrically integral, it follows from the flat base change property
of the higher direct images of coherent sheaves that for every $k$-scheme $T$,
the map $\sO_T \to (f_T)_*(\sO_{X_\fm \times T})$ is an isomorphism of sheaves on $T$,
where $f \colon X_\fm \to \Spec(k)$ is the structure map and $f_T \colon 
X_\fm \times T \to T$ is the base change of $f$ to $T$. Since $P_0 \in X_\fm(k)$,
the map $f$ has a section.

Using the last two properties of $X_\fm$, we conclude
from \cite[Thm.~9.2.5, 9.4.8]{Kleiman} that there exists a group scheme $\Picc(X_\fm)$
over $k$ such that $\Picc(X_\fm)(T) \cong {\Pic(X_\fm \times T)}/{\Pic(T)}$ for 
every $k$-scheme $T$. In particular, we have
$\Picc(X_\fm)(k') \cong \Pic((X_\fm)_{k'})$ for every field
 extension ${k'}/{k}$. Furthermore, the identity component $\Picc^0(X_\fm)$ of
 $\Picc(X_\fm)$ is a geometrically irreducible quasi-projective group scheme over $k$
 and $\Picc^0(X_\fm)(k') \cong \Pic^0((X_\fm)_{k'})$ by \cite[Lem.~9.5.1]{Kleiman}.
It follows from \cite[Cor.~4.2]{Oort} (see also \cite[Prop.~9.5.19]{Kleiman})
 and the fppf descent property of smoothness that
  $\Picc^0(X_\fm)$ is a smooth group scheme over $k$, i.e., it is an
  algebraic group over $k$. In particular, it is geometrically integral.
  The assertion about $\dim(\Picc^0(X|D))$ follows from its smoothness and
  \cite[Cor.~9.5.13]{Kleiman}.

  If $P \in X^o(k)$, then the connected components of $\Picc(X_\fm)$ are
parameterized by $\Z$, and each connected component is isomorphic to $\Picc^0(X_\fm)$
via the transformation $\sL \mapsto \sL \otimes \sO_{X_\fm}(-m P)$ on the scheme points,
where $\deg(\sL) = m \in \Z$.  The theorem now follows
from Lemmas~\ref{lem:Base-extn} and ~\ref{lem:Pic-PB} if we let 
$\Picc(X|D) := \Picc(X_\fm)$. The last claim is clear because the Picard scheme
is a contravariant functor.
\end{proof}

\subsection{Picard scheme of singular curves}\label{sec:Pic-SS}
  Suppose that $Y$ is a geometrically integral singular curve with the normalization
  $f \colon X \to Y$. Let $Z \subset Y$ be a conductor subscheme for $f$ with
  support $Y_\sing$ and let $D = Z \times_Y X$. Assume that $X$ is smooth over $k$.
  Then $(X, D)$ is a 1-dimensional modulus pair, and it is easy to see from the
  construction in \propref{prop:Sing} that the $\fm$-contraction of $X$ has a unique
  factorization
  \begin{equation}\label{eqn:Pic-Sing**-0}
    \xymatrix@C1pc{
      X \ar[r]^-{f} \ar[dr]_-{\psi_\fm} & Y \ar[d]^-{\wt{\psi}_\fm} \\
      & X_\fm.}
    \end{equation}
    We conclude from \thmref{thm:Pic-rep} that there are morphisms of group schemes
    \begin{equation}\label{eqn:Pic-Sing**-1}
    \Picc(X|D) \xrightarrow{\wt{\psi}^*_\fm} \Picc(Y)
      \xrightarrow{f^*} \Picc(X)
    \end{equation}
    whose composition is the canonical morphism $\Picc(X|D) \surj \Picc(X)$.

We now recall the exact sequence
\begin{equation}\label{eqn:Pic-Sing**-2}
  0 \to k^\times \to \sO^\times(Z) \to H^1_\nis(Y, \sK^M_{1, (Y,Z)})  \to
  H^1_\nis(Y, \sK^M_{1,Y}) \to 0.
\end{equation}
Using the isomorphisms  $H^1_\nis(Y, \sK^M_{1, (Y,Z)}) \xrightarrow{\cong}
H^1_\nis(X, \sK^M_{1, (X,D)}) \cong \Pic(X|D)$ (see \cite[Lem.~3.1]{Krishna-ANT} for the
second isomorphism), we get an exact sequence
\begin{equation}\label{eqn:Pic-Sing**-3}
  0 \to {\sO^\times(Z)}/{k^\times} \to \Pic(X|D) \to \Pic(Y) \to 0.
\end{equation}
Since this remains true over every algebraic extension of $k$, it follows that
the map $\wt{\psi}^*_\fm$ in ~\eqref{eqn:Pic-Sing**-1} is surjective.

\section{Relative Picard group over local fields}\label{Alb-map}
  The goal of this section is to prove some topological properties of the relative
  Picard group and the albanese map with respect to the adic topology that the
  schemes acquire when the base field is a local field. Before we do this, we need to
  study the smoothness properties of the albanese
  morphism to the relative Picard scheme and of the morphism from the 
  relative to the ordinary Picard scheme.
  We continue with the notations and assumptions of \S~\ref{sec:Defn}
  (in particular, $k$ is still arbitrary).

 \subsection{Generic smoothness of albanese map}\label{sec:Alb-sm}
  Until we reach \thmref{thm:ALB-rel}, we shall assume in this subsection that
  $\deg(D) \ge 2$ and $X^o(k) \neq \emptyset$. We fix a point $P \in X^o(k)$.
  For any integer $n \ge 1$, let $\mathbf{Div}^n(X_\fm)$ be the quasi-projective
$k$-scheme which parameterizes effective Cartier divisors $E \subset X_\fm$ of
degree $n$. By \cite[Defn.~9.4.6]{Kleiman}, there is a canonical morphism
of $k$-schemes ${\aj}^n_{X_\fm} \colon \mathbf{Div}^n(X_\fm) \to \Picc^0(X_\fm)$ which
sends $E$ to $\sO_{X_\fm}(E-nP)$ if the support of $E$ lies in $X(k)$.
Since $\mathbf{Div}^1(X_\fm) \cong X^o$ by
\cite[Exc.~9.3.8]{Kleiman}, we get a $k$-morphism $\alb_{X|D} \colon X^o \to
\Picc^0(X_\fm)$ (called the albanese morphism)
which has the property that $\alb_{X|D}(x) = [x] -[P]$ for
any $x \in X^o(k)$. Furthermore, it induces
\begin{equation}\label{eqn:Pic-universal}
  {\aj}_{X|D} \colon \CH^{\lw}_0(X_\fm)_0 \cong \Pic^0(X_\fm) \xrightarrow{\cong}
  \Picc^0(X_\fm)(k).
\end{equation}
For $n \ge 1$, the $n$-th power of the albanese map descends to a morphism
$\alb^n_{X|D} \colon \Sym^n(X^o) \to \Picc^0(X_\fm)$, where $\Sym^n(X^o)$ is the
$n$-th symmetric power of $X^o$. We let
  \begin{equation}\label{eqn:ALB}
    \phi^n_{X|D} \colon
    (X^o)^n \xrightarrow{\pi_n} \Sym^n(X^o) \xrightarrow{\alb^n_{X|D}}
    \Picc^0(X_\fm).
  \end{equation}
  denote the composite morphism.
Let $p_a(X_\fm) = \dim_k(H^1(X_\fm, \sO_{X_\fm}))$ denote the arithmetic genus of
$X_\fm$.

Let us now assume that $E \subset X_\fm$ is an effective Cartier divisor disjoint
    from the unique singular point $P_0$ of $X_\fm$ and let $\sL = \sO_{X_\fm}(E)$.
    Equivalently, $E$ is an effective Cartier divisor on $X$ disjoint from $S$ and
    $\sL':= \sO_X(E) \cong \psi^*_\fm(\sL)$.
Since  $E$ is effective, there is a canonical inclusion $\sO_{X} \inj \sL'$.
   This yields a canonical inclusion
  $k \cong H^0(X, \sO_X) \inj  H^0(X, \sL')$ such that the composite $k \inj
  H^0(X, \sL') \to A/I$ is the canonical inclusion of rings $k \inj A/I$. 
Using this and the tensor product of the exact sequence 
\begin{equation}\label{eqn:fiber-dim-0}
  0 \to \sO_{X_\fm} \to (\psi_\fm)_*(\sO_X) \to {\sO_D}/{\sO_{P_0}}
  \to 0
\end{equation}
with $\sL$, we get a commutative diagram with exact rows
  \begin{equation}\label{eqn:Sym-1}
    \xymatrix@C.8pc{
      & & k \ar[d] \ar[dr] & \\
      0 \ar[r] & H^0(X, \sL'(-D)) \ar[r] \ar[d] & H^0(X, \sL') \ar[r] \ar@{=}[d] &
      A_D \ar@{->>}[d] \\
      0 \ar[r] & H^0(X_\fm, \sL) \ar[r] & H^0(X, \sL') \ar[r] & {A_D}/k.}
  \end{equation}
  Hence, we get an exact sequence
 \begin{equation}\label{eqn:Sym-2}
   0 \to H^0(X, \sL'(-D)) \to H^0(X_\fm, \sL) \to k \to 0.
   \end{equation}

 For a line bundle $\sL$ on $X_\fm$, let $|\sL|$ denote the complete
   linear system  $\Proj_k(H^0(X_\fm, \sL))$. 
   If $\sL = \sO_{X_\fm}(E)$ for an effective divisor $E$ on $X^o$, we let
   $U_E = |\sL| \setminus |\sL'(-D)| \subset |\sL|$.

\begin{lem}\label{lem:Linear-system}
    For any Cartier divisor $E$ on $X_\fm$, the linear system $|\sO_{X_\fm}(E)|$
    represents effective Cartier divisors $F$ on $X_\fm$ such that
    $\sO_{X_\fm}(F) \cong \sO_{X_\fm}(E)$ as line bundles on $X_\fm$.
  \end{lem}
  \begin{proof}
    One checks that the proof of a similar statement for
    smooth schemes given in \cite[Prop.~II.7.7]{Hartshorne-AG} (see also
    \cite[\S~13.13, p.~395]{GW}) are based on two claims
    both of which are met in the present case.
    \end{proof}

\begin{lem}\label{lem:alb-fiber}
  Assume that $k$ is algebraically closed and let $E \subset X^o$ be an
  effective Cartier divisor of degree $n \ge 1$.
  Then $E$ is canonically an element of $\Sym^n(X^o)(k)$ and there is a canonical
  isomorphism of $k$-schemes $(\alb^n_{X|D})^{-1}(\alb^n_{X|D}(E)) \cong U_E$.
\end{lem}
\begin{proof}
It is clear that $\alb^n_{X|D}$ has factorization $\Sym^n(X^o) \inj
\mathbf{Div}^n(X_\fm) \xrightarrow{{\aj}^n_{X|D}} \Picc^0(X_\fm)$. As $k$ is
algebraically closed, ${\aj}^n_{X|D}$ is given by $E \mapsto [E] - [nP] \in
\Picc^0(X_\fm)$. Using \lemref{lem:Linear-system}, it is clear from this that
    $(\alb^n_{X|D})^{-1}(\alb^n_{X|D}(E))  = |\sL|  \cap \Sym^n(X^o)$.
It follows from ~\eqref{eqn:Sym-1}  that, up to multiplication by elements of
    $k^\times$, the latter set consists of those sections of $\sL$ which
  do not die in $k$ under the map $H^0(X_\fm, \sL) \to k$
  in the exact sequence ~\eqref{eqn:Sym-2}. In other words,
  $|\sL| \cap \Sym^n(X^o) = U_E$. This finishes the proof.  
\end{proof}

\begin{lem}\label{lem:fiber-dim}
  Assume that $k$ is algebraically closed. Then for all integers $n \gg 0$, all
  non-empty fibers of the morphism $\alb^n_{X|D} \colon \Sym^n(X^o) \to \Picc^0(X_\fm)$
  are affine spaces over the corresponding residue fields of $\Picc^0(X_\fm)$,
  and have equal and positive dimensions.
\end{lem}
\begin{proof}
For any effective Cartier divisor $E \subset X^o$, we let
  $\sL = \sO_{X_\fm}(E)$ and $\sL' = \psi^*_\fm(\sL)$.
  Since $|\sL'(-D)|$ is a hyperplane of the projective space $|\sL|$,
  it follows from \lemref{lem:alb-fiber} that all fibers
  of the morphism $\alb^n_{X|D} \colon \Sym^n(X^o) \to \Picc^0(X_\fm)$ are affine
  spaces over the corresponding residue fields of $\Picc^0(X_\fm)$. It remains to prove
  the independence of these fiber dimensions if $n \gg 0$.

Suppose that $E$ is such that $\deg(E) = n \ge \deg(D) + p_a(X_\fm) + 2g -1$, where
$g = p_a(X)$ denotes the
genus of $X$. Then $\deg(\sL'(-D)) = \deg(\sL') - \deg(D) = \deg(E) - \deg(D) \ge
2g-1$. It follows from the Serre duality and the Riemann-Roch theorem for $X$ that 
\begin{equation}\label{eqn:fiber-dim-2}
  H^1(X, \sL'(-D)) = 0 = H^1(X, \sL').
  \end{equation}
Using the cohomology exact sequence associated to the sheaf exact sequence
\begin{equation}\label{eqn:fiber-dim-3}
  0 \to \sL'(-D) \to \sL' \to \sO_D \to 0, 
\end{equation}
we get an exact sequence
\begin{equation}\label{eqn:fiber-dim-4}
  0 \to H^0(X, \sL'(-D)) \to H^0(X, \sL') \to H^0(D, \sO_D) \to 0.
\end{equation}
In particular, the map $ H^0(X, \sL') \to {H^0(D, \sO_D)}/k$ is surjective.
On the other hand, by tensoring ~\eqref{eqn:fiber-dim-0} with  $\sL$ and passing to
cohomology, we get an exact sequence
\begin{equation}\label{eqn:fiber-dim-1}
0 \to H^0(X_\fm, \sL) \to H^0(X, \sL') \to {H^0(D, \sO_D)}/k \to 
H^1(X_\fm, \sL) \to H^1(X, \sL') \to 0.
\end{equation}
We conclude that $H^1(X_\fm, \sL) = 0$.

Now, we use the Riemann-Roch theorem for $X_\fm$ (e.g., see
\cite[Exc.~IV.1.9]{Hartshorne-AG}) to get the identity
$\dim_k(H^0(X_\fm, \sL)) = n + 1 - p_a(X_\fm)$.
  In other words, $\dim_k(U_E) = \dim_k(|\sL|) = n - p_a(X_\fm) \ge
  \deg(D) +2g -1 \ge 1$. Since $\deg(D) +2g -1$ depends only on $(X,D)$, we
  conclude that all fibers of $\alb^n_{X|D}$ are of a fixed positive dimension.
\end{proof}

\vskip .2cm

 We now let $k$ be any field.

\begin{prop}\label{prop:Alb-main-1}
    For all integers $n \gg 0$, the morphism
    $\alb^n_{X|D} \colon \Sym^n(X^o) \to \Picc^0(X_\fm)$ is smooth and surjective.
  \end{prop}
  \begin{proof}
    Via the standard descent properties of the smooth and surjective morphisms with
    respect to fpqc morphisms to the base scheme
    (e.g., see \cite[Prop.~14.48, \S~16.C]{GW}), and the commutativity of
    ${\fm}$-contraction and $\Picc^0(X_\fm)$ with the field extensions, we easily
    reduce the proof to the case where we can assume $k$ to be algebraically closed. 

Since $k = \ov{k}$ and $\CH^{\lw}_0(X_\fm)_0 \xrightarrow{\cong} \Pic^0(X_\fm)
    \xrightarrow{\cong} \Picc^0(X_\fm)(k)$, it already follows from
    \cite[Defn.~7.1, Thm.~7.2]{ESV} that $\alb^n_{X|D}$ is surjective for all
    $n \gg 0$. Using the fpqc descent, the smoothness claim now follows from
    \lemref{lem:fiber-dim} and \cite[Exc.~III.10.9, Thm.~III.10.2]{Hartshorne-AG}.
\end{proof}

For $n \gg 0$ and $1 \le i < j \le n$, we let
$\Delta^n_{ij}(X^o) \subset (X^o)^n$ be the closed subscheme (via a permutation)
$\Delta_{ij} \times (X^o)^{n-2} \subset (X^o)^n$,
where $\Delta_{ij} \subset X^o \times X^o$ is the diagonal of the product of
the $i$-th and $j$-th factors of $(X^o)^n$. We let $U_n =
(X^o)^n \setminus ({\underset{1 \le i < j \le n}\bigcup} \Delta^n_{ij}(X^o))$.
Then $U_n$ is invariant under the action of the symmetric group $S_n$ and there is a
Cartesian square
\begin{equation}\label{eqn:Quotient}
  \xymatrix@C1pc{
    U_n \ar[r]^-{j} \ar[d]_-{\pi_n} & (X^o)^n \ar[d]^-{\pi_n} \\
    {U_n}/{S_n} \ar[r]^-{j}  & \Sym^n(X^o).}
  \end{equation}
 Furthermore, $\pi_n \colon U_n \to {U_n}/{S_n}$ is an {\'e}tale $S_n$-torsor
 (e.g., see \cite[Prop.~0.9]{Fogarty-Mumford}).
 In particular,  it is finite and {\'e}tale. We let $U_n(X^o) = {U_n}/{S_n}$. 

\begin{cor}\label{cor:Alb-main-2}
For all $n \gg 0$, the morphism $\phi^n_{X|D} \colon U_n(X^o) \to 
\Picc^0(X_\fm)$ is smooth and has dense image.
\end{cor}
\begin{proof}
    Since $\phi^n_{X|D} = \alb^n_{X|D} \circ \pi_n$ (see ~\eqref{eqn:ALB}) and
    $U_n(X^o) \xrightarrow{\pi_n} \Sym^n(X^o)$ is {\'e}tale with dense image,
    the corollary follows from \propref{prop:Alb-main-1}.
\end{proof}

We can now prove the following generic smoothness of the albanese map for a
modulus pair. We let $k$ be any field and $(X,D)$ a 1-dimensional modulus pair over $k$
such that $X$ is connected and smooth over $k$ and
$X^o(k) \neq \emptyset$. We fix a point $P \in X^o(k)$
so that $\alb_{X|D}$ is defined with respect to the base point $P$.

\begin{thm}\label{thm:ALB-rel}
  For all $n \gg 0$, the albanese map
  $\phi^n_{X|D} \colon U_n(X^o) \to \Picc^0(X|D)$ is smooth and has dense image.
\end{thm}
\begin{proof}
  Our assumption implies that $X$ is geometrically integral over $k$.
  If $\deg(D) = 1$, then the canonical map $\Picc^0(X|D) \to \Picc^0(X)$ is an
  isomorphism of algebraic groups over $k$ (see \lemref{lem:deg-1})
  and one knows classically that
  the albanese map $U_n(X^o) \to \Picc^0(X)$ is smooth and has dense image.
  If $\deg(D) \ge 2$, the theorem follows from \thmref{thm:Pic-rep} and
  \corref{cor:Alb-main-2}.
\end{proof}

\subsection{The map from relative to ordinary Picard scheme}
\label{sec:Comp-RO}
In order to compare the relative and ordinary Picard schemes, we need to identify
$\Picc^0(X|D)$ with the generalized Jacobian variety of $(X,D)$ {\`a} la Rosenlicht-Serre
\cite{Serre-AGCF} under some additional assumptions.
We fix a field $k$ and a 1-dimensional modulus pair $(X,D)$ over $k$ such that $X$ is
connected and smooth over $k$. We write
$D = \stackrel{r}{\underset{i =1} \sum} n_i[x_i]$.
If $\deg(D) \ge 2$, we let $\psi_\fm \colon X \to X_\fm$ be the $\fm$-contraction
of $X$.  The following is the main result of Rosenlicht-Serre
for which we refer to \cite[Chap.~V, \S~4.22, 4.23]{Serre-AGCF}.

\begin{thm}\label{thm:Serre-Jac}
  Assume that ${\rm Supp}(D) \subset X(k)$ and $P \in X^o(k)$.
 Then there exists a quasi-projective algebraic group $J_{X|D}$ over $k$ satisfying
  the following properties.
  \begin{enumerate}
\item
There is a canonical isomorphism of group schemes $J_{X_{k'}|D_{k'}} \cong
  (J_{X|D})_{k'}$ over $k'$ for every field extension ${k'}/k$.
\item
  The assignment $Q \mapsto [Q] - \deg(Q) [P]$ on $X^o_{(0)}$
  defines a morphism of $k$-schemes
  (called the albanese morphism) ${\lambda}_{X|D} \colon X^o \to J_{X|D}$.
\item
  The albanese morphism induces a homomorphism of abelian groups
  \[
    {\aj}_{X|D} \colon \CH_0(X|D)_0 \to J_{X|D}({k}).
    \]
 \item
 Given any commutative group scheme $G$ over $k$ and a $k$-morphism 
  $f \colon X^o \to G$ which induces a homomorphism
  $f_* \colon \CH_0(X|D)_0 \to G(k)$, there is a unique homomorphism of
  group schemes $\theta \colon J_{X|D} \to G$ over $k$
  such that $f(x) = f(P) + (\theta \circ \lambda_{X|D})(x)$ for every $x \in X^o$. 
\end{enumerate}
\end{thm}

Assume that $\deg(D) \ge 2$. Using the albanese map $\alb_{X|D} \colon X^o \to
\Picc^0(X_\fm)$ and ~\eqref{eqn:Pic-universal}, it follows from the property (4) of
$J_{X|D}$ that there exists a unique homomorphism (recall here that
$\alb_{X|D}(P) = 0$) of group schemes $\theta_{X|D} \colon J_{X|D} \to \Picc^0(X_\fm)
\cong \Picc^0(X|D)$ over $k$ such that $\alb_{X|D} = \theta_{X|D} \circ \lambda_{X|D}$.

\begin{cor}\label{cor:Pic-Jac-iso-1}
  Under the assumptions of \thmref{thm:Serre-Jac}, the albanese map
  $\lambda_{X|D} \colon X^o \to J_{X|D}$ induces an isomorphism of
  algebraic groups
  $\theta_{X|D} \colon J_{X|D} \xrightarrow{\cong} \Picc^0(X|D)$.
\end{cor}
\begin{proof}
  If $\deg(D) = 1$, it easily follows from \cite[Chap.~V, \S~3.13]{Serre-AGCF}
  that the canonical map $J_{X|D} \to J_X$ is an isomorphism. We have seen
  above that this holds also for $\Picc^0(X|D)$.
  If $\deg(D) \ge 2$, we can replace $\Picc^0(X|D)$ with $\Picc^0(X_\fm)$.
  In the latter case,  we can argue as in the proof of \propref{prop:Alb-main-1} to
  reduce the proof to the case where we can assume that $k$ is algebraically closed.
  But this case is classically known.
\end{proof}

We can now compare the relative and ordinary Picard schemes of a modulus pair.
For any $n \ge 1$, let $\W_n$ be the {\'e}tale sheaf on $\Sch_k$ given
by $\W_n(Z) = \W_n(\sO(Z))$, where $\W_n(R)$ is the ring of big Witt-vectors
over the ring $R$ of length $n$ (e.g., see \cite[App.~A]{Rulling}).
One knows that $\W_n$ is representable by a unipotent linear algebraic group over $k$
(e.g., see \cite[Chap.~V, \S~3.13, Lem.~20]{Serre-AGCF}).

\begin{prop}\label{prop:Pic-sequence}
  Under the assumptions of \thmref{thm:Serre-Jac}, there is an exact sequence of
  algebraic groups
 \begin{equation}\label{eqn:Pic-sequence-0}
   0 \to \Picc^0(X|D)_{\rm aff} \to \Picc^0(X|D) \to \Picc^0(X) \to 0
 \end{equation}
 over $k$, where $\Picc^0(X|D)_{\rm aff}$ is a linear algebraic group 
 canonically isomorphic to $(\G_m)^{r-1} \times (\stackrel{r}{\underset{i =1}
     \prod} \W_{n_i-1})$.
   \end{prop}
   \begin{proof}
     Using \corref{cor:Pic-Jac-iso-1}, we can work with $J_{X|D}$, in which case
     the proposition is well known (e.g., see \cite[Chap.~V, \S~3, no.~13]{Serre-AGCF}).
\end{proof}

\begin{cor}\label{cor:Pic-sequence-1}
Let $X$ be a smooth projective geometrically integral curve over $k$ and let
$D \le D'$ be two effective Cartier divisors on $X$.
  Then the canonical map $\Picc^0(X|D') \to \Picc^0(X|D)$ is smooth, affine and
  surjective. If $Y$ is a singular curve over $k$ with normalization $X$ and
  $Z \subset Y$ is a conductor subscheme for the normalization such that
  $D = Z \times_Y X$, then the map $\Picc^0(X|D) \to \Picc^0(Y)$
  in ~\eqref{eqn:Pic-Sing**-1} is smooth, affine and surjective.
  \end{cor}
\begin{proof}
  To prove either of the statements, we can pass to the algebraic closure of $k$.
  In this case, the first statement follows directly by comparing
  ~\eqref{eqn:Pic-sequence-0} for $\Picc^0(X|D')$ and $\Picc^0(X|D)$.
  The second statement follows directly from ~\eqref{eqn:Pic-Sing**-3}.
\end{proof}

\subsection{The case of local fields}\label{sec:Pic-loc}
We shall assume in this subsection that $k$ is a local field. Recall that $k$ is a
topological field with respect to its valuation  (also called adic) topology such that
the ring of integers $\sO_k$ is an open subring. One also knows
that for any locally of finite type $k$-scheme $Y$, the set $Y(k)$ has a unique
structure of a totally disconnected locally compact Hausdorff topological space
(see \cite[Prop.~5.4]{Conrad}). This topology is characterized by the property that if
$Y \subset \A^n_k$ as a locally closed subscheme, then it coincides with the subspace
topology on $Y(k)$ induced by the product of the adic topology of $k$ on
$\A^n_{k}(k) \cong k^n$.
This assignment of adic topology defines a functor $\Sch_k \to \Top$, where the latter
is the category of topological spaces with continuous maps.
We shall call this the adic topology of $Y(k)$ and say that $Y(k)$ is an adic space.
For any property of the adic topology of $Y(k)$ that we shall use in this paper, the
reader is referred to \cite{Conrad} and \cite[Thm.~10.5.1]{CTS}.

Every finite-dimensional $k$-vector space $M$
is equipped with the product topology of the adic topology of $k$, called the adic
topology of $M$. If $A$ is a finite-dimensional $k$-algebra, the adic topology of
$A^\times$ is the subspace topology induced from $A$. These topologies make
$A$ and $A^\times$ into topological abelian groups. If ${k'}/k$ is a finite field
extension and $A$ is a finite-dimensional $k'$-vector space, then
the adic topology of $A$ (and $A^\times$ if $A$ is a $k'$-algebra) coincides
with its adic topology when considered as a $k$-vector space
(e.g., see \cite[\S~7.2]{Kato-cft-1}). 
In this paper, all finite-dimensional algebras over a local field
and their unit groups will
be assumed to be endowed with the adic topology unless mentioned otherwise.

We now let $(X,D)$ be as in \S~\ref{sec:Comp-RO}. We assume that $X^o(k) \neq \emptyset$
and fix a point $P \in X^o(k)$ so that $\alb_{X|D}$ is defined with respect to the base
point $P$. Let $Y$ be a singular curve
over $k$ with normalization $X$ and a conductor subscheme $D \subset X$.
Since $\Pic(X|D) = \Picc(X|D)(k)$ and $\Pic(Y) = \Picc(Y)(k)$,
we see that $\Pic(X|D)$ and $\Pic(Y)$ are
locally compact Hausdorff topological abelian groups. Furthermore, it follows from
\cite[Thm.~10.5.1]{CTS} that $\Pic^0(X|D)$ and $\Pic^0(Y)$ are open subgroups
of $\Pic(X|D)$ and $\Pic(Y)$, respectively. In the following result, we shall use
the adic topology on all sets.

\begin{cor}\label{cor:Pic-sequence-2}
  We have the following.
  \begin{enumerate}
  \item
    For all $n \gg 0$, the map $\phi^n_{X|D} \colon U_n(X^o)(k) \to
    \Pic^0(X|D)$ is a continuous open map whose image is dense.
  \item
    $\Pic^0(X)$ is a profinite abelian group.
  \item
    The map $\Pic^0(X|D) \to \Pic^0(X)$ is a topological quotient.
  \item
    If $D' \ge D$, the map $\Pic^0(X|D') \to \Pic^0(X|D)$ is a topological quotient.
  \item
    The map $\Pic^0(X|D) \to \Pic^0(Y)$ is a topological quotient.
  \end{enumerate}
  \end{cor}
  \begin{proof}
    The statement (1) follows directly from \thmref{thm:ALB-rel} and
    \cite[Thm.~10.5.1]{CTS}. To prove (2), we note that $\Picc^0(X)$ is an
    abelian variety over $k$. This implies that $\Pic^0(X)$ is compact as an adic
    space (see op. cit.). Since it is also totally disconnected, $\Pic^0(X)$ must be
    profinite. Since the maps $\Pic^0(X|D') \to \Pic^0(X|D)$ and
    $\Pic^0(X|D) \to \Pic^0(Y)$ are surjective, the remaining statements
    of the corollary follow directly from \corref{cor:Pic-sequence-1} and
    \cite[Thm.~10.5.1]{CTS}.
\end{proof}

\section{Pontryagin dual of the Chow group with modulus}
\label{sec:Dual-mod}
The goal of this section is to strengthen \corref{cor:Pic-sequence-2}.
More precisely, we shall show that all maps in the exact sequence
~\eqref{eqn:CCM-0} are continuous with respect to the adic topology and the
resulting dual complex is partially exact. This will be a key step in the proofs of
the main results. We begin by recalling the Kato topology and its relation with the
adic topology.

\subsection{Kato topology}\label{sec:Kato-T}
Let $R$ be an equicharacteristic excellent Henselian discrete valuation ring with
maximal ideal $\fm$ whose residue field $\ff$ is a local field of exponential
characteristic $p \ge 1$. Let $L$ denote the quotient field of $R$. In this case, one can
find a two dimensional excellent normal local integral domain $A$ whose residue field is
finite such that $R \cong (A_\fp)^h$ for some height one prime ideal $\fp \subset A$.
For any ideal $I \subset A$ not contained in $\fp$ and $n \ge 1$, we let
$\Fil^I_n {K}^M_1(R)$ be the subgroup of $R^\times$ generated by $(1 + \fm^n)$
and $(1 + I) \subset A^\times \subset {K}^M_1(R)$. We let $\Fil^I_0 {K}^M_1(R) =
R^\times$.

When $p > 1$, Kato defined a subgroup topology on $R^\times$
(see \cite[\S~7]{Kato-cft-1} and \cite[\S~2.3]{Saito-Invent})
    for which the fundamental system of open neighborhoods of the identity is given
    by subgroups of the form $\Fil^I_n {K}^M_1(R)$, where $I$ and $n$
    vary as above. It follows from \cite[\S~7, Lem.~1]{Kato-cft-1}
    (see also \cite[\S~2.3]{Saito-Invent}) that this topology of
    $R^\times$ does not depend on the choice of $A$. We shall call this the
    Kato topology on $R^\times$. The Kato topology on $L^\times$ is the unique topology
    which is compatible with its group structure and for which $R^\times$ with its
    Kato topology is an open subgroup.

    It is easy to check that the Kato topology of $R^\times$ coincides
  with the subspace topology induced from the Kato topology of $\wh{R}^\times$ via the
    inclusion $R^\times \subset \wh{R}^\times$.
  In particular, the same also holds for the inclusion $L^\times \subset \wh{L}^\times$. 
When $p = 1$, we shall assume the Kato topology on $R^\times$ and $L^\times$ to be
    discrete. By choosing a uniformizer $\pi \in \fm$, we get a canonical
    isomorphism of $\ff$-algebras $\phi \colon \ff[[T]] \xrightarrow{\cong}
    \wh{R}$ which sends $T$ to $\pi$.
In \S~\ref{sec:Pic-loc}, we recalled the adic topology of $\ff$ and $\sO_\ff$. 
We shall use the following description of the Kato topology in terms of the adic
topology.

 \begin{lem}\label{lem:Kato-1}
      When $p > 1$, the map $\phi \colon (\ff[[T]])^\times \to  (\wh{R})^\times$ is an
      isomorphism of topological abelian groups if we endow $\ff[[T]] \cong \ff \times
      \ff^{\N}$ with the product topology of the adic topology of $\ff$ and
      $(\ff[[T]])^\times$ with the subspace topology. In particular, the quotient
      topology of the Kato topology on $({\wh{R}}/{\fm^n})^\times$ coincides with the
      adic topology of $({\ff[T]}/{(T^n)})^\times$. 
     \end{lem}
    \begin{proof}
This is a straightforward consequence of \cite[\S~7, Rem.~1, Lem.~3]{Kato-cft-1} .
\end{proof}

Recall (e.g., see \cite[App.~A]{Rulling}) that for any commutative ring $A$
and the set $\{1, \ldots , m\}$ (where $m = \infty$ is allowed),
the ring of big Witt-vectors of length $m$ is the commutative ring $\W_m(A)$ which is
$A^m$ as a set and whose ring structure is uniquely determined by the condition that the
ghost map ${\rm gh_m} \colon \W_m(A) \to A^m$ is a ring homomorphism
with respect to the product ring structure on $A^m$. The ghost map is in fact a natural
transformation of functors between commutative rings. We write $\W_\infty(A)$ as $\W(A)$. 

Let $\Gamma(A) := (1 + TA[[T]])^\times \subset (A[[T]])^\times$ be the subgroup of the
group of units of the power series ring $A[[T]]$. Recall that every element of
$\Gamma(A)$ has a unique presentation of the form $f(T) =
{\underset{n \ge 1}\prod} (1-a_nT^n)$ with $a_n \in A$.
One knows (e.g., see \cite[App.~A]{Rulling})
that as an abelian group with respect to its addition operation, $\W(A)$ and
$\W_m(A)$ have a simple description. That is, the maps
\begin{equation}\label{eqn:Witt-0}
  \nu_A \colon \W(A) \to \Gamma(A); \ \nu_A((a_n)) =
  {\underset{n \ge 1}\prod} (1-a_nT^n), \ \ and
\end{equation}
\begin{equation}\label{eqn:Witt-1}
  \mu_A \colon \Gamma(A) \to \W(A), \ \ \mu_A({\underset{n \ge 1}\prod} (1-a_nT^n))
  = (a_n).
\end{equation}
are group homomorphisms which are inverses to each other. These maps induce isomorphisms
\begin{equation}\label{eqn:Witt-2}
  \nu_A \colon \W_m(A) \xrightarrow{\cong} \Gamma_m(A) :=
  \frac{(1 + TA[[T]])^\times}{(1 + T^{m+1}A[[T]])^\times}; \ \
\mu_A \colon  \Gamma_m(A) \xrightarrow{\cong} \W_m(A).
\end{equation}
It follows from ~\eqref{eqn:Witt-1} that there are polynomials
$p_i \in A[T_1, \ldots , T_m]$ such that 
$\mu_A$ in ~\eqref{eqn:Witt-2} is
of the form $\mu_A(1 + a_1T + \cdots + a_mT^m) = (p_1(a_1, \ldots , a_m),
\ldots , p_m(a_1, \ldots , a_m))$. We shall let $\W_m(A) = 0$ for $m \le 0$.

\vskip .2cm

Recall from \S~\ref{sec:Comp-RO} that for any field $F$, 
there is a unipotent linear algebraic group ${\W}_{m,F}$ over $F$
such that  $\W_{m}(A) = {\W}_{m,F}(A)$ for any $F$-algebra $A$. The map
$\psi_{m,F} \colon {\W}_{m,F} \to \A^m_F$, given by $\psi_{m,F}((a_i)) = (a_i)$
for $(a_i) \in \W_m(A)$ ($A$ any $F$-algebra) is an isomorphism between $F$-schemes
(e.g., see \cite[Chap.~V, \S~3.13]{Serre-AGCF}).
Note however that this is not an isomorphism of group schemes over $F$.
One can in fact write ${\W}_{m,F} = \Spec(A_{m,F})$, where
\begin{equation}\label{eqn:Witt-3}
  A_{m,F} = \frac{F[X_{ij}| 1 \le i, j \le m+1]}{(X_{ij} \ \mbox{for} \ i > j,
    \ X_{ij} -1  \ \mbox{for} \ i = j, \ X_{ij} - X_{i+1 j+1} \ \mbox{for} \ i < j)}.
\end{equation}
The isomorphism $\psi_{m,F}$ is easily deduced from this description.

\vskip .2cm

Suppose now that $F$ is our local field $k$. For any finite field
extension ${k'}/k$, the group scheme $\W_{m,k'}$ then induces the adic topology on the
abelian group $\W_m(k') = \W_{m,k'}(k')$.
Furthermore, the isomorphism $\W_{m,k'} \cong \A^m_{k'}$ as $k'$-schemes implies
that $\W_m(k') \cong k'^m$ as adic spaces, where the latter has the product of the
adic topology of $k'$. Note here that the adic topology
of $\W_m(k')$ does not depend on whether it has been induced by the $k$-scheme
structure or by the $k'$-scheme structure of $\W_{m,k'}$ (e.g., see
\cite[\S~7.2]{Kato-cft-1}). 
We shall henceforth consider $\W_{m,k'}$ as a unipotent algebraic group over $k'$
and $\W_m(k')$ as a topological abelian group with its adic topology.
Using the descriptions of Kato and adic topologies, we get the following.

\begin{lem}\label{lem:Adic-Witt}
Assume $p > 1$. Let $X$ be a regular curve over $k$ and $\pi_x$ a uniformizer of the
  local ring $\sO_{X,x}$ at a closed point $x \in X$, then the map
  \[
    \mu_{k(x)} \circ \phi^{-1} \colon
    \left(\frac{\wh{\sO_{X,x}}}{(\pi^{m+1}_x)}\right)^\times \to k(x)^\times \times
    \W_{m}(k(x))
    \]
    is a continuous bijection between topological abelian groups for every $m \ge 0$,
    where the left hand side is endowed with the quotient of the Kato topology.
  \end{lem}
  \begin{proof}
We can assume $m \ge 1$, else the statement is immediate from \lemref{lem:Kato-1}.
 Using \lemref{lem:Kato-1} and ~\eqref{eqn:Witt-1},
    we only need to show that
    $\mu_{k(x)} \colon \Gamma_{m}(k(x)) \to \W_{m}(k(x))$ is continuous
    if $\Gamma_{m}(k(x))$ is endowed with the subspace topology from
    $\left({k(x)[T]}/{(T^{m+1})}\right)^\times$. For this, we look at the diagram
    \begin{equation}\label{eqn:Adic-Witt-0}
      \xymatrix@C1pc{
      \Gamma_{m}(k(x)) \ar[r]^-{\mu_{k(x)}} \ar[d]_-{\theta_{k(x)}} & \W_{m}(k(x))  
      \ar[d]^-{\psi_{k(x)}} \\
       k(x)^m \ar[r]^-{\gamma_{k(x)}} & k(x)^m,}
      \end{equation}
      where $\gamma_{k(x)}((a_i)) = (p_1((a_i)), \ldots , p_m((a_i)))$
      (cf. ~\eqref{eqn:Witt-2}).
The arrow $\theta_{k(x)}$ is the canonical bijection induced by the homeomorphism of
      topological spaces 
      $\alpha_{k(x)} \colon (k(x)[[T]])^\times \xrightarrow{\cong} k(x)^\times \times
      k(x)^{\N}$ as in \lemref{lem:Kato-1}.  It is clear from the definition of various
      maps that this diagram is commutative. Since $\psi_{k(x)}$ is a homeomorphism,
      it is enough to show that $\psi_{k(x)} \circ \mu_{k(x)}$ is continuous.
      Equivalently, we need to show that $\gamma_{k(x)} \circ \theta_{k(x)}$ is
      continuous. 

Since $\alpha_{k(x)}$ is a homeomorphism of topological spaces
      (cf. \lemref{lem:Kato-1}) and $\theta_{k(x)}$ is the induced map
      on the quotients with the quotient topologies (note that the projection
      $k(x)^{\N} \to k(x)^m$ is a quotient map), it is a
      homeomorphism. Since each component of $\gamma_{k(x)}$ is a polynomial map,
      it is clearly continuous. In particular, $\gamma_{k(x)}$ is continuous.
      It follows that $\gamma_{k(x)} \circ \theta_{k(x)}$ is continuous.
\end{proof}

\subsection{Pontryagin duals of Chow groups}
  \label{sec:Pont-mod*}
Let $k$ be a local field of exponential characteristic $p \ge 1$.
Let $(X,D)$ be a modulus pair over $k$, where $X$ is a geometrically
connected smooth projective curve over $k$. We let $D^\dagger = D_\red$.
We shall assume that $\emptyset \neq D^\dagger \subset X(k)$. We let
$X^o = X \setminus D$.
We let $j \colon X^o \inj X$ and $t \colon D \inj X$ be the inclusions. We write
$D = \stackrel{r}{\underset{i =1}\sum}n_i [x_i] \in \Div(X)$.
If $U \subset X$ is any Zariski dense open, then $U(k) \subset X(k)$ is an adically
dense open subset (e.g., see \cite[Thm.~10.5.1]{CTS}). Since $X(k) \neq \emptyset$,
it easily follows that $U(k)$ must be infinite. In particular, $X^o(k)$ is infinite.
We fix a point $P \in X^o(k)$ such that $\alb_{X|D}(P) = 0$.
We have noted in \S~\ref{sec:Pic-loc} that $\CH_0(X|D)$ and $\CH_0(X)$ are adic
topological abelian groups and the degree zero parts are their open subgroups.

Assume that $p > 1$. Using the identification $\sO^\times(D) \cong
{\underset{x \in D^\dagger}\prod} \left(\frac{\wh{\sO_{X,x}}}{(\pi^{n_x}_x)}\right)^\times$
as an adic space, we conclude from \propref{prop:Pic-sequence} and
\lemref{lem:Adic-Witt} that in ~\eqref{eqn:CCM-0},
the map $\partial_{X|D}$ is continuous, for, it is the product of maps
$\mu_{k(x)} \circ \phi^{-1}_{x}$ over $x \in D^\dagger$ followed by the
topological quotient $(k^\times)^r \times (\stackrel{r}{\underset{i =1}\prod}
\W_{n_i-1}(k)) \surj (k^\times)^{r-1} \times (\stackrel{r}{\underset{i =1}\prod}
\W_{n_i-1}(k))$. Since the map $\vartheta_{X|D}$
and  the inclusion $k^\times \inj \sO^\times(D)$ are clearly continuous,
~\eqref{eqn:CCM-0} gives rise to a sequence of homomorphisms
\begin{equation}\label{eqn:Chow-Parshin-4*-0}
  0 \to \CH_0(X)^\star \xrightarrow{(\vartheta_{X|D})^\star} \CH_0(X|D)^\star
  \xrightarrow{(\partial_{X|D})^\star} (\sO^\times(D))^\star \to (k^\times)^\star
  \end{equation}

\begin{lem}\label{lem:Main-exact-00}
    ~\eqref{eqn:Chow-Parshin-4*-0} is a chain complex of abelian groups which is exact
    at $\CH_0(X)^\star$ and $\CH_0(X|D)^\star$.
  \end{lem}
  \begin{proof}
The lemma is obvious if $\deg(D) = 1$ by \lemref{lem:deg-1}. We shall therefore
    assume that $\deg(D) \ge 2$.
That ~\eqref{eqn:Chow-Parshin-4*-0} is a complex is a direct consequence of
 ~\eqref{eqn:CCM-0}. The injectivity of
    $(\vartheta_{X|D})^\star$ follows again from ~\eqref{eqn:CCM-0}.
To show the exactness at $\CH_0(X|D)^\star$, we can clearly replace
$\CH_0(X|D)$ and $\CH_0(X)$ by their degree zero subgroups. 
We let $G = (k^\times)^{r-1} \times
(\stackrel{r}{\underset{i =1}\prod} \W_{n_i -1}(k))$. We need to show that the sequence
    \begin{equation}\label{eqn:Main-exact-00-1}
0 \to (\CH_0(X)_0)^\star \xrightarrow{(\vartheta_{X|D})^\star} (\CH_0(X|D)_0)^\star
\to G^\star \to 0
\end{equation}
is exact at the middle term.

To show the above exactness, we note that the sequence
\begin{equation}\label{eqn:Main-exact-00-2}
0 \to (\CH_0(X)_0)^\vee \xrightarrow{(\vartheta_{X|D})^\vee} (\CH_0(X|D)_0)^\vee
\to G^\vee \to 0
\end{equation}
is clearly exact. Hence, given
a continuous character $\chi \in (\CH_0(X|D)_0)^\star$ whose restriction to
$G$ is zero, we get a unique
character $\chi' \in (\CH_0(X)_0)^\vee$ such that $\chi = \chi' \circ \vartheta_{X|D}$.
It remains to show that $\chi'$ is continuous with respect to the adic topology of
$\CH_0(X)_0$. But this is an easy consequence of \lemref{lem:CCM-iso} and
\corref{cor:Pic-sequence-2}.
\end{proof}

\vskip .3cm

We now assume $p = 1$ and prove the analogous result.
We begin with the following.
\begin{lem}\label{lem:Char-0-dual}
  If $p =1$, we have $\CH_0(X|D^\dagger)^\star \cong (\CH_0(X|D)^\star)_\tor \cong
  (\CH_0(X|D)^\vee)_\tor$. We also have
  \begin{equation}\label{eqn:Char-0-dual-1}
    ((\sO^\times(D))^\star)_\tor \cong (\sO^\times(D))^\vee)_\tor \cong
    (\sO^\times(D^\dagger))^\star \cong (\sO^\times(D^\dagger))^\vee)_\tor,
  \end{equation}
  where $\sO^\times(D)$ and $\sO^\times(D^\dagger)$ are endowed with the adic
  topology.
\end{lem}
\begin{proof}
The isomorphisms in ~\eqref{eqn:Char-0-dual-1} are
immediate consequences of \cite[Prop.~II.5.7]{Neukirch} because
$\Ker(\sO^\times(D) \surj \sO^\times(D^\dagger))$ is a divisible group.
To prove the first statement, we first assume $D = \emptyset$.
Since $\CH_0(X)_0 = \Picc^0(X)(k)$ is a profinite abelian group by
\corref{cor:Pic-sequence-2},  it follows that $(\CH_0(X)_0)^\star$ is a torsion group.
Furthermore, the assumption $p =1$ implies using the Kummer sequence that ${\CH_0(X)}/n$
is discrete (and finite) for the quotient topology for every integer $n \ge 1$.
Since $\Picc^0(X) \xrightarrow{n} \Picc^0(X)$
is a smooth isogeny, it follows from \cite[Thm.~10.5.1]{CTS} that
$\CH_0(X) \xrightarrow{n} \CH_0(X)$ is an open map of adic spaces.
In particular, ${\CH_0(X)} \to {\CH_0(X)}/n$ is continuous.

Now, if $\chi \in \CH_0(X)^\star$, then $\chi(\CH_0(X)_0)$ is a finite cyclic group
of the type ${\Z}/n$. This yields a commutative diagram of exact sequences
\begin{equation}\label{eqn:Char-0-dual-0}
  \xymatrix@C1pc{
    0 \ar[r] & \CH_0(X)_0 \ar[r] \ar[d]_-{\chi} & \CH_0(X) \ar[r] \ar[d]_-{\chi} & 
    \Z \ar[r] \ar[d]^-{\wt{\chi}} & 0 \\
    0 \ar[r] & {\Z}/n \ar[r] & {\Q}/{\Z} \ar[r]^-{n} & {\Q}/{\Z} \ar[r] & 0.}
  \end{equation}
  Since the image of $\wt{\chi}$ must be finite, it follows that $\chi$ has a finite
  order. Conversely, if $\chi \in  (\CH_0(X)^\vee)_\tor$, then it factors through
  ${\CH_0(X)}/n$ for some $n \ge 1$ and it follows from 
  the previous paragraph that $\chi$ is continuous on $\CH_0(X)$.

Suppose now that $D$ is nonempty but reduced and let $\chi \in \CH_0(X|D)^\star$.
We have seen above that the restriction of $\chi$ to
  $ (k^\times)^{r-1}$ has finite order. We thus get an integer $n \ge 1$ and
  a commutative diagram of exact sequences (cf. \propref{prop:Pic-sequence})
  \begin{equation}\label{eqn:Char-0-dual-2}
    \xymatrix@C1pc{
0 \ar[r] & (k^\times)^{r-1} \ar[r] \ar[d]_-{\chi} & \CH_0(X|D) \ar[r] \ar[d]_-{\chi} & 
    \CH_0(X) \ar[r] \ar[d]^-{\wt{\chi}} & 0 \\
    0 \ar[r] & {\Z}/n \ar[r] & {\Q}/{\Z} \ar[r]^-{n} & {\Q}/{\Z} \ar[r] & 0.}
\end{equation}
Since $\CH_0(X|D) \surj \CH_0(X)$ is a topological quotient by
\corref{cor:Pic-sequence-2}, we get that $\wt{\chi}$ is continuous, and hence,
its image is finite. We conclude that the image of the middle vertical arrow is finite. 

Conversely, if $\chi \in (\CH_0(X|D)^\vee)_\tor$, then it factors through
${\CH_0(X|D)}/n \xrightarrow{\wt{\chi}} {\Q}/{\Z}$ for some $n \ge 1$.
We have shown above that
${\CH_0(X)}/n$ is finite. It follows from the top row of ~\eqref{eqn:Char-0-dual-2} and
\cite[Prop.~II.5.7]{Neukirch} that ${\CH_0(X|D)}/n$ is finite. On the other hand,
as $\Picc^0(X|D) \xrightarrow{n} \Picc^0(X|D)$ is a smooth isogeny, it follows that
${\CH_0(X|D)}/n$ is a finite and discrete quotient of $\CH_0(X|D)$. We conclude that
$\chi$ is continuous.

If $D$ is not necessarily reduced, we look at the commutative diagram
\begin{equation}\label{eqn:Char-0-dual-3}
    \xymatrix@C1pc{
      \CH_0(X|D^\dagger)^\star \ar[r] \ar[d] &  (\CH_0(X|D^\dagger)^\vee)_\tor \ar[d] \\
      (\CH_0(X|D)^\star)_\tor \ar[r] &  (\CH_0(X|D)^\vee)_\tor,}
  \end{equation}
  in which all arrows are the canonical inclusions. Since any element of
  $(\CH_0(X|D)^\vee)_\tor$ must annihilate the kernel of 
  $\CH_0(X|D) \surj \CH_0(X|D_\red)$, it follows
  that the right vertical arrow in ~\eqref{eqn:Char-0-dual-3} is a bijection.
  We have shown in the previous paragraph that the top horizontal arrow is a
  bijection. A diagram chase shows that all arrows are bijections.
This concludes the proof.
\end{proof}

\begin{lem}\label{lem:Main-exact-003}
  If $p = 1$ and $D$ is reduced, then ~\eqref{eqn:CCM-0} induces a chain complex
  of abelian groups
  \begin{equation}\label{eqn:Main-exact-004}
0 \to \CH_0(X)^\star \xrightarrow{(\vartheta_{X|D})^\star} \CH_0(X|D)^\star
  \xrightarrow{(\partial_{X|D})^\star} (\sO^\times(D))^\star \to (k^\times)^\star
\end{equation}
which is exact at $\CH_0(X)^\star$ and $\CH_0(X|D)^\star$. 
\end{lem}
\begin{proof}
  At any rate, we do have an exact sequence
 \[
 0 \to (\CH_0(X)^\vee)_\tor \xrightarrow{(\vartheta_{X|D})^\vee} (\CH_0(X|D)^\vee)_\tor
  \xrightarrow{(\partial_{X|D})^\vee} ((\sO^\times(D))^\vee)_\tor \to ((k^\times)^\vee)_\tor.
\]
The desired claim now follows by \lemref{lem:Char-0-dual}.
\end{proof}

\section{Brauer group with modulus}\label{sec:BM}
The goal of this section is to define the Brauer group of a modulus pair and
prove some functorial properties. We begin by recalling Kato's ramification
filtration which will play a fundamental role in our exposition.

\subsection{{\'E}tale motivic cohomology}\label{sec:EMC}
Let $k$ be a field of exponential characteristic $p \ge 1$ and let $X$ be a Noetherian
$k$-scheme.
If $n \in k^\times$ is an integer and $r \in \Z$, we let ${\Z}/n(r)$ be the
{\'e}tale sheaf on $X$ defined as the usual Tate twist of the constant sheaf ${\Z}/n$
(e.g., see \cite[p.~163]{Milne-etale}). If $p > 1$ and $n = p^sm$ with $s \ge 0$
and $p \nmid m$, we let ${\Z}/n(r)$ be the object
${\Z}/m(r) \oplus W_s\Omega^r_{X, \log}[-r]$ as an object of $\sD_\et(X)$.
We have the cup product pairing of the (hyper)cohomology of the form
$H^i(X, {\Z}/{n}(j)) \times H^{i'}(X, {\Z}/{n}(j'))
\to H^{i+i'}(X, {\Z}/{n}(j+j'))$.
For $q \in \Z$, we let $H^q_n(X)$ denote the {\'e}tale cohomology group
$H^q(X, {\Z}/{n}(q-1))$. We let $H^q(X) = {\varinjlim}_n \ H^q_n(X)$
with respect to the canonical transition maps ${\Z}/{n}(r) \xrightarrow{m} {\Z}/{mn}(r)$
(see ~\eqref{eqn:Zhao*-0} for their definition in positive characteristic).
If $X = \Spec(A)$ is affine, we write $H^q_n(X)$ (resp. $H^q(X)$) as $H^q_n(A)$
(resp. $H^q(A)$).

For $m, n \in k^\times$ and
$s \ge 1$, we have commutative diagrams of exact sequences
\begin{equation}\label{eqn:Mot-coh-0}
  \xymatrix@C.8pc{
    0 \ar[r] &  {{\Z}/{n}}(1) \ar[r] \ar[d]_-{\can} &
    \sO^\times_X \ar[r]^-{n} \ar[d]^-{\id} &  \sO^\times_X \ar[r] \ar[d]^-{m} & 0 & &
    0 \ar[r] & \sO^\times_X \ar[r]^-{p^s} \ar[d]_-{\id} & \sO^\times_X \ar[r] \ar[d]^-{p}
    & {\sO^\times_X}/{p^s} \ar[r] \ar[d]^-{p} & 0 \\
    0 \ar[r] &  {{\Z}/{mn}}(1) \ar[r] & \sO^\times_X \ar[r]^-{mn} &  \sO^\times_X
    \ar[r] & 0 & &
    0 \ar[r] & \sO^\times_X \ar[r]^-{p^{s+1}}  & \sO^\times_X \ar[r] 
    & {\sO^\times_X}/{p^{s+1}} \ar[r] & 0,}
\end{equation}
\begin{equation}\label{eqn:Mot-coh-1}
\xymatrix@C.8pc{
    0 \ar[r] &  {{\Z}/{mn}}(1) \ar[r] \ar[d]_-{m} &
    \sO^\times_X \ar[r]^-{mn} \ar[d]^-{m} &  \sO^\times_X \ar[r] \ar[d]^-{\id} & 0 & &
    0 \ar[r] & \sO^\times_X \ar[r]^-{p^{s+1}} \ar[d]_-{p} & \sO^\times_X \ar[r] \ar[d]^-{\id}
    & {\sO^\times_X}/{p^{s+1}} \ar[r] \ar[d]^-{\can} & 0 \\
    0 \ar[r] &  {{\Z}/{n}}(1) \ar[r] & \sO^\times_X \ar[r]^-{n} &  \sO^\times_X
    \ar[r] & 0 & &
    0 \ar[r] & \sO^\times_X \ar[r]^-{p^{s}}  & \sO^\times_X \ar[r] 
    & {\sO^\times_X}/{p^{s}} \ar[r] & 0}
\end{equation}
 of {\'e}tale sheaves, where the diagrams on the right make sense when $p > 1$.

The commutative diagrams in ~\eqref{eqn:Mot-coh-0}
give rise to an exact sequence of ind-abelian groups
\begin{equation}\label{eqn:Mot-coh-2}
  0 \to \{{\Br(X)}/n\} \to \{H^3(X, {\Z}/m(1))\} \to \{_mH^3(X, \sO^\times_X)\}
  \to 0,
\end{equation}
which are indexed by $\N$ and whose transition maps
are induced by ${\Z}/m \xrightarrow{n} {\Z}/{mn}$.
Taking the limits and noting that $H^i(X, \sO^\times_X)$ is a torsion group for
$i \ge 2$ (e.g., see \cite[Lem.~3.5.3]{CTS}) when $X$ is regular, we get the following.
\begin{lem}\label{lem:Mot-coh-3}
  If $X$ is regular, then the canonical map
  \[
    H^3(X, {\Q}/{\Z}(1)) \to H^3(X, \sO^\times_X)
  \]
  is an isomorphism.
\end{lem}
\begin{proof}
  This is an easy application of the above discussion once we use an elementary fact that
  $A \otimes B = 0$ if $A$ is a torsion abelian group and $B$ is a divisible abelian
  group.
\end{proof}

We also get the following folklore result which we shall use throughout this paper
without giving reference.

\begin{lem}\label{lem:Brauer-H2}
  Let $A$ be a equicharacteristic regular local ring. Then one has a canonical
  isomorphism $H^2(A) \xrightarrow{\cong} \Br(A)$.
\end{lem}

For a local ring $A$ over $k$, we have the
Norm residue map (e.g., see \cite[\S~5.1]{Gupta-Krishna-REC})
\begin{equation}\label{eqn:NR}
  {\rm NR}_A \colon {K^M_{i}(A)}/{n} \to H^i(A, {\Z}/n(i)).
\end{equation}
Composing this with the cup product, we see that there is a canonical
bilinear pairing
\begin{equation}\label{eqn:NR-0}
  \<,\> \colon H^i(A, {\Z}/{n}(j)) \times {K^M_{i'}(A)}/{n} \to
  H^{i+i'}(A, {\Z}/{n}(j+i')).
  \end{equation}
 
  If $A$ is any equicharacteristic local integral domain with  maximal ideal $\fm_A$
  and quotient field $E$, we let $\Fil_0 {K}^M_1(E) = A^\times$ and $\Fil_n {K}^M_1(E) =
(1 + \fm^n_A)$ if $n \ge 1$.
For $i \ge 2$, we let $\Fil_n {K}^M_i(E)$ be the image of the
cup product map $\Fil_n {K}^M_1(E) \otimes {K}^M_{i-1}(E) \to {K}^M_{i}(E)$.
We let $\Fil_n {K}^M_i(E) = {K}^M_i(E)$ for $n < 0$. 
It is clear that $\Fil_\bullet {K}^M_i(E)$ is a decreasing filtration of
${K}^M_i(E)$. We shall call this `the logarithmic filtration' of ${K}^M_i(E)$.

\subsection{Kato's ramification filtration}\label{sec:RFil}
In order to define the Brauer group with modulus, we need to recall Kato's ramification
filtration. We let $L$ be a Henselian discrete valuation field with the ring of
integers $\sO_L$, the maximal ideal $\fm_L = (\pi_L)$ and the residue field $\fl$
such that $\Char(L) = \Char(\fl) = p \ge 0$. Let $\wh{L}$ denote the completion of $L$. 
 Recall the following from \cite[Cor.~2.5, Prop.~6.3]{Kato-89}, where we have shifted
 Kato's filtration one place to the right.
\begin{defn}\label{defn:Kato-filt}
  Let $q \ge 1$ be an integer.
  \begin{enumerate}
\item
  If $p = 0$, we let $\Fil_0 H^q(L) = H^q(\sO_L)$ and $\Fil_n H^q(L) =
  H^q(L)$ if $n \ge 1$.
\item
  If $p > 0$ and $n \ge 0$, we let $\Fil_n H^q(L)$ be the
  subgroup of elements $\chi \in H^q(L)$ such that
  $\<\chi, 1 + \pi^n_L\sO_{L'}\> = 0$
  for all Henselian discrete valuation 
  fields $L'$ such that $\sO_L \subset \sO_{L'}$
  and $\fm_{L'} = \fm_L \sO_{L'}$, where $1 + \pi^0_L\sO_{L'} := \sO^\times_{L'}$.
\item
  We let $\Fil_n H^q(L) = 0$ for $n < 0$.
\end{enumerate}
\end{defn}
It follows from \cite[Lem.~2.2]{Kato-89} that
\begin{equation}\label{eqn:Kato-filt-0}
  H^q(L) = {\underset{n \ge 0}\bigcup} \Fil_n H^q(L).
  \end{equation}
We shall call $\Fil_\bullet H^q(L) := \{\Fil_n H^q(L)\}_{n \in \Z}$,
the ramification filtration of $H^q(L)$.
Recall that for any $\chi \in H^q(L)$, the Swan conductor $\Sw(\chi)$ is the smallest
integer $n$ such that $\chi \in \Fil_n H^q(L)$.
The following result shows that one can characterize $\Fil_\bullet H^q(L)$ purely in terms
of characters of the Milnor $K$-groups of $L$ in cases of our interest.

\begin{lem}\label{lem:Hensel-complete}
  Assume that $p > 0$ and $\fl$ is a local field. Let $1 \le q \le 2$ and
  $n \ge 0$ be integers. Then the following hold.
    \begin{enumerate}
    \item
      There are canonical isomorphisms
      \[
        H^3(L) \xrightarrow{\cong} H^3(\wh{L}) \xrightarrow{\cong} H^2(\fl)
        \xrightarrow{\cong} {\Q}/{\Z}.
      \]
      \item
      The canonical map $H^q(\sO_L) \to H^q(L)$ induced by the inclusion
      $\sO_L \inj L$ fits into a split short exact sequence
      \[
        0 \to H^q(\sO_L) \to \Fil_1 H^q(L) \to H^{q-1}(\fl) \to 0,
      \]
     which is canonical for a given choice of $\pi_L$.
      \item
        $\Fil_0 H^q(L) = H^q(\sO_L)$.
      \item
        For $n \ge 1$, an element
    $\chi \in H^q(L)$ lies in $\Fil_n H^q(L)$ if and only if
    $\<\chi, \Fil_n {K}^M_{3-q}(L)\> = 0$ under
    the pairing
    \[
      H^q(L) \times {K}^M_{3-q}(L) \xrightarrow{\<,\>}  H^3(L) \xrightarrow{\cong}
      {\Q}/{\Z}.
      \]
    \end{enumerate}
  \end{lem}
  \begin{proof}
Part (1) of the lemma is a special case of the general
    isomorphism ($\forall \ q$)
    \begin{equation}\label{eqn:Hensel-complete-4}
      H^q(L) \xrightarrow{\cong} H^q(\wh{L})
    \end{equation}
    as shown in \cite[Lem.~21]{Kato-Invitation}, and the well known case of
    complete discrete valuation fields (see \cite[\S~3.2, Prop.~1]{Kato-80}).
    For $q = 1$, the other parts of the lemma follow from
    \cite[Thm.~6.3]{Gupta-Krishna-REC}. For $q = 2$, part (2)
    follows directly from \cite[Prop.~6.1]{Kato-89}.

To prove (3), 
    suppose first that $\chi \in \Fil_0 H^2(L)$. We consider the commutative diagram
    of exact sequences
    \begin{equation}\label{eqn:Hensel-complete-0}
\xymatrix@C.8pc{
  0 \ar[r] & H^2(\sO_L) \ar[r] \ar[d] &  \Fil_1 H^2(L) \ar[r]^-{\gamma} \ar[d]^-{\alpha} &
  H^{1}(\fl) \ar[r] \ar[d] & 0 \\
  0 \ar[r] & (\Z)^\vee \ar[r] & (L^\times)^\vee \ar[r]^-{\beta} &
  (\sO^\times_L)^\vee \ar[r] & 0,}
\end{equation}
where the vertical arrows are induced by ~\eqref{eqn:NR-0}.
Our hypothesis says that $\beta \circ \alpha (\chi) = 0$. Since the right vertical
arrow is injective by the class field theory of local fields, it follows that
$\chi \in H^2(\sO_L)$. Conversely, suppose that $\chi \in H^2(\sO_L)$.
Then a diagram chase of ~\eqref{eqn:Hensel-complete-0} tells us that
$\chi \in  \Fil_1 H^2(L)$ such that $\beta \circ \alpha (\chi) = 0$.
But this is equivalent to saying that $\chi \in \Fil_0 H^2(L)$.
This proves (3).

We now assume $n \ge 1$ and $q = 2$.
In view of \cite[Prop.~6.3]{Kato-89}, we only need to show
      that if $\chi \in H^2(L)$ is a character such that $\<\chi, \Fil_n L^\times\> = 0$,
    then it lies in $\Fil_n H^2(L)$. At any rate, we know that $\chi \in
    \Fil_m H^2(L)$ for some $m \gg 0$. We can assume that $m > n$ (else we are done).
    Our hypothesis implies that as a character of $L^\times$, $\chi$ factors through
    ${L^\times}/{ \Fil_n L^\times}$.

We now note that the ring of integers
    $\sO_{\wh{L}}$ of $\wh{L}$ is the $\fm_L$-adic completion $\wh{\sO_L}$ of $\sO_{L}$.
    Since ${\sO^\times_L}/{\Fil_n L^\times}  \xrightarrow{\cong} 
    {\sO^\times_{\wh{L}}}/{\Fil_n {\wh{L}^\times}}$, it follows from the commutative diagram
    \begin{equation}\label{eqn:Hensel-complete-1}
      \xymatrix@C.8pc{
        0 \ar[r] & \sO^\times_L \ar[r] \ar[d] & L^\times \ar[r]^-{v_L} \ar[d] & \Z
        \ar@{=}[d] \ar[r] & 0 \\
        0 \ar[r] & \sO^\times_{\wh{L}} \ar[r] & {\wh{L}}^\times \ar[r]^-{v_{\wh{L}}} &
        \Z \ar[r] & 0}
      \end{equation}    
of exact sequences that ${L^\times}/{\Fil_n L^\times} \xrightarrow{\cong} 
      {{\wh{L}^\times}}/{\Fil_n {\wh{L}^\times}}$. Hence, $\chi$ factors through
        ${{\wh{L}^\times}}/{\Fil_n {\wh{L}^\times}}$ as a character of $\wh{L}^\times$.
        We conclude from \cite[Prop.~6.3, Rem.~6.6]{Kato-89} that
        $\chi \in \Fil_n H^2(\wh{L})$.

By an iteration of the commutative diagram
          \begin{equation}\label{eqn:Hensel-complete-2}
      \xymatrix@C.8pc{
        0 \ar[r] & \Fil_{m-1}H^2(L) \ar[r] \ar[d] & \Fil_{m}H^2(L) \ar[d] \ar[r] &
        \frac{\Fil_{m}H^2(L)}{\Fil_{m-1}H^2(L)} \ar[d] \ar[r] & 0 \\
         0 \ar[r] & \Fil_{m-1}H^2(\wh{L}) \ar[r] & \Fil_{m}H^2(\wh{L}) \ar[r] &
        \frac{\Fil_{m}H^2(\wh{L})}{\Fil_{m-1}H^2(\wh{L})} \ar[r] & 0,}
    \end{equation}
  it remains to show that the right vertical arrow in this
    diagram is injective.

We now look at the diagram
     \begin{equation}\label{eqn:Hensel-complete-3}
      \xymatrix@C1pc{ 
        \frac{\Fil_{m}H^2(L)}{\Fil_{m-1}H^2(L)} \ar[r]^-{{\rm rsw}_L} \ar[d] &
        \Omega^2_{\fl} \oplus \Omega^1_{\fl} \ar[d]^-{\cong} \\
         \frac{\Fil_{m}H^2(\wh{L})}{\Fil_{m-1}H^2(\wh{L})} \ar[r]^-{{\rm rsw}_{\wh{L}}}
& \Omega^2_{\fl} \oplus \Omega^1_{\fl}.}
\end{equation}
Since $\pi_L$ is also a uniformizer of $\sO_{\wh{L}}$, it follows from
\cite[Thm.~0.1]{Kato-89} that this diagram is commutative and the
horizontal arrows are injective. It follows that the left vertical arrow is injective.
This concludes the proof.
    \end{proof}

We shall use the following result of Kato to prove various functorial properties
of the Brauer and Picard groups of modulus pairs.

\begin{lem}\label{lem:Kato-PB-PF}
  Let ${L'}/L$ be a finite extension of Henselian discrete valuation fields of
  characteristic $p > 0$. Let $e$ denote the ramification index of ${L'}/L$.
  Let $\fl$ (resp. $\fl'$)
  denote the residue field of $L$ (resp. $L'$). Assume that $\fl$ and $\fl'$ are local
  fields. Then we have the following.
  \begin{enumerate}
  \item
 There is a commutative  diagram
  \begin{equation}\label{eqn:Kato-PB-PF-0}
    \xymatrix@C1pc{
 H^{2}(L') \times L'^\times \ar[r]^-{\cup}
 \ar@<-4ex>[d]_{\Cores} &  H^{3}(L') \ar[r]^-{\partial_{L'}} \ar[d]^-{\Cores} &
 H^2(\fl') \ar[d]^-{\Cores} \ar[dr]^-{\inv_{\fl'}} & \\
H^{2}(L) \times L^\times \ar@<-4ex>[u]_-{\Ress}
\ar[r]^-{\cup} &  H^{3}(L) \ar[r]^-{\partial_L} & H^2(\fl) \ar[r]^-{\inv_{\fl}}  &
  {\Q}/{\Z}.}
\end{equation}
The same holds if we interchange $\Cores$ and $\Ress$ on the left side of the cup
product maps.
\item
  $\Cores(\Fil_{en} H^2(L')) \subseteq \Fil_n H^2(L)$ for any $n \ge 0$.
  \item
$\Ress(\Fil_n H^2(L)) \subseteq \Fil_{en} H^2(L')$ for any $n \ge 0$.
\end{enumerate}
\end{lem}
\begin{proof}
(1) follows from \cite[\S~3.2, Lem.~1, Prop.~1(2)]{Kato-80}, 
(2) follows from (1) using that $\Ress(\Fil_n L^\times) \subset \Fil_{en} L'^\times$
by definition, and (3) follows from (1) and \cite[Lem.~6.19]{Raskind}.
\end{proof}

\subsection{The Brauer group with modulus}\label{sec:BGM*}
Let $k$ be a local field with ring of integers $\sO_k$,  maximal ideal
$\fm = (\pi)$ and residue field $\ff = {\sO_k}/{\fm} \cong \F_q$, where
$q = p^n$ for some prime $p \ge 2$ and integer $n \ge 1$.
Let $X$ be a connected regular quasi-projective scheme over $k$ of dimension
$d \ge 1$. Let $D \subset X$ be an effective divisor (possibly empty) and
$D^\dagger = D_\red$.
We write $D = \sum_{x \in X^{(1)}} n_x \ov{\{x\}}$. Clearly, $n_x = 0$ unless $x$ is a
generic point of $D^\dagger$, and $D = 0$ if $D^\dagger$ is empty. We let
$\Div_{D^\dagger}(X)$ be the set of all effective
divisors on $X$ whose support is $D^\dagger$. This is a filtered set under
inclusion. We let $\Div(X)$ denote the filtered set of all effective Cartier
(equivalently Weil) divisors on $X$.
We let $X^o = X \setminus D$. We let $j \colon X^o \inj X$ and $t \colon D \inj X$ be
the inclusions. We let $K$ denote the function field of $X$.

For any $x \in X^{(1)}$, we let $K_x$ denote the quotient field of 
$\sO^h_{X,x}$ and $\wh{K}_x$ the quotient field of $\wh{\sO_{X,x}}$. 
We shall say that an element $\chi \in H^i(K)$ is unramified at a point $x\in X^{(1)}$ if
the image of $\chi$ under the canonical map $H^i(K) \to H^i(K_x)$ lies in the image of
$H^i(\sO^h_{X,x})$. We shall say that $\chi$ is unramified on an open subscheme $U
\subseteq X$ if it lies in the image of the canonical map $H^i(U) \to H^i(K)$.

\begin{lem} \label{lem:Fil-D-U}
  Let $U$ be a connected Noetherian regular scheme with function field $K$.
  Let $\chi \in \Br(K)$ be unramified at all codimension one
  points of $U$. Then $\chi$ is unramified on $U$. 
\end{lem}
\begin{proof}
  This is an easy consequence of the purity theorem for Brauer group, see
  \cite[Thm.~3.7.7]{CTS}.
\end{proof}

\begin{defn}\label{defn:BGM}
  We let $\Br^\divv(X|D)$ denote the subgroup of $\Br(K)$ consisting of elements $\chi$
  such that for every $x \in X^{(1)}$, the image $\chi_{x}$ of
$\chi$ under the canonical map $\Br(K) = H^2(K) \to H^2(K_{x})$ 
lies in $\Fil_{n_{x}} H^2(K_{x})$. 
\end{defn}

\begin{lem}\label{lem:Brauer-reln}
  We have the following relations between the subgroups of $\Br(K)$.
  \begin{enumerate}
    \item
  For every pair of effective Cartier divisors $D \le D'$ on $X$, one has
  \[
    \Br(X) \subseteq \Br^\divv(X|D) \subseteq \Br^\divv(X|D') \subseteq \Br(X^o)
    \subseteq \Br(K).
  \]
\item
  \[
    {\underset{n \ge 1}\varinjlim} \ \Br^\divv(X|nD) \ \xrightarrow{\cong}
    {\underset{D' \in \Div_{D^\dagger}(X)}\varinjlim} \Br^\divv(X|D')
    \xrightarrow{\cong} \Br(X^o).
  \]
\item
  \[
  {\underset{D' \in \Div(X)}\varinjlim} \Br^\divv(X|D')
  \xrightarrow{\cong} \Br(K).
\]
\item
  If $\Char(k) = 0$, then $\Br^\divv(X|D) \xrightarrow{\cong} \Br(X^o)$ for every
  $D \in \Div_{D^\dagger}(X)$.
  \end{enumerate}
\end{lem}
\begin{proof}
Part (1) follows easily from Lemmas~\ref{lem:Hensel-complete} and
~\ref{lem:Fil-D-U}. To prove (2), let $w \in \Br(X^o)$. For every
irreducible component $D^\dagger_i$ of $D^\dagger$, ~\eqref{eqn:Kato-filt-0} says that
the image of $w$ in $\Br(K_{x_i})$ lies in $\Fil_{m_i} H^2(K_{x_i})$ for some $m_i \ge 0$,
where $D^\dagger_i = \ov{\{x_i\}}$. We now take $D = \sum_i m_i D^\dagger_i \in  
\Div_{D^\dagger}(X)$. Then it is clear that $w \in \Br^\divv(X|D)$.
Part (3) follows from (2) and \cite[Lem.~III.1.16]{Milne-etale}, and (4) follows
from the definition of $\Br^\divv(X|D)$.
\end{proof}

We shall now define the Brauer group of a modulus pair.
We let $\sC(X)$ denote the set of integral curves
on $X$ and let $\sC(X|D)$ denote the subset of $\sC(X)$ consisting of those curves
which are not contained in $D$. For any $C \in \sC(X)$, we let $\nu \colon C_n \to
X$ denote the canonical map from the normalization of $C$ and let $\nu^*(D)$
denote the scheme theoretic Pull-back of $D$. We let $C^o_n = \nu^{-1}(X^o)$.

\begin{defn}\label{defn:BGM-0}
 We let $\Br(X|D)$ denote the subgroup of $\Br(X^o)$ consisting of elements $\chi$
 such that for every $C \in \sC(X|D)$, the Brauer class $\nu^*(\chi)
 \in \Br(C^o_n)$ lies in the subgroup $\Br^\divv(C_n|\nu^*(D))$.
The group $\Br(X|D)$ will be called the Brauer group of the modulus pair $(X,D)$.
\end{defn}

The Brauer group of modulus pairs has the following functorial properties first of
which is not clear for $\Br^\divv(X|D)$.
Recall that an admissible (resp. coadmissible) morphism of modulus pairs
$f \colon (X', D') \to (X,D)$ is a morphism of schemes $f \colon X' \to X$ such that
$D' \le f^*(D)$ (resp. $D' \ge f^*(D)$) as Cartier divisors on $X'$.
One says that $f$ is strict if $D' = f^*(D)$.

\begin{prop}\label{prop:BGM-1}
  Let $f \colon (X', D') \to (X,D)$ be a coadmissible morphism of modulus pairs over
  $k$. Then the pull-back on {\'e}tale cohomology induces a homomorphism
  \[
    f^* \colon \Br(X|D) \to \Br(X'|D')
  \]
  such that $f^*((\Br(X|E)) \subset \Br(X'|E')$ if $E \le D$ and $D' \ge E' \ge f^*(E)$.
  If $g \colon (X'', D'') \to (X', D')$ is another coadmissible morphism of modulus
  pairs, then $(f \circ g)^* = g^* \circ f^*$.
\end{prop}
\begin{proof}
We let $U' = f^{-1}(X^o)$ so that $X'^o \subset U'$. We then have the pull-back map
  $f^* \colon \Br(X^o) \to \Br(U') \subset \Br(X'^o)$. Suppose now that $\chi \in
  \Br(X|D)$ and let $C' \in \sC(X'|D')$. If the image of $C'$ under $f$
  is a closed point, then
  this closed point must lie in $X^o$. In the latter case, it is clear that
  the pull-back of $\chi$ under the composite map $C'_n \to U' \to X^o$ lies in
  $\Br(C'_n) \subset \Br(C'_n|\nu'^*(D')) = \Br^\divv(C'_n|\nu'^*(D'))$.

Otherwise, the scheme theoretic image of $C'$ under $f$ is a curve $C$ which is
  necessarily integral and lies in $\sC(X|D)$. Since $f \colon C' \to C$ is also
  dominant, we get a commutative diagram
  \begin{equation}\label{eqn:BGM-1-0}
    \xymatrix@C1pc{
      C'_n \ar[r]^-{\nu'} \ar[d]_-{g} & X' \ar[d]^-{f} \\
      C_n \ar[r]^-{\nu} & X.}
  \end{equation}
  If we let $E = \nu^*(D)$ and $E' = \nu'^*(D')$, then we get $E' \ge g^*(E)$ by our
  assumption. In other words, $g \colon (C'_n, E') \to (C_n, E)$ is a dominant morphism
  of one-dimensional modulus pairs. We let $L$ (resp. $L'$) be the function field of
  $C$ (resp. $C'$).

We now let $\omega = \nu^*(\chi)$ and let
  $x \in E'$ be a closed point. If $x \notin g^{-1}(E)$, then $g^*(\chi) \in
  \Fil_0 H^2(L'_x)$. If $x \in g^{-1}(E)$, we let $y = g(x)$. We then get an
  inclusion of Henselian discrete valuation fields $L_y \inj L'_x$. It follows from
  Definition~\ref{defn:Kato-filt} that this inclusion induces a map
  $\Br(L_y) \to \Br(L'_x)$ which preserves the ramification filtrations.
  In particular, if $n_y$ (resp. $n_x$) denotes the multiplicity of $E$ (resp. $E'$)
  at $y$ (resp. $x$), then we get $g^*(w_y) \in \Fil_{n_y} \Br(L'_x) \subset
  \Fil_{n_x} \Br(L'_x)$, where the latter inclusion holds  because $E' \ge g^*(E)$.
  It follows that $g^*(\omega) \in \Br(C'_n|E')$. We have thus shown that
  $f^*(\chi) \in \Br(X'|D')$. The inclusion $f^*((\Br(X|E)) \subset \Br(X'|E')$, as well
  as the composition law, is clear from the definition of the pull-back map. This
  concludes the proof.
\end{proof}

\begin{prop}\label{prop:BGM-2}
  Let $f \colon (X', D') \to (X,D)$ be a finite morphism of modulus pairs over $k$.
  Assume that $f$ is
strict. Then the norm map $f_*(\sO^\times_{X'}) \to \sO^\times_X$ induces a homomorphism
  \[
    f_* \colon \Br(X'|D') \to \Br(X|D)
  \]
  such that $f_*((\Br(X'|E')) \subset \Br(X|E)$ if $E \le D$ and $E' = f^*(E)$.
  If $g \colon (X'', D'') \to (X', D')$ is another finite and strict morphism
  of modulus pairs, then $(f \circ g)_* = f_* \circ g_*$. 
\end{prop}
\begin{proof}
Since $f$ is strict, we have $X'^o = f^{-1}(X^o)$ and the norm $f_*(\sO^\times_{X'}) \to
  \sO^\times_X$ induces $f_* \colon \Br(X'^o) \to \Br(X^o)$.
  To show that $f_*$ preserves the ramification filtration, we can assume that
  $\dim(X) = \dim(X') = 1$. We let $K'$ be the function field of $X'$.

 We let $w \in \Br(X'|D')$ and $x \in D^\dagger$.
  By \cite[Prop.~3.8.1, Lem.~3.8.6]{CTS}, there is a commutative diagram
  \begin{equation}\label{eqn:BGM-2-0}
    \xymatrix@C1pc{
      \Br(X'^o) \ar[r] \ar[d]_-{f_*} & {\underset{y \in f^{-1}(x)}\prod} \
      \Br(A'_y) \ar[d]^-{f_* = \Cores} \\
      \Br(X^o) \ar[r] & \Br(K_x),}
  \end{equation}
  where ${\underset{y \in f^{-1}(x)}\prod} A'_y = X'^o \times_{X^o} \Spec(K_x)$
  and the horizontal arrows are the pull-back maps. We fix $y \in f^{-1}(x)$
  and let $w_y$ be the restriction of $w$ to $\Br(A'_y) \cong \Br(K'_y)$. 
It suffices to show that $\Cores_{{K'_y}/{K_x}}(w_y) \in \Fil_{n_x} \Br(K_x)$.
But this follows from \lemref{lem:Kato-PB-PF} because $n_y = n_x e_{y/x}$,
  where $e_{y/x}$ is the ramification index of ${K'_y}/{K_x}$ by the definition of
  $f^*(D)$.
\end{proof}

It is clear that $\Br^\divv(X|D) \cong \Br(X|D)$ when $X$ is a curve. For future study,
it will be of interest to know an answer to the following.

\begin{ques}\label{ques:Div-curve}
  Let $(X,D)$ be a modulus pair over $k$. Is it true that $\Br^\divv(X|D) \cong \Br(X|D)$?
\end{ques}


\section{Brauer-Manin pairing for modulus pairs}
\label{sec:BM-mod**}
In this section, we shall define an extension of the classical Brauer-Manin pairing
between the Picard and Brauer groups of smooth projective curves over local fields to
setting of relative Picard and Brauer groups of modulus pairs.
Throughout this section, we fix a local field $k$ of exponential characteristic
$p \ge 1$ and a connected regular projective scheme $X$ of dimension $d \ge 1$ over $k$.
We also fix a divisor $D \subset X$ and let $D^\dagger = D_\red$.
We let $j \colon X^o \inj X$ and $\iota \colon D \inj X$ be the inclusions,
where $X^o = X \setminus D$. We let $K$ denote the function field of $X$.
All other notations will be those of \S~\ref{sec:BGM*}.

\subsection{Idele class group}\label{sec:CGM}
For any $C \in \sC(X)$, we let $k(C)$ be the function field of
$C$. Given $C \in \sC(X|D)$, we 
let $k(C)_\infty = {\underset{x \in \nu^{-1}(D^\dagger)}\bigoplus} k(C)_x$, where
$\nu \colon C_n \to X$ is the canonical map from the normalization of $C$. 
We let $k(C)^\times_\infty = {\underset{x \in \nu^{-1}(D^\dagger)}\bigoplus} k(C)^\times_x, \
I(C_n|D) = {\underset{x \in \nu^{-1}(D^\dagger)}\bigoplus} (1 + \sI_D\sO^h_{C_n, x})$ and
$\wh{I}(C_n|D) = {\underset{x \in \nu^{-1}(D^\dagger)}\bigoplus}
(1 + \sI_D \wh{\sO_{C_n, x}})$,
where $\sI_D \subset \sO_X$ is the ideal sheaf defining $D$.
The inclusions $k(C) \inj k(C)_x$ induce a canonical map
$\iota_C \colon k(C)^\times \to {k(C)^\times_\infty}/{I(C_n|D)}$.
We also have the map $\partial_{C_n} \colon
k(C)^\times \xrightarrow{\divf} \sZ_0(C_n^o) \xrightarrow{\nu_*} \sZ_0(X^o)$,
where $C^o_n = \nu^{-1}(X^o)$.
We let $\sO_{C_n, \infty} = \sO_{C_n, \nu^{-1}(D^\dagger)}$.

It is easy to check that the left square in the diagram
\begin{equation}\label{eqn:Chow-Parshin}
  \xymatrix@C.8pc{
    {\underset{C \in \sC(X|D)}\bigoplus} {K}^M_1(\sO_{C_n, \infty},
    \sI_D \sO_{C_n, \infty}) \ar[r]^-{\divf} \ar[d] &
    \sZ_0(X^o) \ar[r] \ar[d] & \CH_0(X|D) \ar@{.>}[d] \ar[r] & 0 \\
    {\underset{C \in \sC(X|D)}\bigoplus} k(C)^\times \ar[r]^-{(\divf, \iota_C)} &
    \sZ_0(X^o) \bigoplus \left({\underset{C \in \sC(X|D)}\bigoplus}
      \frac{k(C)^\times_\infty}{I(C_n|D)}\right) \ar[r] &
      \CH_0(X|D)' \ar[r] & 0}
  \end{equation}
  is commutative, where the left vertical arrow is the canonical inclusion,
  the middle vertical arrow is the identity map of $\sZ_0(X^o)$ and
  $\CH_0(X|D)'$ is the cokernel of the left horizontal arrow in
  the bottom row. A straightforward application of the weak approximation theorem
  (e.g., see \cite[Lem.~6.3]{Kerz-MRL}) allows one to conclude that the left square
  induces a natural isomorphism $\CH_0(X|D) \xrightarrow{\cong} \CH_0(X|D)'$.
  In other words, there is a canonical exact sequence
  \begin{equation}\label{eqn:Chow-Parshin-0}
{\underset{C \in \sC(X|D)}\bigoplus} k(C)^\times \xrightarrow{(\divf, \iota_C)} 
    \sZ_0(X^o) \bigoplus \left({\underset{C \in \sC(X|D)}\bigoplus}
      \frac{k(C)^\times_\infty}{I(C_n|D)}\right) \to \CH_0(X|D) \to 0.
    \end{equation}

\begin{remk}\label{remk:BMP-compln}
  We remark that the canonical map $\frac{k(C)^\times_\infty}{I(C_n|D)} \to
\frac{\wh{k(C)}^\times_\infty}{\wh{I}(C_n|D)}$ is an isomorphism for
every $C \in \sC(X|D)$. In particular, the exact sequence ~\eqref{eqn:Chow-Parshin-0}
remains unchanged if we replace $k(C)^\times_\infty$ by $\wh{k(C)}^\times_\infty$
and $I(C_n|D)$ by $\wh{I}(C_n|D)$ for $C \in \sC(X|D)$. This can be easily verified. 
\end{remk}

\begin{defn}\label{defn:Idele}
     The idele class group $C(X^o)$ is the cokernel of the map
      \[
      {\underset{C \in \sC(X|D)}\bigoplus} k(C)^\times \xrightarrow{(\divf, \iota_C)} 
    \sZ_0(X^o) \bigoplus \left({\underset{C \in \sC(X|D)}\bigoplus}
      {\wh{k(C)}^\times_\infty}\right).
  \]
  \end{defn}

One can show using \cite[Prop.~5.3]{GKR} that $C(X^o)$ depends only on
$X^o$ and not on $X$. Using Remark~\ref{remk:BMP-compln},
it is clear that there is a canonical surjective morphism
  of pro-abelian groups $\{C(X^o)\} \surj \{\CH_0(X|D)\}_{D \in \Div_{D^\dagger}(X)}$.
  Taking the limit, we get a homomorphism of abelian groups
  \begin{equation}\label{eqn:Idele-0}
    \theta_{X^o} \colon C(X^o) \to {\underset{D \in \Div_{D^\dagger}(X)}
      \varprojlim} \ \CH_0(X|D).
  \end{equation}
It is well known (e.g., see the proof of \cite[Prop.~5.3]{GKR}) that the sequence
    ~\eqref{eqn:Chow-Parshin-0} and the one in the definition of $C(X^o)$ are
    covariantly functorial for a finite morphism of modulus pairs $f \colon (X',D') \to
    (X,D)$. Hence, they induce push-forward maps $f_* \colon \CH_0(X'|D') \to \CH_0(X|D)$
    and $f_* \colon C(X'^o) \to C(X^o)$.
    Furthermore, ~\eqref{eqn:Idele-0} is compatible with the push-forward maps.

\subsection{The pairing between Brauer and Chow groups with modulus}
\label{sec:Pair**}
Let the notations be as above.
We shall now define a bilinear pairing between the Brauer group and the Chow group
of the modulus pair $(X,D)$. In the following sections, we shall prove several
properties of this pairing which will be key to the proofs of our main results.

Given a closed point $x \in X^o$ with the inclusion
$\iota_x \colon \Spec(k(x)) \inj X^o$,
we have the pull-back map $\iota^*_x \colon \Br(X^o) \to \Br(k(x))$. By composing with
$\inv_{k(x)}$, we get a canonical map ${\iota}^*_x \colon \Br(X^o) \to {\Q}/{\Z}$.
We thus get a bilinear pairing $\Br(X^o) \times \sZ_0(X^o) \to {\Q}/{\Z}$ given by
$\<\omega, \sum_i n_i[x_i]\> = \sum_i n_i {\iota}^*_{x_i}(\omega)$.

Suppose now that $d = 1$ and $x \in D$, where $D = \sum_x n_x [x]$.
Since $K_x$ is a Henselian discrete
valuation field whose residue field is a local field, the cup product of
~\eqref{eqn:NR-0} and \lemref{lem:Hensel-complete}(1) (we note here that
\lemref{lem:Hensel-complete}(1) holds in characteristic zero as well, see
\cite[p.~311]{Kato-80}) together induce a
bilinear pairing $\Br(K_x) \times K^\times_x \to {\Q}/{\Z}$. By composing with
the canonical map $\Br(X^o) \to \Br(K_x)$, this yields a pairing
$\Br(X^o) \times K^\times_x \to {\Q}/{\Z}$. Moreover, it follows from
Definition~\ref{defn:Kato-filt} that the induced pairing
$\Br(X|D) \times K^\times_x \to {\Q}/{\Z}$ uniquely factors through
$\Br(X|D) \times {K^\times_x}/{\Fil_{n_x} K^\times_x} \to {\Q}/{\Z}$.
Summing over the points of $D$, we get a pairing
$\Br(X|D) \times {K^\times_\infty}/{I(X|D)} \to {\Q}/{\Z}$.

It follows from the above discussion that for any $d \ge 1$, there
exist bilinear pairings
\begin{equation}\label{eqn:BMP-0*}
  \Br(X^o) \times \left[\sZ_0(X^o) \bigoplus
    \left({\underset{C \in \sC(X|D)}\bigoplus} {k(C)^\times_\infty} 
    \right)\right] \to {\Q}/{\Z}; \ \ \mbox{and}
\end{equation}
\begin{equation}\label{eqn:BMP-0}
  \Br(X|D) \times \left[\sZ_0(X^o) \bigoplus
    \left({\underset{C \in \sC(X|D)}\bigoplus} \frac{k(C)^\times_\infty}{I(C_n|D)} 
    \right)\right] \to {\Q}/{\Z}.
\end{equation}

One checks (cf. Remark~\ref{remk:BMP-compln}) that the pairing
  ~\eqref{eqn:BMP-0} remains unchanged if we replace
  $k(C)^\times_\infty$ by $\wh{k(C)}^\times_\infty$ and $I(C_n|D)$ by
  $\wh{I}(C_n|D)$ for $C \in \sC(X|D)$.

  \vskip .2cm
  
We recall the following reciprocity law due to Kato-Saito \cite{Kato-Saito-83}
when $d =1$.
Let $m \neq 0$ be an integer. Using the localization sequence, together with the
purity theorem (due to Gabber when $p \nmid m$) and \corref{cor:Gersten-10}
when $m = p^n$), we get maps
\begin{equation}\label{eqn:SD-4}
  H^3(K, {\Z}/m(2)) \xrightarrow{\partial_X} {\underset{x \in X_{(0)}}\bigoplus}
  H^2(k(x), {\Z}/m(1))
  \xrightarrow{\sum_x\inv_{k(x)}} {\Z}/m.
\end{equation}

\begin{lem}\label{lem:Rec-Law}
  The composite map $(\sum_x \inv_{k(x)}) \circ \partial_X$ in ~\eqref{eqn:SD-4} is zero.
\end{lem}
\begin{proof}
  This is proven in \cite[\S~5, p.~120]{Kato-Saito-83}.
  \end{proof}

\begin{prop}\label{prop:BMP-1}
  The pairings ~\eqref{eqn:BMP-0*} and ~\eqref{eqn:BMP-0} descend to bilinear maps
  of abelian groups
  \[
    \Br(X^o) \times C(X^o) \to {\Q}/{\Z}; \ \mbox{and}
    \]
    \[
    \Br(X|D) \times \CH_0(X|D) \to {\Q}/{\Z}.
  \]
\end{prop}
\begin{proof}
We let $C \in \sC(X|D)$ and let $\nu \colon C_n \to X$ be the canonical map from the
  normalization of $C$. We write $E = \nu^*(D)$.
  It follows from \propref{prop:BGM-1} and the remark below
  ~\eqref{eqn:Chow-Parshin-0} that there is a diagram
  \begin{equation}\label{eqn:BMP-1-0}
    \xymatrix@C1pc{
      \Br(X|D) \times \sZ_0(X^o) \ar@<-7ex>[d]_{\nu^*} \ar[r] &  {\Q}/{\Z} \ar[d]^-{\id} \\
      \Br(C_n|E) \times \sZ_0(C^o_n) \ar@<-7ex>[u]_-{\nu_*} \ar[r] & {\Q}/{\Z}.}
    \end{equation}
    It follows easily form the construction of ~\eqref{eqn:BMP-0} that this diagram is
    commutative. Using this, and a similar commutative diagram
    associated to the first pairing, it suffices to prove the
    proposition when $d =1$. We shall thus assume that $X$ is a curve.

We let $w \in \Br(X^o)$ and let $\chi_w \colon
    \sZ_0(X^o) \bigoplus {K^\times_\infty} \to {\Q}/{\Z}$ be the associated
    character. We shall show that the composite map
    \begin{equation}\label{eqn:BMP-1-1}
      K^\times \to \sZ_0(X^o) \bigoplus K^\times_\infty \xrightarrow{\chi_w}
      {\Q}/{\Z}
    \end{equation}
    is zero. This will automatically imply that the composite map
    \[
       K^\times \to \sZ_0(X^o) \bigoplus ({K^\times_\infty}/{I(X|D)}) \xrightarrow{\chi_w}
       {\Q}/{\Z}
     \]
     is zero if $w \in \Br(X|D)$. 

We let $I(X) := {\underset{x \in X_{(0)}}{\prod'}} K^\times_x$ denote the restricted
    product with respect to the subgroups $(\sO^h_{X,x})^\times$. We consider the diagram
    \begin{equation}\label{eqn:BMP-1-2}
      \xymatrix@C1pc{
        K^\times \ar[r]^-{\alpha} \ar[dr]_-{(\divf, \iota_X)} & I(X) \ar[r]^-{\gamma}
        \ar[d]^-{\beta} & \Hom_{\Ab}(H^2(K), {\Q}/{\Z}) \ar[d] \\
        & \sZ_0(X^o) \bigoplus {K^\times_\infty} \ar[r]^-{\chi_w} &
        \Hom_{\Ab}(\Br(X^o), {\Q}/{\Z}),}
    \end{equation}
    where $\alpha$ is the canonical inclusion,
    the bottom horizontal arrow is induced by the pairing ~\eqref{eqn:BMP-1-1}
    and the right vertical arrow is induced by the canonical inclusion
    $\Br(X^o) \inj \Br(K)$. To describe $\beta$, we write
    $I(X) = I(X^o) \bigoplus K^\times_\infty$. We let $\beta$ to be identity on
    $K^\times_\infty$ and $\beta(a_x) = v_x(a_x)$ if $a_x \in K^\times_x$ with
    $x \in X^o$. This is well-defined because $v_x((\sO^h_{X,x})^\times) = 0$.
    It is clear that the left triangle commutes.

To define $\gamma$, we let $\chi \in H^2(K)$ and $(a_x) \in I(X)$. It follows
    from \lemref{lem:Brauer-reln}(3) that $\chi \in \Br(U)$ for some dense
    open subscheme $U \subset X^o$. We conclude by \cite[Thm.~2.7(4)]{Saito-Invent}
    that $\<\chi, a_x\> = 0$ for all $x \in U$. In particular, the sum
    $\sum_{x \in X_{(0)}} \<\chi, a_x\>$ is finite. This shows that the cup product
    of \lemref{lem:Hensel-complete}(4) induces a pairing
    \begin{equation}\label{eqn:BMP-1-3}
      H^2(K) \times I(X) \to {\Q}/{\Z},
    \end{equation}
    which is clearly compatible with ~\eqref{eqn:BMP-0*}.
    Letting $\gamma$ denote the induced map $I(X) \to \Hom_{\Ab}(H^2(K), {\Q}/{\Z})$,
it follows that the right square in ~\eqref{eqn:BMP-1-2} is commutative. 
Hence, it suffices to show that $\gamma \circ \alpha = 0$.

We let $a \in K^\times, \ \chi \in H^2(K)$ and write $b = \<\chi, a\> \in
H^3(K)$. We let $a_x$ (resp. $\chi_x$) be the image of $a$ (resp. $\chi$) in
$K^\times_x$ (resp. $H^3(K_x)$). We need to show that
$\sum_{x \in X_{(0)}} \<\chi_x, a_x\> = 0$. Equivalently, we need to show that
$(\sum_x \inv_{k(x)}) \circ \partial_X(b) = 0$ in ~\eqref{eqn:SD-4}.
But this follows from \lemref{lem:Rec-Law}.  
\end{proof}

The following is straightforward from the construction of the pairings
~\eqref{eqn:BMP-0*} and ~\eqref{eqn:BMP-0}.

\begin{cor}\label{cor:Mod-non-mod}
  There is a commutative diagram of bilinear pairings
  \begin{equation}\label{eqn:Mod-non-mod-0}
    \xymatrix@C1pc{
      ``{{\underset{n}\varinjlim}}" \Br(X|nD) \times ``{{\underset{n}\varprojlim}}"
      \CH_0(X|nD) \ar@<-7ex>[d] \ar[r] &  {\Q}/{\Z} \ar[d]^-{\id} \\
      \Br(X^o) \hspace*{.4cm} \times \hspace*{.4cm}
      C(X^o) \ar@<-7ex>[u] \ar[r] & {\Q}/{\Z}}
    \end{equation}
    between ind and pro abelian groups.
\end{cor}

The following corollary provides a quick proof of 
the Brauer-Manin duality isomorphism $\beta_X \colon \Br(X) \xrightarrow{\cong}
(\CH_0(X)^\star)_\tor$ of Lichtenbaum \cite{Lichtenbaum} and Saito \cite{Saito-Invent},
assuming that $\beta_X$ is injective. Recall that the main difficulty in proving that
$\beta_X$ is an isomorphism lies in showing its surjectivity.

\begin{cor}\label{cor:LS-thm}
  Let $X$ be a geometrically connected smooth projective curve over $k$.
  Then the pairing of \propref{prop:BMP-1} (with $D = \emptyset$) induces
  an isomorphism $\beta_X \colon \Br(X) \xrightarrow{\cong} (\CH_0(X)^\star)_\tor$.
\end{cor}
\begin{proof}
We fix an integer $n \ge 1$ and look at the diagram
  \begin{equation}\label{eqn:LS-thm-0}
    \xymatrix@C1pc{
      0 \ar[r] & {\CH_0(X)}/{p^n} \ar[r] \ar[d]_{\alpha_{p^n}} & H^1(X, W_n\Omega^1_{X, \log})
      \ar[d]^-{\beta_{p^n}} \ar[r] & _{p^n}\Br(X) \ar[r] \ar[d]^-{\gamma_{p^n}} & 0 \\
      0 \ar[r] & (_{p^n}\Br(X))^\star \ar[r]  & H^1(X, W_n\Omega^1_{X, \log})^\star \ar[r] &
      ({\CH_0(X)}/{p^n})^\vee, & }
    \end{equation}
    where $\alpha_{p^n}$ and $\gamma_{p^n}$ are induced by the pairing
    ~\eqref{eqn:Mod-non-mod-0}. The map $\beta_{p^n}$ is induced by
    the pairing of \thmref{thm:Duality-Main}.
    From the definition of the torsion-by-profinite topology of
    $H^1(X, W_n\Omega^1_{X, \log})$, one knows that the image of the map
    ${\CH_0(X)}/{p^n} \to H^1(X, W_n\Omega^1_{X, \log})$
    in the top row ~\eqref{eqn:LS-thm-0}
    is open (see ~\eqref{eqn:H1-fin-4} and \S~\ref{sec:Top-X}).
    In particular, the quotient topology of $_{p^n}\Br(X)$ is discrete.
    This shows that the bottom arrow of ~\eqref{eqn:LS-thm-0} is defined and
    exact. It is an easy exercise that this diagram is commutative.

The middle arrow is bijective by \thmref{thm:Duality-Main}.
    Since $\gamma_{p^n}$ is known to be injective (cf. \lemref{lem:BM-pf-0}), it follows
    that $\alpha_{p^n}$ is an isomorphism. Since
    ${\CH_0(X)}/{p^n}$ is profinite\footnote{This is the only place where the
      smoothness of $X$ is used.}, it follows from \lemref{lem:Prof-dual} and
    Pontryagin duality between profinite and discrete torsion groups that
    $\gamma_{p^n}$ maps $_{p^n}\Br(X)$ isomorphically onto $({\CH_0(X)}/{p^n})^\star$.
    Taking the limit, we conclude that
    $\beta_X \colon \Br(X)\{p\} \xrightarrow{\cong} (\CH_0(X)^\star)\{p\}$
    is an isomorphism. 
 An identical argument (with \thmref{thm:Duality-Main}
    replaced by Saito-Tate duality, see \cite[Thm.~9.9]{GKR})
    shows that $\beta_X \colon \Br(X)\{p'\} \xrightarrow{\cong} (\CH_0(X)^\star)\{p'\}$.
    is also bijective. This concludes the proof.
\end{proof}

We end this section with the following functorial property of the Brauer-Manin
pairings of modulus pairs.

\begin{lem}\label{lem:BMP-2}
  Let $f \colon (X',D) \to (X,D)$ be a finite and strict morphism between modulus pairs,
  where $X$ and $X'$ are connected regular projective curves over $k$. Then we have
  a commutative diagram
  \begin{equation}\label{eqn:BMP-2-0}
    \xymatrix@C1pc{
\Br(X'|D') \times \CH_0(X'|D') \ar@<-7ex>[d]_-{f_*} \ar[r] &  {\Q}/{\Z}
      \ar[d]^-{\id} \\
      \Br(X|D) \times \CH_0(X|D) \ar@<-7ex>[u]_-{f^*} \ar[r] & {\Q}/{\Z}.}
    \end{equation}
  \end{lem}
  \begin{proof}
We first note that $f^*$ is the well-known flat pull-back map between
    the Chow groups with modulus and $f_*$ is defined by \propref{prop:BGM-2}.
    To show the commutativity of
    ~\eqref{eqn:BMP-2-0}, we fix a Brauer class $w' \in \Br(X'|D')$ and let
    $w = f_*(w')$. We need to show that the right side triangle in the diagram
   \begin{equation}\label{eqn:BMP-2-1}  
\xymatrix@C1.2pc{ 
\sZ_0(X^o) \ar@{->>}[r] \ar[d]_-{f^*}  & \CH_0(X|D) \ar[d]_-{f^*} \ar[dr]^-{\chi_w} & \\
\sZ_0(X'^o) \ar@{->>}[r] &  \CH_0(X'|D') \ar[r]^-{\chi_{w'}} & {\Q}/{\Z}}
\end{equation}
is commutative.
Since the left square is commutative and its top horizontal arrow is surjective,
it suffices to show that the outer trapezium in ~\eqref{eqn:BMP-2-1} is commutative
when evaluated on every free generator of $\sZ_0(X^o)$.

Let $x \in X^o$ be  a closed point and let $f^*([x]) =
\stackrel{r}{\underset{i = 1}\sum} n_i [y_i]$. Let $\iota_x \colon \Spec(k(x)) \inj X^o$
and $\iota_{y_i} \colon \Spec(k(y_i)) \inj X'^o$ be the inclusion maps.
By the definition of the Brauer-Manin pairing in \S~\ref{sec:Pair**}, we need to show
that $\inv_{k(x)}(\iota^*_x(w)) = \stackrel{r}{\underset{i = 1}\sum} n_i
\inv_{k_{y_i}}(\iota^*_{y_i}(w'))$. But this follows by a combination of
\cite[Prop.~3.8.1, Lem.~3.8.6]{CTS} and \cite[\S~3.2, Prop.~1(2)]{Kato-80}.
\end{proof}

\section{Local vs. global duality for $\G_m$ on compact curve}
\label{sec:LGD}
In this section, we shall establish a duality for some cohomology groups of $\G_m$ on
a smooth projective curve over a local field and show that it is
compatible with the duality {\`a} la Kato \cite{Kato-80} for the local cohomology of
$\G_m$ at closed points of the curve. This compatibility
will be another key step in the proof of \thmref{thm:Main-0}.

\subsection{Local duality for $\G_m$}\label{sec:LD-units}
Let $R$ be a equicharacteristic Henselian discrete valuation ring whose residue field
$\ff$ is a local field of exponential characteristic $p \ge 1$.
Let $L$ denote the quotient field of $R$. Let $x = \Spec(\ff)$ denote the
closed point of $W = \Spec(R)$. We let $\wh{W} = \Spec(\wh{R})$.
We choose a common uniformizer $\pi$ of $R$ and $\wh{R}$,
and let $W_n = \Spec(R/{(\pi^n)})$ for $n \ge 1$. 
The following result is due to Kato \cite[\S~3.5]{Kato-80}.

\begin{thm}\label{thm:Kato-Rec}
  There is a perfect pairing of topological abelian groups
  \[
    \Br(L) \times (L^\times)^\pf \to {\Q}/{\Z},
  \]
  where $\Br(L)$ and $L^\times$ are endowed with the discrete and the Kato topologies,
  respectively.
\end{thm}

We remark that Kato proved this result for $\wh{L}$. However, it is easily checked
that the canonical maps $\Br(L) \to \Br(\wh{L})$ and $({L}^\times)^\pf \to  
(\wh{L}^\times)^\pf$ are topological isomorphisms.

Let $\gamma_L \colon \Br(L) \to (L^\times)^\star$ denote the map induced by the
above pairing. It follows from \thmref{thm:Kato-Rec}
(see \cite[Rem.~2.14]{Saito-Invent}) that there is a commutative diagram of exact
sequences
\begin{equation}\label{eqn:Loc-dual-1}
  \xymatrix@C1pc{
    0 \ar[r] & \Br(R) \ar[r] \ar[d]_-{\gamma_R} & \Br(L) \ar[r] \ar[d]^-{\gamma_L} &
    H^3_{\et,x}(W, \sO^\times_W) \ar[r] \ar[d]^-{\gamma_x} & 0 \\
    0 \ar[r] & {\Z}^\star \ar[r]^-{(v_L)^\star} & (L^\times)^\star \ar[r] & (R^\times)^\star
    \ar[r] & 0,}
\end{equation}
where the duals $(L^\times)^\star$ and $(R^\times)^\star$ are considered with respect to the
Kato topology.

In the above diagram, $\gamma_R$ is an isomorphism because it is the composition 
$\Br(R) \xrightarrow{\cong} \Br(\ff) \xrightarrow{\inv_\ff} {\Q}/{\Z} = {\Z}^\star$.
$\gamma_L$ is an isomorphism if $p > 1$ by \cite[\S~3.5, Rem.~4]{Kato-80}.
It maps $\Br(L)$ isomorphically onto the subgroup $((L^\times)^\star)_\tor$ if $p =1$.
It follows that $\gamma_x$ is an isomorphism if $p > 1$, and maps
$H^3_{x}(W, \sO^\times_W)$ isomorphically onto $((R^\times)^\star)_\tor$ if $p =1$.
Recall that each $\sO^\times(W_n)$ is endowed with the adic topology.

\begin{lem}\label{lem:Loc-dual-3}
  For $R$ as above, we have the following.
  \begin{enumerate}
    \item
      There exists a unique isomorphism $\gamma_x \colon H^3_{x}(W, \sO^\times_W)
      \xrightarrow{\cong} ((R^\times)^\star)_\tor$ such that ~\eqref{eqn:Loc-dual-1} is a
      commutative diagram.
\item
 If $p > 1$, each of the groups $(R^\times)^\star, \  (\wh{R}^\times)^\star$ and
 $(\sO^\times({W_n}))^\star$ ($n \ge 1$) is torsion.
\item
  The canonical maps $H^3_{x}(W, \sO^\times_W) \to H^3_{x}(\wh{W}, \sO^\times_{\wh{W}})$
  and $(\wh{R}^\times)^\star \to  (R^\times)^\star$ are isomorphisms.
\item
If $p > 1$, one has a canonical isomorphism
\begin{equation}\label{eqn:Loc-dual-5}
  \gamma_x \colon H^3_{x}(W, \sO^\times_W) \xrightarrow{\cong}
  {\varinjlim}_n \ (\sO^\times(W_n))^\star,
\end{equation}
where the duals on the right hand side are taken with respect to the adic topology.
\item
  If $p =1$, the canonical map $(\ff^\times)^\star \to (\wh{R}^\times)^\star$
  induces an isomorphism $(\ff^\times)^\star \cong ((\wh{R}^\times)^\star)_\tor$.
\end{enumerate}
\end{lem}
\begin{proof}
 To prove (2), note that the map $(\sO^\times({W_n}))^\star \inj  (\wh{R}^\times)^\star$
  is injective for all $n \ge 1$. Using the bottom row of ~\eqref{eqn:Loc-dual-1},
  it suffices therefore to show that $(L^\times)^\star$ and $(\wh{L}^\times)^\star$ are
  torsion groups. Using the definition of Kato topology, it is an easy exercise that the
  canonical map $(\wh{L}^\times)^\star \to (L^\times)^\star$ is an isomorphism. But
\cite[\S~3.5, Rem.~4]{Kato-80} says that $(\wh{L}^\times)^\star$ is a torsion group
if $p > 1$. The statement (1) is already shown above.
The statement (3) follows easily by comparing ~\eqref{eqn:Loc-dual-1} for $R$ and
$\wh{R}$ and using the isomorphism $(\wh{L}^\times)^\star \xrightarrow{\cong}
(L^\times)^\star$. To show (4),
  we note that the map ${R}/{(\pi^n)} \to {\wh{R}}/{(\pi^n)}$ is bijective
  for $n \ge 1$.  Furthermore, the map
  $\wh{R}^\times \to {\varprojlim}_n \ \sO^\times({W_n})$
  is an isomorphism of topological abelian groups by \lemref{lem:Kato-1}.
  It is an elementary checking (cf. \lemref{lem:dual-surj}) that the induced map
  ${\varinjlim}_n \ (\sO^\times({W_n}))^\star \to (\wh{R}^\times)^\star$
  is bijective. 

To prove (5), we can replace $(\wh{R}^\times)^\star$ by $(\wh{R}^\times)^\vee$
  since  $\wh{R}^\times$ is a discrete group. By
  \lemref{lem:Char-0-dual}, we can also replace $(\ff^\times)^\star$ by
  $((\ff^\times)^\vee)_\tor$. We are now done because it is clear that
  $(\ff^\times)^\vee \to (\wh{R}^\times)^\vee$ is injective, and it is surjective at the
  level of torsion subgroups by \cite[Lem.~5.9]{GKR}. This concludes the proof.
\end{proof}

\subsection{Global duality for $\G_m$}\label{sec:GD-units}
Let $k$ be a local field of exponential characteristic $p \ge 1$ and let $X$ be a
smooth and geometrically connected projective curve over $k$.
Let $K$ denote the function field of
$X$. The subject of this subsection is the construction of a perfect duality
between $H^0(X, \sO^\times_X)$ and $H^3(X, \sO^\times_X)$ which is compatible with
the local duality of \S~\ref{sec:LD-units}.

One can deduce from the
Hochschild-Serre spectral sequence that there is an isomorphism
$H^3(X, \sO^\times_X) \cong H^2(k, \Pic(X_s))$, where $X_s$ is the base change of $X$
by $k_s$. Using the decomposition $\Pic(X_s) \cong \Pic^0(X_s) \oplus \Z$ as
$G_k$-module and the vanishing $H^2(k, \Pic^0(X_s)) = 0$
(see \cite[Thm.~7.8]{Milne-Duality}), one deduces an isomorphism
$\nu_X \colon H^3(X, \sO^\times_X) \cong (G_k)^\star \cong (k^\times)^\star$.
However, this isomorphism does not serve our purpose. The reason is that it is
not known if $\nu_X$ is compatible with the local duality isomorphism $\gamma_x$ in
\lemref{lem:Loc-dual-3} for a closed point $x \in X$, in positive characteristic.
As mentioned above, this compatibility is very crucial for proving our main results.
The compatibility of $\nu_X$ with the local duality seems like a
challenging independent problem. The main obstacle in solving this is the failure
of purity for $H^2_x(X, W_n\Omega^1_{X.\log})$.

To serve our purpose, we shall construct a different pairing between
$H^0(X, \sO^\times_X)$ and $H^3(X, \sO^\times_X)$. We shall show that this pairing is
perfect and is compatible with the local duality isomorphism $\gamma_x$ for
$x \in X_{(0)}$. We do not have any guess whether the two duality maps coincide.
We let $H^i(X, \wh{\Z}(j)) := {\varprojlim}_n \ H^i(X, {\Z}/n(j))$,
where the limit is taken over $\N$. We begin with the following.

\begin{lem}\label{lem:GD-units-0}
There exists a canonical bilinear pairing
\begin{equation}\label{eqn:GD-units-00}
H^3(X, \sO^\times_X) \times H^0(X, \sO^\times_X) \to {\Q}/{\Z}.
\end{equation}
\end{lem}
\begin{proof}
By \lemref{lem:Mot-coh-3}, we can replace $H^3(X, \sO^\times_X)$ by
  $H^3(X, {\Q}/{\Z}(1))$. Using ~\eqref{eqn:Limit-duality} and the analogous diagram
  for the sheaves ${\Z}/{n}(r)$ (with $n$ prime to $p$), we get a bilinear pairing
  between ind-abelian and pro-abelian groups
  \[
    ``{{\underset{n}\varinjlim}}" \ H^i(X, {\Z}/n(j)) \times
    ``{{\underset{n}\varprojlim}}" \ H^{i'}(X, {\Z}/n(j')) \to
    ``{{\underset{n}\varinjlim}}" \ H^{i+i'}(X, {\Z}/n(j+j')).
  \]
Taking the limits, we conclude by the Saito-Tate duality
  \cite{Saito-Duality} (see also \cite[Thm.~9.9]{GKR}) in the prime-to-$p$ case
  and \corref{cor:Limit-duality-1} in the $p$-primary case
  that there is a bilinear pairing between abelian groups
  \begin{equation}\label{eqn:GD-units-1}
    H^3(X, {\Q}/{\Z}(1)) \times H^1(X, \wh{\Z}(1)) \to {\Q}/{\Z}.
    \end{equation}
Composing this with the canonical maps (see ~\eqref{eqn:Mot-coh-0})
\begin{equation}\label{eqn:GD-units-2} 
H^0(X, \sO^\times_X) \to {\varprojlim}_n \ {H^0(X, \sO^\times_X)}/n 
\inj H^1(X, \wh{\Z}(1)),
\end{equation}
we get the desired pairing.
\end{proof}

In the remainder of this subsection, our goal is to prove the perfectness of
~\eqref{eqn:GD-units-00} with respect to the discrete topology on
$H^3(X, \sO^\times_X)$ and the adic topology on $H^0(X, \sO^\times_X) \cong k^\times$.
We begin by noting that the diagrams in ~\eqref{eqn:Mot-coh-1} give rise to an exact
sequence of pro-abelian groups
\begin{equation}\label{eqn:GD-units-3}
  0 \to \{{H^0(X, \sO^\times_X)}/{n}\} \to \{H^1(X, {\Z}/n(1))\} \to
  \{_n\Pic^0(X)\} \to 0,
\end{equation}
whose transition maps are induced by the canonical surjections
${\Z}/{mn} \surj {\Z}/{n}$. Taking the limits, we get an exact sequence
\begin{equation}\label{eqn:GD-units-4}
  0 \to {\varprojlim}_n \ {H^0(X, \sO^\times_X)}/n \to H^1(X, \wh{\Z}(1))
  \to {\varprojlim}_n \ _n\Pic^0(X) \to 0.
  \end{equation}

  Applying the dual functor in the category of abelian groups, we get an
  exact sequence
 \begin{equation}\label{eqn:GD-units-5}  
0 \to  {\varinjlim}_n \ (_n\Pic^0(X))^\vee \to
     {\varinjlim}_n \ (H^1(X, {\Z}/n(1)))^\vee \to
     {\varinjlim}_n \ ({H^0(X, \sO^\times_X)}/n)^\vee \to 0.
     \end{equation}

\begin{lem}\label{lem:GD-units-6} 
  We have ${\underset{n}\varprojlim} \ _n\Pic^0(X) = {\underset{n}\varinjlim} \
  (_n\Pic^0(X))^\vee = 0$.
\end{lem}
\begin{proof}
  One knows that there is a canonical inclusion
  ${\varprojlim}_n \ _n\Pic^0(X)  \inj {\varprojlim}_n \ _n J_X(k)$, where $J_X$
  is the Jacobian variety of $X$ (e.g., see \cite[Rem.~1.5]{Milne-Jac}).
  On the other hand, $J_X(k)$ is a profinite abelian group
  (see the proof of \corref{cor:Pic-sequence-2}) and ${\varprojlim}_n \  _n J_X(k)$
  coincides with the Tate module of $J_X(k)$. It follows that
  ${\varprojlim}_n \ _n J_X(k) = 0$ as the Tate module of a profinite abelian
  group is zero.

  To show that ${\varinjlim}_n \ (_n\Pic^0(X))^\vee$ is zero, it suffices to
  show the stronger assertion that the pro-abelian group
  $\{_n\Pic^0(X)\}$ is zero as it would imply that the ind-abelian group
  $\{(_n\Pic^0(X))^\vee\}$ is zero.
  But this follows from our first assertion and the well known fact that a
  pro-abelian group satisfying the Mittag-Leffler condition is zero if and only if its
  limit is zero (e.g., see \cite[Tag~07KV, Lem.~15.86.13]{SP}).
\end{proof}

\begin{lem}\label{lem:GD-units-7} 
  The map $\delta_X \colon H^3(X, \sO^\times_X) \to  H^0(X, \sO^\times_X)^\vee$,
  induced by ~\eqref{eqn:GD-units-00}, is injective.
\end{lem}
\begin{proof}
  By \lemref{lem:Mot-coh-3} and ~\eqref{eqn:GD-units-2}, $\delta_X$ is the same as
  the composite map
  \[
H^3(X, {\Q}/{\Z}(1)) \to  {\varinjlim}_n \ H^1(X, {\Z}/n(1))^\vee
   \to {\varinjlim}_n \  ({H^0(X, \sO^\times_X)}/n)^\vee \to
   H^0(X, \sO^\times_X)^\vee.
 \]
   The first arrow in this sequence is injective in the $p$-primary case by
   \thmref{thm:Duality-Main}
   and isomorphism in the prime-to-$p$ case by \cite[Thm.~9.9]{GKR}.
   The second arrow is injective by ~\eqref{eqn:GD-units-5} and \lemref{lem:GD-units-6}.
   The third arrow is easily seen to be injective. This concludes the proof.
 \end{proof}

The following shows the compatibility between the local and global duality maps.

\begin{lem}\label{lem:LD-GD}
  Let $x \in X_{(0)}, \ R = \sO^h_{X,x}$ and let the notations be as in
Lemma~\ref{lem:Loc-dual-3}. Then the diagram
  \begin{equation}\label{eqn:LD-GD-0}
    \xymatrix@C1pc{
      H^3_{x}(W, \sO^\times_W) \ar[d]_-{\gamma_x} \ar[r] & H^3(X, \sO^\times_X)
      \ar[d]^-{\delta_X} \\
      (R^\times)^\star \ar[r] & (H^0(X, \sO^\times_X))^\vee}
  \end{equation}
  is commutative if we let the top horizontal arrow be the forget support map and
  the bottom horizontal arrow be the dual of the canonical pull-back map
  $H^0(X, \sO^\times_X) \to R^\times$.
\end{lem}
\begin{proof}
Using ~\eqref{eqn:Loc-dual-1}, the lemma is equivalent to showing that the diagram
 \begin{equation}\label{eqn:LD-GD-1}
    \xymatrix@C1pc{ 
H^2(K_x) \ar[r]^-{\partial_x} \ar[d]_-{\gamma_{K_x}} & H^3(X, \sO^\times_X) 
  \ar[d]^-{\delta_X} \\
  (K^\times_x)^\star \ar[r]^-{\phi_x} &  (H^0(X, \sO^\times_X))^\vee}
\end{equation}
is commutative, where $\phi_x$ is induced by the canonical pull-back map.
To show this, we can replace $ H^3(X, \sO^\times_X)$ by
$H^3(X, {\Q}/{\Z}(1))$ using \lemref{lem:Mot-coh-3}. It also suffices to
prove the commutativity after replacing ${\Q}/{\Z}$ by ${\Z}/n$ for all $n \ge 1$.

We now fix an integer $n \ge 1$. We let $\alpha \in H^2(K_x, {\Z}/n(1))$ and
$\beta \in H^0(X, \sO^\times_X)$ be arbitrary elements. We look at the diagram
 \begin{equation}\label{eqn:LD-GD-2}
    \xymatrix@C1pc{ 
      H^2(K_x, {\Z}/n(1)) \times H^1(K_x, {\Z}/n(1)) \ar@<-7ex>[d]_-{\partial_x}
      \ar[r]^-{\cup} & H^3(K_x, {\Z}/n(2)) \ar[d]^-{\partial_x}_-{\cong} \ar[r]^-{\Tr_x}
      & {\Z}/n \ar[dd]^-{\id} \\
      H^3_{x}(W, {\Z}/n(1)) \times H^1(W, {\Z}/n(1))  \ar@<-7ex>[u]
      \ar@<-7ex>[d]_-{\cong} \ar[r]^-{\cup} &
      H^4_{x}(W, {\Z}/n(2)) \ar[d]^-{\cong}  & \\
      H^3_{x}(X, {\Z}/n(1)) \times H^1(X, {\Z}/n(1))  \ar@<-7ex>[u]
      \ar@<-7ex>[d] \ar[r]^-{\cup} & H^4_{x}(X, {\Z}/n(2)) \ar[d]
      \ar[r]^-{\epsilon_x} & {\Z}/n \ar[d]^-{\id} \\
      H^3(X, {\Z}/n(1)) \times H^1(X, {\Z}/n(1))
      \ar@<-7ex>[u]_-{\id} \ar@<-7ex>[d]_-{\id} \ar[r]^-{\cup} &
      H^4(X, {\Z}/n(2)) \ar[r]^-{\Tr_X} & {\Z}/n \\
      H^3(X, {\Z}/n(1)) \times {H^0(X, \sO^\times_X)}/n \ar@<-7ex>[u]_-{\psi_X}
      \ar[r] &   H^4(X, {\Z}/n(2)) \ar[u]^-{\id} \ar[ur]_-{\Tr_X}, &}
  \end{equation}
  where $\psi_X$ is the canonical map (see ~\eqref{eqn:GD-units-3}).

All squares on the left are induced by the cup products, pull-back and the boundary
maps in {\'e}tale cohomology. In particular, these squares are known to be
commutative (e.g., see ~\eqref{eqn:Duality-Main-4} for the commutativity of the top
square on the left).
The upper square on the right is commutative by the definition of $\Tr_x$
(see \corref{cor:Gersten-10} and note that this also holds when $p \nmid n$).
The lower square on the right commutes by ~\eqref{eqn:Gersten-8} when
$n = p^n$ and by \cite[Lem.~9.7]{GKR} when $p \nmid n$.
The bottom right triangle clearly commutes.

We let $\theta \colon H^2(K_x, {\Z}/n(1)) \to H^3(X, {\Z}/n(1))$ be the
    composition of all vertical arrows going down on the extreme left of
    ~\eqref{eqn:LD-GD-2} and let $\theta' \colon {H^0(X, \sO^\times_X)}/n \to
    H^1(K_x, {\Z}/n(1))$ be the composition of all vertical arrows going up
    in the middle of ~\eqref{eqn:LD-GD-2}. It follows then by a diagram chase that
    $\Tr_x(\alpha \cup \theta'(\beta)) = \Tr_X(\theta(\alpha) \cup \psi_X(\beta))$.
On the other hand, it is straightforward to check that
    \[
      \<\phi_x \circ \gamma_{K_x}(\alpha), \beta\> = \Tr_x(\alpha \cup \theta'(\beta))
      \ \ \mbox{and} \ \
      \<\delta_X \circ \partial_x(\alpha), \beta\> = \Tr_X(\theta(\alpha) \cup
      \psi_X(\beta)).
    \]
    This proves the commutativity of ~\eqref{eqn:LD-GD-1} and concludes the proof of
    the lemma.
\end{proof}

We shall also need the following independent result for proving 
the perfectness of ~\eqref{eqn:GD-units-00} as well as in the proof of
\lemref{lem:Dual-surjection}.

\begin{lem}\label{lem:Lifting}
  Let ${k'}/{k}$ be a finite field extension of $k$. Then the inclusion
  $k^\times \inj k'^\times$ induces a surjective map $(k'^\times)^\star \surj
  (k^\times)^\star$ when the duals are taken with respect to the adic topologies.
\end{lem}
\begin{proof}
  Let $\chi \in  (k^\times)^\star$.
  Since the inclusion $k^\times \inj k'^\times$ is continuous and its
  image is closed (e.g., see \cite[\S~7.2]{Kato-cft-1} or
  \cite[Prop.~II.5.7]{Neukirch}) with respect to  the adic
  topologies, it follows from \cite[Cor.~4.42]{Folland} that there exists a
  continuous homomorphism $\chi' \colon k'^\times \to \T$ whose restriction to $k^\times$
  is $\chi$. It remains to show that $\chi'$ has finite order.
  Since $\sO^\times_{k'}$ is profinite, it is clear that $\chi'$ has finite order when
  we restrict it to $\sO^\times_{k'}$. We let $G = (k^\times) \cdot (\sO^\times_{k'}) \subset
  k'^\times$. Then $G$ is a closed subgroup of $k'^\times$ and $\chi'$ has finite order
  (say, $m$) on $G$.

  We now look at the commutative diagram of exact sequences
  \begin{equation}\label{eqn:Dual-surjection-7}
      \xymatrix@C1pc{
        0 \ar[r] & G \ar[r] \ar@{->>}[d] & k'^\times \ar[r] \ar[d]^-{\chi'} &
        {\Z}/{e} \ar[d]
        \ar[r] & 0 \\
       0 \ar[r] & {\Z}/m \ar[r] & \T \ar[r]^-{m} & \T \ar[r] & 0,}
\end{equation}
where $e$ is the ramification index of ${k'}/{k}$.
It follows that $\chi'(k'^\times)$ is finite. This concludes the proof.
\end{proof}

We can now prove the main result of \S~\ref{sec:LGD}.

\begin{thm}\label{thm:GD-units-main}
  The bilinear pairing of ~\eqref{eqn:GD-units-00} induces a perfect pairing of
  topological abelian groups
\begin{equation}\label{eqn:GD-units-main*}
      \delta_X \colon H^3(X, \sO^\times_X) \times H^0(X, \sO^\times_X)^\pf
      \to {\Q}/{\Z}.
      \end{equation}
  \end{thm}
  \begin{proof}
We can identify $H^0(X, \sO^\times_X)$ with $k^\times$ and
    $H^3(X, \sO^\times_X)$ with $H^3(X, {\Q}/{\Z}(1))$ (cf. \lemref{lem:Mot-coh-3}).
     In the first step, we show that $\delta_X \colon H^3(X, \sO^\times_X) \to
     (k^\times)^\vee$ factors through $(k^\times)^\star$. Since $(k^\times)^\star \cong
     \Hom_\Tab(k^\times, {\Q}/{\Z}\{p'\}) \bigoplus 
    \Hom_\Tab(k^\times, {\Q_p}/{\Z_p})$ and $H^3(X, \sO^\times_X)$ is a torsion
    group, it suffices to consider the prime-to-$p$ and $p$-primary cases
    separately.

On the prime-to-$p$ torsion subgroup,  $\delta_X$ is identified with the composite map
    \begin{equation}\label{eqn:GD-units-main-0}
      H^3(X, {\Q}/{\Z}(1))\{p'\} \to  {\underset{n \in k^\times}\varinjlim} \
      (H^1(X, {\Z}/n(1))^\vee \to {\underset{n \in k^\times}\varinjlim} \
      ({k^\times}/n)^\vee \to (k^\times)^\vee.
 \end{equation}
 Since the groups $H^i(X, {\Z}/n(j))$ are finite for $n \in k^\times$
 (e.g., see \cite[Thm.~9.9]{GKR}),
 we have $H^1(X, {\Z}/n(1))^\star \cong H^1(X, {\Z}/n(1))^\vee$.
 Since ${k^\times}/n$ is also discrete (and finite),
 it follows that $({k^\times}/n)^\star \cong
 ({k^\times}/n)^\vee$ for every $n \ge 1$. Hence, the above sequence of maps has a
 factorization
 \begin{equation}\label{eqn:GD-units-main-1}
   H^3(X, {\Q}/{\Z}(1))\{p'\} \to  {\underset{n \in k^\times}\varinjlim} \
 H^1(X, {\Z}/n(1))^\star \to {\underset{n \in k^\times}\varinjlim} \
 ({k^\times}/n)^\star \to (k^\times)^\star.
 \end{equation}

To prove the $p$-primary case, we can assume $p > 1$. We
   choose a closed point $x \in X$. We follow the
   notations of \lemref{lem:Loc-dual-3} and look at the commutative diagram
   \begin{equation}\label{eqn:GD-units-main-2}
   \xymatrix@C1pc{
  H^3_{x}(W, \sO^\times_W)\{p\} \ar[d]_-{\gamma_x} \ar[r] & H^3(X, \sO^\times_X)\{p\}
      \ar@{^{(}->}[dr]^-{\delta_X} \ar@{.>}[d] & \\
      {\varinjlim}_n \ H^0(W_n, \sO^\times_{W_n})^\star  \ar[r] &
      (k^\times)^\star \ar@{^{(}->}[r] &  (k^\times)^\vee.}
  \end{equation}
  Note that $\delta_X$ is injective by \lemref{lem:GD-units-7}.
It is clear that the map $H^0(W_n, \sO^\times_{W_n})^\star  \to (k^\times)^\vee$
  factors through $(k^\times)^\star$ for every $n \ge 1$. In particular, the bottom
  left arrow is defined. Moreover, the outer trapezium commutes by \lemref{lem:LD-GD}.
  To show that the image of $\delta_X$ lies in $ (k^\times)^\star$, it suffices therefore
  to show that the top horizontal arrow is surjective.

To show the last claim, we set $U = X \setminus \{x\}$. We now recall that
  $H^i(Z, \sO^\times_Z)$ is a torsion group for all $i \ge 2$ if $Z\in \{X,U\}$.
  Using the support cohomology exact sequence, this implies that
  $H^i_{x}(W, \sO^\times_W) \cong H^i_{x}(X, \sO^\times_X)$ is a torsion group for
  $i \ge 3$. (the case of interest $i = 3$ also follows from
  \lemref{lem:Loc-dual-3}). In particular, the sequence
  \[
  H^3_{x}(W, \sO^\times_W)\{p\} \to H^3(X, \sO^\times_X)\{p\} \to H^3(U, \sO^\times_U)\{p\}
\]
is exact.

On the other hand, there is a surjection
$H^2(U, W_n\Omega^1_{U, \log}) \surj ~_{p^n}H^3(U, \sO^\times_U)$ for any $n \ge 1$, as
follows from ~\eqref{eqn:Mot-coh-0}. It suffices therefore to show that
$H^2(U, W_n\Omega^1_{U, \log}) = 0$ for all $n \ge 1$.
 For $n =1$, note that in the exact sequence
  \[
    0 \to \Omega^1_{U, \log} \to \Omega^1_{U} \xrightarrow{\pi - \ov{F}}
    \frac{\Omega^1_{U}}{d\sO_U} \to 0
  \]
  (see ~\eqref{eqn:dlog-2-*}), the middle and the term on the right are coherent
  $\sO_U$-modules. In particular, their cohomological dimension is
  zero as $U$ is affine. It follows that $H^2(U, \Omega^1_{U, \log}) = 0$.
  The general case follows by induction on $n$ using the cohomology sequence
  associated to ~\eqref{eqn:Zhao*-0}.

We shall now show that $\delta_X \colon  H^3(X, \sO^\times_X) \to (k^\times)^\star$
is bijective. It is injective by \lemref{lem:GD-units-7}.
To prove its surjectivity between the prime-to-$p$ torsion subgroups, note that
all arrows in ~\eqref{eqn:GD-units-main-1} are isomorphisms (the first arrow
is an isomorphism by the Saito-Tate duality, see \cite[Thm.~9.9]{GKR}).
To prove the surjectivity of $\delta_X$ between the $p$-primary torsion subgroups,
note that this map is the same as the dotted arrow in ~\eqref{eqn:GD-units-main-2}.
The map $\gamma_x$ is bijective by \lemref{lem:Loc-dual-3}(4). It suffices therefore
to show that the map ${\varinjlim}_n \ H^0(W_n, \sO^\times_{W_n})^\star  \to (k^\times)^\star$
is surjective. Equivalently, the map $(\wh{R}^\times)^\star \to (k^\times)^\star$
is surjective, where recall that $R = \sO^h_{X,x}$.
But this latter map is induced by the canonical inclusion
$k^\times \inj \wh{R}^\times$. We are now done because the composite map
$(k(x)^\times)^\star \to (\wh{R}^\times)^\star \to (k^\times)^\star$ is surjective by
\lemref{lem:Lifting}.

To prove the perfectness of ~\eqref{eqn:GD-units-main*}, we let $F = (k^\times)^\pf$.
We now note that
the canonical map $F^\star \to (k^\times)^\star$ is an isomorphism, as one
easily checks using \cite[Prop.~II.5.7]{Neukirch} and the profiniteness of $\sO^\times_k$.
It follows that ~\eqref{eqn:GD-units-00} gives rise to ~\eqref{eqn:GD-units-main*}
and the resulting map $\delta_X \colon H^3(X, {\Q}/{\Z}(1)) \to F^\star$ is an
isomorphism.
Since this map is clearly continuous and $F$ is profinite, it follows that
~\eqref{eqn:GD-units-main*} is a perfect pairing of topological abelian groups.
This concludes the proof.
\end{proof}

\begin{remk}\label{remk:Vanish-open-p}
  We showed in the proof of \thmref{thm:GD-units-main} that $H^3(U, \sO^\times_U)\{p\}
  = 0$ if $U$ is the complement of a closed point in $X$. But the
  proof actually shows that $H^i(U, \sO^\times_U)\{p\} = 0$ if $U$ is
  any affine open in $X$ and $i \ge 3$ is any integer. This observation will be used
  in the proof of \thmref{thm:Main-0}.
  \end{remk}

\begin{cor}\label{cor:GD-units-main-3}
  For any nonempty effective Cartier divisor $D \subset X$, the forget support map
  $H^3_{D}(X, \sO^\times_{X}) \to H^3(X, \sO^\times_{X})$ is surjective.
\end{cor}
\begin{proof}
  Using Lemmas~\ref{lem:Loc-dual-3}, ~\ref{lem:LD-GD} and \thmref{thm:GD-units-main},
  it suffices to show that the map $(k(x)^\times)^\star \to (k^\times)^\star$ is
  surjective for any $x \in D$. But this follows from \lemref{lem:Lifting}.
\end{proof}

\section{Continuity of Brauer-Manin pairing}\label{sec:CBMP}
Let $k$ be a local field of exponential characteristic $p \ge 1$.
The goal of this section is to show that the Brauer-Manin pairing for
a 1-dimensional modulus pair $(X,D)$ over $k$ is continuous with respect to the
discrete topology of $\Br(X|D)$ and the adic topology of $\Pic(X|D)$.
This is an important step in the proofs of the main results.  We shall also prove few
more properties of this pairing which will be used in the proof of \thmref{thm:Main-1}.
We begin with the following general statement about regular (but not necessarily smooth)
curves.

\begin{lem}\label{lem:Limit-nD}
  Let $X$ be a connected and regular projective curve over $k$ and let $D \subset X$ be
  a divisor with the complement $X^o$. Then the map  $\theta_{X^o} \colon C(X^o) \to
  {\underset{D' \in \Div_{D_\red}(X)} \varprojlim} \ \CH_0(X|D')$ is bijective.
  In particular, the map $C(X^o) \to \CH_0(X|D')$ is surjective for every
$D' \in \Div_{D_\red}(X)$.
  \end{lem}
  \begin{proof}
    We can replace the right hand side by ${\varprojlim}_n \CH_0(X|nD)$.
    Using Remark~\ref{remk:BMP-compln}, we can replace ${k(X)}^\times_\infty$ (resp.
    $I(X|nD)$) by $\wh{k(X)}^\times_\infty$ (resp. $\wh{I}(X|nD)$) in
    ~\eqref{eqn:Chow-Parshin-0}.
It suffices now to show the stronger statement that the canonical map
$\sZ_0(X^o) \bigoplus {\wh{k(X)}^\times_\infty} \to {\varprojlim}_n
\left(\sZ_0(X^o) \bigoplus \frac{\wh{k(X)}^\times_\infty}{\wh{I}(X|nD)}\right)$
is bijective. It is clear that the right hand side of the latter map is
$\sZ_0(X^o) \bigoplus \left({\varprojlim}_n
  \frac{\wh{k(X)}^\times_\infty}{\wh{I}(X|nD)}\right)$.
It remains therefore to show that for every $x \in D$, the map
$\wh{k(X)}_x^\times \to {\varprojlim}_n \frac{\wh{k(X)}^\times_x}{\Fil_n \wh{k(X)}^\times_x}$
is bijective. But this can be easily checked using the completeness of
$\wh{\sO_{X,x}}$ and the strict exact sequence of pro-abelian groups
\[
  0 \to \left\{\frac{\wh{\sO_{X,x}}^\times}{\Fil_n \wh{k(X)}^\times_x}\right\}_n \to
  \left\{\frac{\wh{k(X)}^\times_x}{\Fil_n \wh{k(X)}^\times_x}\right\}_n \xrightarrow{v_x}
  \Z \to 0.
\]

To see that $C(X^o) \to \CH_0(X|D')$ is surjective, we only have to note that
$\{\CH_0(X|nD'\}_n$ is an inverse system indexed by $\N$ with surjective transition maps
whose limit is $C(X^o)$.
\end{proof}

We now let $X$ be a smooth projective geometrically integral curve
over $k$ and $D \subset X$ an effective divisor with support $D^\dagger$.
We write $D = {\underset{x \in D^\dagger}\sum} n_x [x] \in \Div(X)$. Let
$j \colon X^o = X \setminus D \inj X$ be the inclusion. We shall consider
$C(X^o) = \varprojlim_n \CH_0(X|nD)$ as a topological abelian group, endowed with the
inverse limit of the adic topologies of $\CH_0(X|nD)$ using \lemref{lem:Limit-nD}.

\begin{prop}\label{prop:BM-cont}
  The pairings
  \[
    \Br(X|D) \times \CH_0(X|D) \to {\Q}/{\Z}, \ \ \mbox{and} \ \
    \Br(X^o) \times C(X^o) \to {\Q}/{\Z}
  \]
  are continuous.
\end{prop}
\begin{proof}
Let $\beta_{X|D} \colon \Br(X|D) \to \CH_0(X|D)^\vee$ and $\beta_{X^o} \colon
  \Br(X^o) \to C(X^o)^\vee$ be the induced homomorphisms.
As mentioned in \S~\ref{sec:Duality-X}, the proposition is equivalent to
the statement that the image of $\beta_{X|D}$ (resp. $\beta_{X^o}$) lies in
$\CH_0(X|D)^\star$ (resp. $C(X^o)^\star$). Note that this uses the discreteness of
$\Br(X^o)$.
Since $\varinjlim_n \CH_0(X|nD)^\star \xrightarrow{\cong} C(X^o)^\star$,
it suffices to show that $\beta_{X|D}(\Br(X|D)) \subset \CH_0(X|D)^\star$.
To prove the latter statement, we
  fix $w \in \Br(X|D)$ and let $\chi_w \colon \CH_0(X|D) \to {\Q}/{\Z}$ be the
  induced character.

 Assume first that $p = 1$. As  $nw = 0$ for some $n \ge 1$, the map
  $\chi_w \colon \CH_0(X|D) \to {\Q}/{\Z}$ must factor through ${\CH_0(X|D)}/n$.
  On the other hand, ${\CH_0(X|D)}/n$ is a finite discrete group,
  as one easily checks using ~\eqref{eqn:CCM-0}, \cite[Prop.~II.5.7]{Neukirch}
  and the Kummer sequence.
  In particular, the map $\chi_w \colon {\CH_0(X|D)}/n \to {\Q}/{\Z}$ is continuous.
  It follows that $\beta_{X|D}(w) \in \CH_0(X|D)^\star$.
  We shall now assume in the rest of the proof that $p > 1$.

We begin by proving the special case of the proposition when $D$ and
$X^o(k)$ are not empty. We fix a closed point $P \in X^o(k)$.
As $\chi_w$ is a group homomorphism, it will be enough to show that
  it is continuous at some point of $\Pic(X|D)$.
We now note that the composition 
  $\chi^1_w \colon X^o(k) \inj \sZ_0(X^o) \surj \Pic(X|D) \xrightarrow{\chi_w} {\Q}/{\Z}$
  (see ~\eqref{eqn:BMP-2-1}) is given by $\chi^1_w(x) = \inv_{k(x)}(\iota^*_x(w))$, where
  $\iota_x \colon \Spec(k(x)) \inj X^o$ is the inclusion. It follows from
  \cite[Prop.~10.5.2]{CTS} (see also \cite[Prop.~8.2.9]{Poonen})
  that $\chi^1_w$ is continuous. In particular, the map $\chi^r_w \colon (X^o(k))^r \to
  {\Q}/{\Z}$, given by $\chi^r_w(x_1, \ldots , x_r) = \stackrel{r}{\underset{i =1}\sum}
  \chi^1_w(x_i)$, is continuous for every $r \ge 1$.

We let $\alb_{X|D} \colon X^o(k) \to \Picc^0(X|D)$ be the albanese map,
  given by $\alb_{X|D}(x) = [x] - [P]$ (cf. ~\eqref{eqn:ALB}).
  It follows from \cite[Thm.~10.5.1]{CTS} and \thmref{thm:ALB-rel} that for all
  $r \gg 0$, there exists a Zariski open subscheme $U \subset (X^o)^r$ 
  such that the composite map $\phi^r_{X|D} \colon U(k) \inj (X^o(k))^r \to
  \Picc^0(X|D)(k) = \Pic^0(X|D)$ between the adic spaces is open and has a dense image.
  We now look at the commutative diagram
  \begin{equation}\label{eqn:BM-cont-0}
    \xymatrix@C1pc{
      U(k) \ar[r]^-{\phi^r_{X|D}} \ar[d]_-{\chi^r_w} & \Pic^0(X|D) \ar[d]^-{\chi_w} \\
      {\Q}/{\Z} \ar[r]^-{-r\chi_w(P)} & {\Q}/{\Z}.}
  \end{equation}
  We showed above that the left vertical arrow is continuous. Since the bottom
  horizontal arrow is clearly continuous, it follows that $\chi_w \circ \phi^r_{X|D}$
  is continuous. Since the image of $\phi^r_{X|D}$ is open dense and
  the map $U(k) \to \phi^r_{X|D}(U(k))$ is open, it follows that
  $\chi_w$ is continuous on the nonempty open subset $\phi^r_{X|D}(U(k))$.
  This proves the proposition when $X^o(k)$ and $D$ are not empty.

We assume next that $D \neq \emptyset$ but $X(k) = \emptyset$. We let
  $G^\star(X^o) = {\varinjlim}_n \ \CH_0(X|nD)^\star$ and
  $G^\vee(X^o) = {\varinjlim}_n \ \CH_0(X|nD)^\vee$.
  Recall from \lemref{lem:Brauer-reln} that ${\varinjlim}_n \Br(X|nD) \xrightarrow{\cong}
  \Br(X^o)$. Suppose we know that the map $\beta_{X^o} \colon \Br(X^o) \to G^\vee(X^o)$
  factors through $G^\star(X^o)$. Then it follows from ~\eqref{eqn:Mod-non-mod-0} that
  $\beta_{X|D}(\chi) \in G^\star(X^o) \cap \CH_0(X|D)^\vee$. In particular, the composite
  map $ \CH_0(X|nD) \surj \CH_0(X|D) \xrightarrow{\chi_w}  {\Q}/{\Z}$ is
  continuous for $n \gg 0$. This implies that $\chi_w$ is continuous because
  the first arrow is a quotient map of adic spaces by \corref{cor:Pic-sequence-1} and
  \cite[Thm.~10.5.1]{CTS}. It remains therefore to show that
  $\beta_{X^o}(\Br(X^o)) \in G^\star(X^o)$.

We fix an element $w \in \Br(X^o)$ and let $\chi_w = \beta_{X^o}(w)$. By
  \cite[Thm.~4.3.1]{Temkin} (see also \cite[Thm.~2.2]{Geisser-Alt}),
  we can find a finite field extension ${k'}/k$ and a geometrically connected
  smooth projective curve $Y$ over $k'$ together with a finite morphism
  $g \colon Y \to X_{k'}$ such that the degrees of ${k'}/k$ and $g$
  are powers of $p$ and $Y$ admits a strict semi-stable reduction over $k'$. That is,
  there exists a projective and flat morphism $\phi \colon \sY \to \Spec(\sO_{k'})$
  such that $\sY$ is regular whose generic fiber is $Y$ and whose (scheme theoretic)
  closed fiber $Y_s$ is a reduced curve with only double point singularities.

  We let $Y^o$ (resp. $E^\dagger$) be the inverse image of $X^o$ (resp. $D^\dagger$)
  under the composite map $Y \to X_{k'} \to X$ and let $\sY^o$ be the complement
  of the closure of $E^\dagger$ in $\sY$.
  By an application of Weil conjectures over the residue field $\F_q$ of $k'$, we can
  find a closed point $P_s \in \sY^o \cap (Y_s)_\reg$ whose degree over $\F_q$ is a
  power of $p$. It is easy to see that there exists a closed point $P \in Y^o$
  whose degree over $k'$ is same as that of $P_s$ over $\F_q$. We let
  $F = k'(P)$ and let $X' = Y_{F}$. Let $f \colon X' \to X$ be the composite finite
  surjective map. We let $X'^o = f^{-1}(X^o)$ and $E = f^*(D)$.  
  It follows then that $X'$ is a smooth and geometrically connected curve over
  the local field $F$ such that $P \in X'^o(F)$ and $\deg(f) = p^n$ for some $n \ge 1$.

  We have the canonical maps $\Br(X^o) \xrightarrow{f^*} \Br(X'^o) \xrightarrow{f_*}
  \Br(X^o)$ whose composition is multiplication by $p^n$ (e.g., see
  \cite[\S~3.8]{CTS}). In particular, $\coker(f_*)$ is torsion of exponent $p^n$.
  On the other hand, $\coker(f_*)$ is a $p$-divisible group by \cite[Thm.~3.2.3]{CTS}.
  It follows that $f_*$ is surjective. In particular, $w = f_*(w')$ for some
  $w' \in \Br(X'^o)$. We can assume that $w' \in \Br(X'|nE)$ for some $n \gg 0$.
Since $X'^o(F) \neq \emptyset$, the previous case of the
lemma shows that $\beta_{X'|nE}(w') \in \CH_0(X'|nE')^\star$. We let
$\chi_{w'} = \beta_{X'|nE}(w')$.

We now look at the diagram
   \begin{equation}\label{eqn:BM-cont-2}
    \xymatrix@C1pc{
      \CH_0(X|nD) \ar[d]_-{f^*} \ar[r]^-{\chi_w}  & {\Q}/{\Z} \\
      \CH_0(X'|nE) \ar[ur]_-{\chi_{w'}}. &}
  \end{equation}
  It follows from \lemref{lem:BMP-2} that this diagram is commutative.
  Since $f^* \colon \Picc(X|nD) \to \Picc(X'|nE)$ is a morphism between locally
  of finite type $k$-schemes by \thmref{thm:Pic-rep}, it follows from
  \cite[Prop.~5.4]{Conrad} that the left vertical arrow is 
  continuous. We deduce that $\chi_w$ is continuous. 
  This concludes the proof when $D \neq \emptyset$.

To prove the proposition for $(X, \emptyset)$, we choose a nonempty effective 
  divisor $D \subset X$ and let $w \in \Br(X) \subset \Br(X|D)$.
   We have shown above that the composite
   map $\CH_0(X|D) \surj \CH_0(X) \xrightarrow{\chi_w} {\Q}/{\Z}$ is continuous.
   On the other hand, it follows from \corref{cor:Pic-sequence-1}
  and \cite[Thm.~10.5.1]{CTS} that the first arrow is a topological quotient map.
  We deduce that $\chi_w$ is continuous. 
\end{proof}

\begin{lem}\label{lem:Dual-surjection}
Let ${k'}/k$ be a finite field extension. Let $X' = X_{k'}, \ D' = D_{k'}$ and
$X'^o = X' \setminus D'$. Let $f \colon X' \to X$ be the projection map.
Then the following hold.
\begin{enumerate}
\item
  The map $f^* \colon \Pic(X|D) \to \Pic(X'|D')$ is a closed embedding of
  adic spaces.
\item
$f^*$ induces a surjective homomorphism
  $(f^*)^\star \colon \Pic(X'|D')^\star \surj \Pic(X|D)^\star$.
\end{enumerate}
\end{lem}
\begin{proof}
 By \thmref{thm:Pic-rep}, the pull-back map $f^* \colon \Pic(X|D) \to \Pic(X'|D')$
  is the canonical inclusion $\Picc(X|D)(k) \subset \Picc(X|D)(k')$. To prove (1), we
  therefore need to show that the inclusion $\Picc(X|D)(k) \subset \Picc(X|D)(k')$ of
  adic spaces is closed. But this follows from \cite[Prop.~5.11(3)]{Conrad} since the
  inclusion $k \subset k'$ is easily seen to be closed (e.g., see
  \cite[Prop.~II.5.7]{Neukirch}).

Since ${\Q}/{\Z}$ is a divisible group, one easily reduces the proof of (2) to showing
  that the map $(f^*)^\star \colon \Pic^0(X'|D')^\star \to \Pic^0(X|D)^\star$ is surjective.
  Let $\chi \colon \Pic^0(X|D) \to {\Q}/{\Z}$ be a continuous character.
By part (1) of the lemma and \cite[Cor.~4.42]{Folland}, $\chi$ extends to a continuous
character $\chi' \colon \Pic^0(X'|D') \to \T$. It remains to show that the image of
$\chi'$ lies in ${\Q}/{\Z}$.
We write $F = F_1 \times F_2$, where $F_1$ is the cokernel of the diagonal inclusion
$k^\times \inj \stackrel{r}{\underset{i =1}\prod} k(x_i)^\times$ and
$F_2 = \stackrel{r}{\underset{i =1}\prod} \W_{n_i-1}(k(x_i))$.
   We then note using
  ~\eqref{eqn:CCM-0} and the proof of \lemref{lem:Kato-1} (see also
  \propref{prop:Pic-sequence}) that there is an exact sequence of topological groups
  \begin{equation}\label{eqn:Dual-surjection-3}
    0 \to F \to \Pic^0(X|D) \to \Pic^0(X) \to 0.
    \end{equation}
We claim that $\chi(F)$ is finite.

Since $F_2$ is a torsion group of bounded exponent, $\chi(F_2)$ must be finite.
For every $1 \le i \le r$, we have an exact sequence of topological groups
  \begin{equation}\label{eqn:Dual-surjection-4}
    0 \to \sO^\times_{k(x_i)} \to k(x_i)^\times \to \Z \to 0.
    \end{equation}  
Since $\sO^\times_{k(x_i)}$ is a profinite group, its image under $\chi$ must be finite.
In particular, it is a cyclic subgroup of ${\Q}/{\Z}$ of the type ${\Z}/n$ for
    some integer $n \ge 1$. We now look at the commutative diagram of exact sequences
    \begin{equation}\label{eqn:Dual-surjection-5}
      \xymatrix@C1pc{
        0 \ar[r] & \sO^\times_{k(x_i)} \ar[r] \ar[d] &  k(x_i)^\times \ar[r] \ar[d] & \Z 
        \ar[d] \ar[r] & 0 \\
        0 \ar[r] & \chi(\sO^\times_{k(x_i)}) \ar[r] & {\Q}/{\Z} \ar[r]^-{n} &
        {\Q}/{\Z} \ar[r] & 0.}
    \end{equation}
    Since the image of the right vertical arrow must be finite, it follows that
    $\chi(k(x_i)^\times)$ is finite. Summing over $1 \le i \le r$, we get that
    $\chi(F_1)$ is finite. This proves the claim.

 We let $D'_\red = \{y_1, \ldots , y_s\}$ and $D' = \sum_i m_i [y_i]$.
    We let $F' = F'_1 \times F'_2$, where
    $F'_1$ is the cokernel of the diagonal inclusion
  $k'^\times \inj \stackrel{s}{\underset{i =1}\prod} k(y_i)^\times$ and
  $F'_2 = \stackrel{s}{\underset{i =1}\prod} \W_{m_i-1}(k(y_i))$.
  We then have an exact sequence similar to ~\eqref{eqn:Dual-surjection-3}.
  Since $F'_2$ is a torsion group of bounded exponent, $\chi'(F_2)$ must be finite.
  On the other hand, each $\chi'(k(y_i)^\times)$ is finite by \lemref{lem:Lifting}.
  Summing over $1 \le i \le s$, we get that $\chi'(F'_1)$ is finite. We deduce that
  $\chi'(F')$ is finite.

Finally, we look at the commutative diagram of exact sequences
\begin{equation}\label{eqn:Dual-surjection-8}
      \xymatrix@C1pc{
        0 \ar[r] & F' \ar[r] \ar@{->>}[d]_-{\chi'} & \Pic^0(X'|D') \ar[r]^-{\alpha}
        \ar[d]^-{\chi'} &
      \Pic^0(X') \ar[r] \ar[d]^-{\wt{\chi}'} & 0 \\
            0 \ar[r] & {\Z}/{m'} \ar[r] & \T \ar[r]^-{m'} &
      \T \ar[r] & 0.}
  \end{equation}

It follows from \corref{cor:Pic-sequence-1} that
$\alpha$ is a quotient map. Since $m' \circ \chi'$ is continuous, it follows that
$\wt{\chi}'$ is a continuous homomorphism. On the other hand, $ \Pic^0(X')$ is
profinite. It follows that the image of $\wt{\chi}'$ is finite. We conclude that the
image of $\Pic^0(X'|D')$ under $\chi'$ is finite.
We have thus shown that $\chi' \in \Pic^0(X'|D')^\star$.
This proves (2) and concludes the proof of the lemma.
\end{proof}

\section{Perfectness of Brauer-Manin pairing for modulus pairs}
\label{sec:PBMP}
The goal of this section is to prove the perfectness of the Brauer-Manin pairing for
a one-dimensional modulus pair. We shall also prove \thmref{thm:Main-0} using the
perfectness of the Brauer-Manin pairing.
We let $k$ be a local field of exponential characteristic $p \ge 1$. Let $X$ be a
geometrically integral smooth projective curve over $k$ and let $D \subset X$ be an
effective divisor. We let $j \colon X^o \inj X$ be the inclusion of the complement of
$D$. We write $D = \sum_i n_i [x_i]$, where $D^\dagger = \{x_1, \ldots , x_r\}$ is the
support of $D$ with reduced closed subscheme structure.
We let $K$ denote the function field of $X$.
For $x \in X_{(0)}$, we let $X_x = \Spec(\sO_{X,x})$ and $X^h_x = \Spec(\sO^h_{X,x})$.

\subsection{Proof of \thmref{thm:Main-1}}\label{sec:Pf-1}
By  \propref{prop:BM-cont}, the Brauer-Manin pairing with modulus
(cf. \propref{prop:BMP-1}) induces continuous homomorphisms
$\beta_{X|D} \colon \Br(X|D) \to \CH_0(X|D)^\star$ and $\beta_{X^o} \colon \Br(X^o) \to
C(X^o)^\star$. We begin by showing the injectivity of these homomorphisms.

\begin{lem}\label{lem:BM-pf-0}
  $\beta_{X|D}$ is a monomorphism.
\end{lem}
\begin{proof}
 Let $w \in \Br(X|D)$ be such that $\chi_w := \beta_{X|D}(w) = 0$.
   For $x \in X_{(0)}$, let $\iota_x \colon \Spec(k(x)) \inj X$ be the
   inclusion. We let $w_x$ be the image of $w$ under the canonical composite map
  $\Br(X^o) \inj \Br(K) \to \Br(K_x)$. Let $\chi^1_w \colon \sZ_0(X^o) \to {\Q}/{\Z}$ and
  $\chi^2_w \colon {K^\times_\infty} \to {\Q}/{\Z}$ be the two components of $\chi_w$
  (see \S~\ref{sec:Pair**}). Our assumption implies that $\chi^i_w = 0$ for $i = 1,2$.

Suppose $x \in D$. Since $\chi^2_w = 0$, it follows from
  \cite[\S~6, Thm.~1]{Kato-cft-1} that $w_x = 0$.
The commutative diagram of exact sequences
  \begin{equation}\label{eqn:BM-pf-0-1}
    \xymatrix@C1pc{
      0 \ar[r] & \Br(X_x) \ar[r] \ar[d] & \Br(K) \ar[r] \ar[d] &
      H^3_{x}(X_x, \sO^\times_{X_x}) \ar[d]^-{\cong} \\
    0 \ar[r] & \Br(X^h_x) \ar[r] & \Br(K_x) \ar[r] &
    H^3_{x}(X^h_x, \sO^\times_{X^h_x})}
\end{equation}
implies that as an element of $\Br(K)$, the Brauer class $w$ lies in $\Br(X_x)$
and it dies in $\Br(X^h_x)$. In particular, $\iota^*_x(w) = 0$.
Since $\chi^1_w = 0$, we also have $\iota^*_x(w) = 0$ for every $x \in X^o_{(0)}$.
It follows that $w$ is a class in $\Br(X)$ such that $\iota^*_x(w) = 0$ for
$x \in X_{(0)}$. We conclude from \cite[Thm.~9.2]{Saito-Invent} that $w = 0$.
This finishes the proof.
\end{proof}

\begin{lem}\label{lem:Br-Main-00}
 We have the following.
\begin{enumerate}
  \item
If $p > 1$, the map $\beta_{X|D} \colon \Br(X|D) \to \CH_0(X|D)^\star$
is an isomorphism.
\item
  If $p =1$, the map $\beta_{X|D} \colon \Br(X|D) \to (\CH_0(X|D)^\star)_\tor$
is an isomorphism.
\end{enumerate}
\end{lem}
\begin{proof}
In view of \lemref{lem:BM-pf-0}, we only need to show that $\beta_{X|D}$ is surjective.
We first prove this surjectivity under the assumption that $D^\dagger \subset X(k)$.
  It follows from \lemref{lem:Brauer-reln} that
  $\Br(X|D^\dagger) = \Br(X|D) = \Br(X^o)$ when $p = 1$.
  Using \lemref{lem:Char-0-dual}, we can assume in this case that $D$ is reduced.
  Hence, $D$ will assumed to be reduced (and $n =1$) in the following argument
  when $p =1$.

We let $G^\star(X^o) = {\varinjlim}_n \ \CH_0(X|nD)^\star$.
  We shall first show that the map $\beta_{X^o} \colon \Br(X^o) \to G^\star(X^o)$
 is surjective. We consider the diagram
  \begin{equation}\label{eqn:Br-Main-01}
    \xymatrix@C.8pc{
      0 \ar[r] & \Br(X) \ar[r]^-{j^*} \ar[d]_-{\beta_X} & \Br(X^o) \ar[r]^-{\partial_X}
      \ar[d]^-{\beta_{X^o}} & {\underset{x \in D}\oplus} H^3_{x}(X, \sO^\times_X)
      \ar[d]^-{\gamma_{X|D}} \ar[r]^-{\eta_{X|D}} & H^3(X, \sO^\times_X) \ar[r]
      \ar[d]^-{\delta_{X}} & 0 \\
        0 \ar[r] & \CH_0(X)^\star \ar[r]^-{j^*} & G^\star(X^o) \ar[r] &
        {\underset{x \in D}\oplus} H^0(X^h_x, \sO^\times_{X^h_x})^\star
        \ar[r]^-{{(\iota^*_D)}^\star} & (k^\times)^\star \ar[r] & 0,}
      \end{equation}
where $\gamma_{X|D}$ is the direct sum of maps $\gamma_x$ given by
      \lemref{lem:Loc-dual-3} and $\delta_X$ is given by \lemref{lem:GD-units-7}.
The top row is the localization sequence for the {\'e}tale sheaf $\sO^\times_X$
and hence is exact except that we need to explain the surjectivity of $\eta_{X|D}$.
But this follows from \corref{cor:GD-units-main-3}.
The bottom row is a complex which is exact at $\CH_0(X)^\star$ and $G^\star(X^o)$
      by Lemmas~\ref{lem:Main-exact-00}, ~\ref{lem:Main-exact-003} and
      ~\ref{lem:Loc-dual-3}.

 The left square in ~\eqref{eqn:Br-Main-01} is commutative by
      \corref{cor:Mod-non-mod} and the right square is commutative by
      \lemref{lem:LD-GD}. To show the commutativity of the middle square,
 we look at the diagram
    \begin{equation}\label{eqn:Br-Main-02}
    \xymatrix@C.8pc{  
      \Br(X|nD) \ar[r] \ar[d]_{\beta_{X|nD}} & {\underset{x \in D}\oplus} \Br(K_x)
      \ar[d]^-{\sum_x \gamma_{K_x}} \\
      \CH_0(X|nD)^\star \ar[r] & {\underset{x \in D}\oplus} (K^\times_x)^\star}
  \end{equation}
  for $n \ge 1$, where the dual of $K^\times_x$ is taken with respect to its
  Kato topology. This diagram is commutative by the definition of $\beta_{X|D}$, where
  $\gamma_{K_x}$ is as in ~\eqref{eqn:Loc-dual-1}. When we restrict $\beta_{X^o}$ to
  $\Br(X|nD)$, the middle square in ~\eqref{eqn:Br-Main-01} commutes because
  it is the composition of ~\eqref{eqn:Br-Main-02} with the sum (over $x \in D$)
  of right squares in  ~\eqref{eqn:Loc-dual-1}. Taking the limit over $n \ge 1$,
  we see that the middle square in ~\eqref{eqn:Br-Main-01} commutes.

To show that $\beta_{X^o}$ is an isomorphism, we note that $\beta_X$ is an isomorphism
  by \cite[Thm.~4]{Lichtenbaum} and \cite[Thm.~9.2]{Saito-Invent}.
The arrow $\gamma_{X|D}$ is an isomorphism by \lemref{lem:Loc-dual-3}, and
  $\delta_X$ is an isomorphism by \thmref{thm:GD-units-main}.
An easy diagram chase shows that the bottom row of ~\eqref{eqn:Br-Main-01} is
  exact and $\beta_{X^o}$ is an isomorphism. This also proves that
  $\beta_{X|D}$ is surjective when $p =1$.
  To pass from $\beta_{X^o}$ to $\beta_{X|D}$ when $p >1$, we fix a character
  $\chi \in \CH_0(X|D)^\star$. It follows from the surjectivity of $\beta_{X^o}$ that
  there exists a class $w \in \Br(X^o) = {\underset{n}\bigcup} \ \Br(X|nD)$ such that
$\chi = \chi_w := \beta_{X^o}(w)$.  It follows from ~\eqref{eqn:Chow-Parshin-0} and
the definition of $\beta_{X|D}$ that the map
$\chi_w \colon K^\times_\infty \to {\Q}/{\Z}$ annihilates $I(X|D)$.
Using Definition~\ref{defn:BGM-0}, this forces $w$ to lie in $\Br(X|D)$. In this case,
we must have $\chi = \beta_{X|D}(w)$. This finishes the proof of surjectivity of
$\beta_{X|D}$ when $D^\dagger \subset X(k)$.

To prove the general case, we choose a finite field extension ${k'}/k$
such that ${\rm Supp}(D_{k'}) \subset X(k')$. We let $X' = X_{k'}$ and $D' = D_{k'}$.
We let $f \colon (X',D') \to (X,D)$ be the projection and consider the diagram
\begin{equation}\label{eqn:Br-Main-03}
    \xymatrix@C1pc{  
      \Br(X'|D') \ar[r]^-{\beta_{X'|D'}} \ar[d]_-{f_*} & \CH_0(X'|D')^\star
      \ar[d]^-{(f^*)^\star} \\
      \Br(X|D) \ar[r]^-{\beta_{X|D}} & \CH_0(X|D)^\star.}
  \end{equation}
 This diagram is commutative by \lemref{lem:BMP-2}. We have shown above that
  $\beta_{X'|D'}$ is surjective, and \lemref{lem:Dual-surjection} says that
  the right vertical arrow is surjective. It follows that $\beta_{X|D}$ is
  surjective. This concludes the proof.
\end{proof}

\vskip .2cm

{End of the proof of \thmref{thm:Main-1}:}
Combine Lemmas~\ref{lem:Prof-dual}, ~\ref{lem:Brauer-reln}, ~\ref{lem:BM-pf-0},
~\ref{lem:Br-Main-00} and \propref{prop:BM-cont}.
$\hfill\qed$

\vskip .2cm

As an application of \thmref{thm:Main-1}, we get the following result about the
norm map between the Brauer groups of regular curves over local fields.

\begin{thm}\label{thm:Br-Main-04}
  Let $X$ be a geometrically connected regular quasi-projective curve over a local
  field $k$ and let $k'$ be a finite field extension of $k$. Let $f \colon X_{k'} \to X$
  be the projection. Assume that $X$ is either smooth or affine.
  Then the map $f_* \colon \Br(X_{k'}) \to \Br(X)$ is surjective.
\end{thm}
\begin{proof}
 We choose an open embedding $j \colon X \inj \ov{X}$, where $\ov{X}$ is a connected
  regular projective curve. Then $\ov{X}$ is necessarily geometrically connected.
  We let $D = \ov{X} \setminus X$ with the reduced closed subscheme structure.
  We let $\ov{X}' = \ov{X}_{k'}, \ X' = X_{k'}$ and $D' = D_{k'}$. If $\ov{X}$ is smooth
  (e.g., when $p =1$), we conclude the proof by combining \propref{prop:BGM-2},
  \lemref{lem:Dual-surjection} and \thmref{thm:Main-1}. We shall
  now assume that $p > 1$ and $X$ is affine.

If ${k'}/k$ is purely inseparable, then
  $\coker(f_*)$ is annihilated by some power of $p$
(e.g., see \cite[\S~3.8]{CTS}). On the other hand, $\coker(f_*)$ is $p$-divisible
by \cite[Thm.~3.2.3]{CTS}. It follows that $\coker(f_*) = 0$.
In general, we can get a factorization $X_{k'} \to X_{k''} \to X$, where
${k''}/k$ is separable and ${k'}/{k''}$ is purely inseparable.
Since $X_{k''}$ is a geometrically connected regular affine curve,
the map $\Br(X_{k'}) \to \Br(X_{k''})$ is surjective.
We can thus assume that ${k'}/k$ is separable.

As in the proof of \propref{prop:BM-cont}, we can find a finite field extension
  $l/k$, a geometrically connected smooth projective curve $\ov{Y}$ over $l$, and a
  finite surjective morphism of $k$-schemes $g \colon \ov{Y} \to \ov{X}$ such that
  $\deg(g)$ and $[l:k]$ are some powers of $p$. We let $Y = g^{-1}(X)$ and
  $E = g^*(D)$. We let $A = l \otimes_k k'$ and $\ov{Y}_A = \ov{Y} \times_{\Spec(l)}
  \Spec(A)$. Since ${k'}/k$ is separable, we see that
  $A = \stackrel{r}{\underset{i =1}\prod} l_i$, where each ${l_i}/l$ is a finite
  separable field extension.

 We consider the diagrams
  \begin{equation}\label{eqn:Br-Main-05}
    \xymatrix@C1pc{
      \stackrel{r}{\underset{i =1}\prod} \ov{Y}_{l_i} \ar[r]^-{g'} \ar[d]_-{h} & \ov{X}'
      \ar[d]^-{f} & &  \stackrel{r}{\underset{i =1}\oplus} \Br({Y}_{l_i})
      \ar[r]^-{g'_*} \ar[d]_-{h_*} & \Br({X}')  \ar[d]^-{f_*} \\
      \ov{Y} \ar[r]^-{g} & \ov{X} & & \Br({Y}) \ar[r]^-{g_*} & \Br(X),}
  \end{equation}
  where the left square is Cartesian whose all arrows are finite.
  The right square is commutative by \propref{prop:BGM-2}.

We now proceed as follows. Since $\ov{Y}$ is a geometrically connected smooth projective
  curve over $l$, we have argued previously that $h_* \colon \Br(Y_{l_i}) \to \Br(Y)$ is
  surjective for each $i$. Since $\deg(g) = p^n$ for some $n \ge 1$, we have also
  seen previously that $g_* \colon \Br(Y) \to \Br(X)$ is surjective. In particular,
  $f_* \circ g'_* = g_* \circ h_*$ is surjective. It follows that $f_*$ is surjective.
  This concludes the proof.
\end{proof}

\subsection{Proof of \thmref{thm:Main-0}}\label{sec:Pf-0}
We let the notations and assumptions be as stated in the beginning of
\S~\ref{sec:PBMP}. We shall divide the proof of \thmref{thm:Main-0} into several
cases. We begin with the following case.
Recall that $H^q_{cc}(X^o, \G_m) = {\varprojlim}_n H^q(X, \G_{m, (X,nD)}) \cong
 {\varprojlim}_n H^q(X, \sK^M_{1, (X,nD)})$.

\begin{lem}\label{lem:Main-pf-0}
  One has $H^0_{cc}(X^o, \sO^\times_{X^o}) = H^i(X^o, \sO^\times_{X^o}) = 0$ for all
  $i \ge 3$.
\end{lem}
\begin{proof}
It is easy to see that $H^0(X, \sK^M_{1, (X,nD)}) = 0$ for every
$n \ge 1$. In particular, $H^0_{cc}(X^o, \sO^\times_{X^o}) = 0$. To show 
that $H^i(X^o, \sO^\times_{X^o}) = 0$ for $i \ge 3$ is equivalent to show that
$H^i(X^o, \sO^\times_{X^o})_\tor = 0$ for all $i \ge 3$ (e.g., see \cite[Lem.~3.5.3]{CTS}).
It already follows from Remark~\ref{remk:Vanish-open-p} that
$H^i(X^o, \sO^\times_{X^o})\{p\} = 0$ for all $i \ge 3$.

We now show the prime-to-$p$ case.
For $i \ge 4$, we have a surjection $H^i(X^o, {\Q}/{\Z}\{p'\}(1)) \surj
H^i(X^o, \sO^\times_{X^o})\{p'\}$. On the other hand, the term on the left of this
surjection is zero because $cd(X^o) \le 3$. For $i =3$, note that there is an
inclusion $H^3(X^o, \sO^\times_{X^o})\{p'\} \inj H^{4}_{D}(X, \sO^\times_X)\{p'\}$ by
\corref{cor:GD-units-main-3}.
On the other hand, there is a canonical surjection
$H^{4}_{D}(X, {\Q}/{\Z}\{p'\}(1)) \surj H^{4}_{D}(X, \sO^\times_X)\{p'\}$.
Using Gabber's purity, the term on the left of this surjection
is isomorphic to ${\underset{x \in D}\bigoplus}  H^{2}(k(x),  {\Q}/{\Z}\{p'\})$.

We are thus reduced to showing that $H^2(k',  {\Q}/{\Z}\{p'\}) = 0$ for $i \ge 2$ if
${k'}/k$ is any finite field extension. By \cite[Thm.~9.9]{GKR}, we have an isomorphism
$H^2(k',  {\Q}/{\Z}\{p'\}) \cong {\varinjlim}_{p \nmid n}
\Hom_{\Ab}(\mu_n(k'), {\Q}/{\Z})$, where the limit is taken with respect to
the maps $\mu_{mn}(k') \xrightarrow{m} \mu_n(k')$.
But this limit is zero since $({k'}^\times)\{p'\}$ is finite.
This concludes the proof.
\end{proof}

Recall that there is an exact sequence
\begin{equation}\label{eqn:Main-pf-1}
  {\Z}^r \to \CH_0(X) \xrightarrow{j^*} \CH_0(X^o) \to 0,
\end{equation}
where the first arrow is the sum of cycle class maps for points in $D$.
We endow $\CH_0(X^o)$ with the quotient of the adic topology of $\CH_0(X)$
via this exact sequence. Note that this topology of $\CH_0(X^o)$ is
independent of any choice of a regular compactification of $X^o$.
We shall consider $H^2_{cc}(X^o, \sO^\times_{X^o})$ to be a discrete abelian group.

\begin{lem}\label{lem:Main-pf-2}
  There is a perfect pairing of topological abelian groups
  \[
    \CH_0(X^o)^\pf \times H^2_{cc}(X^o, \sO^\times_{X^o}) \to {\Q}/{\Z}.
  \]
\end{lem}
\begin{proof}
  We look at the diagram
  \begin{equation}\label{eqn:Main-pf-2-0}
    \xymatrix@C1pc{
    0 \ar[r] & H^2_{cc}(X, \sO^\times_{X^o}) \ar[r] \ar@{.>}[d] &
    \Br(X) \ar[r] \ar[d]^-{\beta_X} & \Br(D^\dagger) \ar[d]^-{\beta_{D^\dagger}} \\
    0 \ar[r] & \CH_0(X^o)^\star \ar[r]^-{(j^*)^\star} & \CH_0(X)^\star \ar[r] &
    ({\Q}/{\Z})^r,}
\end{equation}
where the right vertical arrow is the sum (over $x \in D^\dagger$)
of the maps ${\rm inv}_{k(x)}$.
It follows from the exactness of the top row that the canonical map
$H^2_{cc}(X, \sO^\times_{X^o}) \to  H^2(X, \sK^M_{1,(X,D^\dagger)})$ is an
isomorphism.

It is immediate from the construction of the Brauer-Manin pairing ~\eqref{eqn:BMP-0}
(with $D = \emptyset$) that the right square in ~\eqref{eqn:Main-pf-2-0} is commutative.
The middle and the right vertical arrows are isomorphisms. It follows that there
is a unique isomorphism $\beta^1_{X^o} \colon H^2_{cc}(X^o, \sO^\times_{X^o})
\xrightarrow{\cong} \CH_0(X^o)^\star$ such that ~\eqref{eqn:Main-pf-2-0} is commutative.
Since $\beta^1_{X^o}$ is automatically continuous, we get
a continuous pairing $\CH_0(X^o)^\pf \times H^2_{cc}(X^o, \sO^\times_{X^o}) \to
{\Q}/{\Z}$. Since $\CH_0(X)^\star$ is a torsion group, $\CH_0(X^o)^\star$ must also be
torsion. We can now apply \lemref{lem:Prof-dual} and the Pontryagin duality 
between profinite and discrete torsion groups to conclude the proof.
\end{proof}

\vskip .3cm

{End of the proof of \thmref{thm:Main-0}:}
Combine \thmref{thm:Main-1} with Lemmas~~\ref{lem:Prof-dual}, \ref{lem:Brauer-reln},
~\ref{lem:Main-pf-0} and ~\lemref{lem:Main-pf-2}.
$\hfill\qed$

\begin{cor}\label{cor:Open-perf-pairing}
  There is a perfect pairing of topological abelian groups
  \[
    \Br(X^o) \times C(X^o)^{\pf} \to {\Q}/{\Z}.
  \]
\end{cor}
\begin{proof}
  Combine \thmref{thm:Main-0} and \lemref{lem:Limit-nD}.
  \end{proof}

\subsection{Duality for the cohomology of $\G_{m,K}$}
\label{sec:Dual-K}
Let the notations and assumptions be as stated in the beginning of
\S~\ref{sec:PBMP}. Recall that $K$ denotes the function field of $X$.
We let $H^q_{cc}(K, \G_m) := {\underset{D \in \Div(X)}\varprojlim} H^q(X,
\G_{m, (X,D)})$. We let $C(K)$ be the cokernel of the canonical map
$K^\times \xrightarrow{\alpha} \wh{I}(X) := {\underset{x \in X_{(0)}}{\prod'}}
\wh{K}^\times_x$, where $\alpha$ is the canonical inclusion (cf. ~\eqref{eqn:BMP-1-2}).
We endow $H^1_{cc}(K, \G_m)$ with the inverse limit of the adic topologies
of $\Pic(X|D)$ for $D \in \Div(X)$, and $\Br(K)$ with the discrete topology.

To prove the duality theorem for the {\'e}tale cohomology of
$\G_{m,K}$, we need the following.

\begin{lem}\label{lem:Limit-U}
  For every effective divisor $D \subset X$ with complement $X^o$, the canonical maps
  $H^1_{cc}(K, \G_m) \to H^1_{cc}(X^o, \G_m) \to H^1(X, \G_{m, (X,D)})$ are surjective.
\end{lem}
\begin{proof}
  Using \lemref{lem:Limit-nD}, it suffices to show that for every reduced effective
  divisor $D \subset X$ with $X^o = X \setminus D$, the map
  $C(K) \to C(X^o)$ is surjective, and the map $C(K) \to {\varprojlim}_{D} C(X^o)$ is
  bijective, where the limit is over all reduced effective divisors $D \subset X$.

  We have a commutative diagram of exact sequences
  \begin{equation}\label{eqn:Limit-0-0}
    \xymatrix@C.8pc{
      0 \ar[r] & K^\times \ar[r]^-{\alpha} \ar@{=}[d] & \wh{I}(X) \ar[r]
      \ar[d]^-{\beta_D} &
      C(K) \ar[d]^-{\gamma_D} \ar[r] & 0 \\
      0 \ar[r] & K^\times \ar[r]^-{\alpha_D} & \sZ_0(X^o) \oplus B(D)
      \ar[r] & C(X^o) \ar[r] & 0,}
  \end{equation}
  where we let $B(D):= \wh{K}^\times_\infty = {\prod}_{x \in D} \wh{K}^\times_x$.
  We let $A(D) = \sZ_0(X^o) \oplus B(D)$.
  It is straightforward to check that each $\beta_D$ is surjective.
  In particular, each transition map of the cofiltered system $\{A(D)\}$
  is surjective. It remains to show that the map
  $\beta := {\varprojlim}_D \beta_D \colon \wh{I}(X) \to {\varprojlim}_D A(D)$ is
  bijective.

It is easy to see that ${\prod}_{x \in X_{(0)}} \wh{K}^\times_x
  \xrightarrow{\cong} {\varprojlim}_D B(D)$. 
  We let $\pi_D \colon A(D) \to B(D)$ denote the projection.
  For $D' > D$, the transition map $\Phi_{D'>D} \colon A(D') \to A(D)$
  has the property that if $(a_1, a_2, a_3) \in A(D')$ with
  $a_1 \in \sZ_0(X \setminus D'), \ a_2 \in {\prod}_{x \in D} \wh{K}^\times_x$ and
$a_3 \in {\prod}_{x \in D' \setminus D} \wh{K}^\times_x$, then
\begin{equation}\label{eqn:Limit-0-1}
\Phi_{D'>D}(a_1) = a_1, \ \Phi_{D'>D}(a_2) = a_2 \ \mbox{and} \ \Phi_{D'>D}(a_3) =
{\prod}_{x \in D' \setminus D} v_x(a_3).
\end{equation}
It follows that $\pi_{D} \circ \Phi_{D'>D} = \Phi_{D'>D} \circ \pi_{D'}$.
That is, the projection $\{\pi_D\}  \colon \{A(D)\} \to  \{B(D)\}$
is a strict morphism of pro-abelian groups. Taking the limits, we get a
commutative diagram
\begin{equation}\label{eqn:Limit-0-2}
  \xymatrix@C1pc{
    {\varprojlim}_D A(D) \ar[r]^{\pi} & {\prod}_{x \in X_{(0)}} \wh{K}^\times_x \\
    & \wh{I}(X) \ar[lu]^-{\beta} \ar[u]_-{\beta'},}
\end{equation}
where $\beta'$ is the canonical inclusion
$\wh{I}(X) \inj {\prod}_{x \in X_{(0)}} \wh{K}^\times_x$. In particular, $\beta$ is
injective.

We claim that $\pi$ is injective. Indeed, if $(a_D) \in {\varprojlim}_D A(D)$
is an element indexed by reduced divisors $D \subset X$
with $(a_D) = (a_0, a_1)$ (where $a_0 \in \sZ_0(X\setminus D)$ and $a_1 \in B(D)$)
such that $\pi((a_D)) = 0$, then we must have $a_1 =0$ for each reduced divisor
$D \subset X$. That is, $(a_D) \in {\varprojlim}_D \sZ_0(X \setminus D)
= {\bigcap}_D \sZ_0(X \setminus D)$. The claim now follows because it is not hard to
see that ${\bigcap}_D \sZ_0(X \setminus D) =0$.

To prove that $\beta$ is surjective, let $(a_D) \in {\varprojlim}_D A(D)
\subset {\prod}_{x \in X_{(0)}} \wh{K}^\times_x$. Let $b = \pi((a_D))$ and write
$b = (b_x)_{x \in X_{(0)}}$. If $v_x(b_x) \neq 0$ for infinitely many
$x \in X_{(0)}$, then it follows from ~\eqref{eqn:Limit-0-1} that
the projection map ${\varprojlim}_D A(D) \to A(D)$ is not defined at $b$ for any
reduced divisor $D \subset X$. But this is absurd. We conclude that
$b = \pi((a_D)) \in \wh{I}(X)$. This shows that $\beta$ is surjective,
and concludes the proof.
\end{proof}

Using \lemref{lem:Limit-U}, we endow $C(K)$ with the inverse limit of the 
adic topologies of $\Pic(X|D)$ for $D \in \Div(X)$.
We can now prove the duality theorem for the cohomology of $\G_{m,K}$.

\begin{thm}\label{thm:Main-6}
For every integer $q \neq 0$, there is a bilinear pairing
  \[
    H^q(K, \G_m) \times H^{3-q}_{cc}(K, \G_m) \to {\Q}/{\Z}
  \]
  which induces perfect pairings of topological abelian groups
  \[
    H^1(K, \G_{m})^{\pf} \times H^2_{cc}(K, \G_{m}) \to {\Q}/{\Z},
    \]
    \[
      H^2(K, \G_{m}) \times H^1_{cc}(K, \G_{m})^{\pf} \to {\Q}/{\Z}
      \]
      and
      \[
      H^3(K, \G_{m}) \times H^0_{cc}(K, \G_{m})^{\pf} \to {\Q}/{\Z}.  
      \]
      \end{thm}
      \begin{proof}
The existence of the pairing is an immediate consequence of  
\thmref{thm:Main-0} and a limit argument (over $D \in \Div(X)$).
The perfectness of the third pairing follows from \lemref{lem:Main-pf-0}
using a limit argument. The perfectness of the second pairing follows
by applying \thmref{thm:Main-1}, Lemmas~~\ref{lem:Prof-dual},
~\ref{lem:dual-surj}, \ref{lem:Brauer-reln} and ~\ref{lem:Limit-U},
and taking limit over $\Div(X)$.
To prove the perfectness of the first pairing, we only need to show that
$H^2_{cc}(K, \G_{m}) = 0$. But this follows by taking limit over $\Div(X)$
of the top exact sequence in ~\eqref{eqn:Main-pf-2-0} and observing that
the map $\Br(X) \to {\prod}_{x \in X_{(0)}} \Br(k(x)) \cong {\varprojlim}_D \Br(D_\red)$ is
injective by \lemref{lem:BM-pf-0} (with $D = \emptyset$).
\end{proof}

By ~\eqref{eqn:BMP-1-2}, we have a bilinear pairing
$\Br(K) \times \wh{I}(X) \to {\Q}/{\Z}$, and the proof of \propref{prop:BMP-1}
shows that this factors through $\Br(K) \times C(K) \to {\Q}/{\Z}$.
The following result is the global version of \cite[Thm.~2.10, Cor.~2.12]{Saito-Invent}
and provides an explicit description of $\Br(K)$.

\begin{cor}\label{cor:Main-6-0}
 The above pairing induces a perfect pairing of topological abelian groups 
 \[
   \Br(K) \times C(K)^{\pf} \to {\Q}/{\Z}.
 \]
 In particular, $\Br(K) \xrightarrow{\cong} (C(K)^\star)_\tor$.
\end{cor}

\section{Brauer-Manin pairing for singular curves}\label{sec:BMS}
  In this section, we shall apply \thmref{thm:Main-1} to extend the Brauer-Manin
  pairing for smooth projective curves {\`a} la Lichtenbaum-Saito to singular curves
  over local fields.
  To achieve this, we shall introduce a refined version of the Brauer group for
  singular varieties. We shall then show that this refined group is directly related
  to the Picard group for singular curves.
We fix a local field $k$ of exponential characteristic $p \ge 1$.

\subsection{Levine-Weibel Brauer group of singular varieties}
  \label{sec:LWBG}
  Before we introduce the refined Brauer group for singular quasi-projective varieties,
  we recall some known facts about the completions of local
  rings. For any semilocal ring $R$, let $\wh{R}$ denote the completion of $R$ with
  respect to its Jacobson radical.
 Let $A$ be a local integral domain which is essentially of finite type over a
 field and let $B$ denote the integral closure of $A$. We let $\fm_A$ (resp. $\fm_B$)
 denote the maximal ideal of $A$ (resp. the Jacobson radical of $B$).
 It is an elementary fact (e.g., see \cite[Chap.~10]{AM})
 that $\wh{A} \inj \wh{B} \cong B \otimes_A \wh{A} \cong \prod_i \wh{B_{\fm_i}}$,
 where the product is over all maximal ideals of $B$.
 Furthermore, $\wh{B}$ coincides with the normalization of $\wh{A}$ if $\dim(A) = 1$.  In
 particular, it is a product of discrete valuation rings $\wh{B_{\fm_i}}$.
 In the latter case, it is also easy to see using the faithfully flatness property of
 $A \inj \wh{A}$ that $\wh{A} \cap K = A$, where $K$ is the quotient field of $A$.

We let $X$ be an integral quasi-projective curve over $k$.  Let $\pi \colon
   {X'} \to X$ be the normalization of $X$ and let $K$ denote the function field of $X$.
   For every closed point $x \in X$
 and every $y \in \pi^{-1}(x)$, let $U^0_{x,y} \subset \wh{K}_y$ be the image of
   $(\wh{\sO_{X,x}})^\times$ under the composite map
  $\wh{\sO_{X,x}} \inj \wh{\sO_{X',y}} \inj \wh{K}_y$.

 \begin{defn}\label{defn:LWB-0}
     We let $\Br^{\lw}(X)$ be the subgroup of $\Br(X_\reg)$ consisting of elements $w$
     such that for every singular point $x \in X$ and every $y \in \pi^{-1}(x)$, the
     image of $w$ under the canonical map $\Br(X_\reg) \inj \Br(K) \to \Br(\wh{K}_y)$ has
     the property that the associated character (under Kato's pairing)
     $\chi_w \colon \wh{K}^\times_y \to {\Q}/{\Z}$ annihilates $U^0_{x,y}$.
\end{defn}

Let $X$ be an integral quasi-projective variety over $k$.
  Let $\sC^\lci(X)$ be the set of all integral curves $C \subset X$
  such that $C$ is not contained in $X_\sing$ and the inclusion $C \inj X$ is a local
  complete intersection at every point of $C \cap X_\sing$. 
 \begin{defn}\label{defn:LWB-1}
We let $\Br^{\lw}(X)$ be the subgroup of $\Br(X_\reg)$ consisting of elements $w$
such that for every $C \in \sC^\lci(X), \ x \in X_\sing \cap C$
and $y \in \nu^{-1}(x)$, the image of $w$ under the canonical composite map
$\Br(X_\reg) \to \Br(\nu^{-1}(X_\reg)) \to \Br(\wh{k(C)}_y)$ has
     the property that the associated character (under Kato's pairing)
     $\chi_w \colon \wh{k(C)}^\times_y \to {\Q}/{\Z}$ annihilates $U^0_{x,y}$,
     where $\nu \colon C_n \to X$ is the canonical map.
\end{defn}

It easily follows from the above definitions that $\Br^{\lw}(X) \subset \Br^{\lw}(U)$
if $U \subset X$ is open. This is in contrast with the classical Brauer group of
singular varieties. It is also clear that $\Br^{\lw}(X) = \Br(X)$ if $X$ is regular.
We shall refer to $\Br^{\lw}(X)$ as the `Levine-Weibel Brauer group' of $X$.

\subsection{The Brauer-Manin pairing}\label{sec:BMP-sing}
For the rest of \S~\ref{sec:BMS}, we shall work with the following set-up.
  Let $X$ be a geometrically integral projective curve over $k$ and $\pi \colon
  X_n \to X$ the normalization of $X$. It is clear that ${X}_n$ is geometrically
  integral. We shall assume that ${X}_n$ is smooth over $k$ (this is automatic if
  $p =1$).
  We let $K$ denote the function field of $X$. We let $S$ be the singular locus of $X$
  with the reduced closed subscheme structure
  and let $E$ be the scheme theoretic inverse image of $S$ in ${X}_n$. We let
  $X^o = X_\reg$ be the regular locus of $X$. We represent this datum in the
   Cartesian squares
  \begin{equation}\label{eqn:BMS-0}
    \xymatrix@C1pc{
      E \ar[r]^-{\iota'} \ar[d]_-{\pi'} & {X}_n \ar[d]^-{\pi} & X^o \ar[l]_-{j'}
      \ar@{=}[d] \\
      S \ar[r]^-{\iota} & X & X^o \ar[l]_-{j}.}
  \end{equation}

We have recalled the definition of the Levine-Weibel Chow group $\CH^{\lw}_0(X)$
  and the isomorphism $\cyc_X \colon \CH^{\lw}_0(X) \xrightarrow{\cong} \Pic(X)$
  after \lemref{lem:Base-extn}. We let $\CH^{\lw}_0(X)_0 =
  \Ker(\deg \colon \CH^{\lw}_0(X) \to \Z) \cong \Pic^0(X)$.
  We shall use $\Pic(X)$ and $\CH^{\lw}_0(X)$ interchangeably.
In order to construct a duality between the Levine-Weibel Brauer group and
  the Picard group of $X$, we need to give a new description of the latter group.

For a point $x \in S$, we let
  $E_x = \pi^{-1}(x)$ and $\wh{K}_{x,\infty} = {\underset{y \in E_x}\prod} \wh{K}_y$.
  We let $U^0_{x,\infty} = {\underset{y \in E_x}\prod} U^0_{x,y}$.
  We have the canonical maps $K^\times \inj  \wh{K}^\times_{x,\infty} \surj
  {\wh{K}^\times_{x,\infty}}/{U^0_{x,\infty}}$. We let $\delta_x$ be the composite map.
  Let $\delta = {\underset{y \in E_x}\prod} \delta_x$,

\begin{lem}\label{lem:LW-0}
    There is a canonical exact sequence
 \begin{equation}\label{eqn:LW-3}
     0 \to k^\times \to  K^\times \xrightarrow{(\divf, \delta)}
      \sZ_0(X^o) \bigoplus \left({\underset{x \in S}\prod} \frac{\wh{K}^\times_{x, \infty}}
        {U^0_{x,\infty}}\right) \to \CH^{\lw}_0(X) \to 0.
\end{equation}
    \end{lem}
    \begin{proof}
We consider the diagram
      \begin{equation}\label{eqn:LW-1}
        \xymatrix@C.8pc{
          \sO^\times_{X,S} \ar[r]^-{\divf} \ar[d] & \sZ_0(X^o) \ar[r] \ar[d] & \CH^{\lw}_0(X)
          \ar@{.>}[d] \ar[r] & 0 \\
     K^\times \ar[r]^-{(\divf, \delta)} &
      {\sZ_0(X^o) \bigoplus \left({\underset{x \in S}\prod} \frac{\wh{K}^\times_{x, \infty}}
          {U^0_{x,\infty}}\right)} \ar[r] &  \ov{\CH^{\lw}_0(X)} \ar[r] & 0,}
  \end{equation}
  where the vertical arrows are the canonical inclusions and
  $\ov{\CH^{LW}_0(X)}$ is defined to make the bottom row exact.
The top row is exact by definition of $\CH^{\lw}_0(X)$.
It is clear that the left square is commutative. This yields a unique
homomorphism $\psi_X \colon \CH^{\lw}_0(X) \to \ov{\CH^{\lw}_0(X)}$ such that the
right square also commutes. 

To prove that $\psi_X$ is bijective, it is equivalent to show that 
\begin{equation}\label{eqn:LW-2}
  0 \to \sO^\times_{X,S} \to K^\times \xrightarrow{\delta} {\underset{x \in S}\prod}
  \frac{\wh{K}^\times_{x, \infty}} {U^0_{x,\infty}} \to 0
  \end{equation}
is exact.
Now, it is clear that if $f \in K^\times$ is such that $\delta_x(f) = 0$ for every
$x \in S$, then $f \in {\underset{x\in S}\bigcap} \sO^\times_{X,x}$ (see the first
paragraph of \S~\ref{sec:LWBG}). But it is not
hard to see that $\sO^\times_{X,S} = {\underset{x\in S}\bigcap} \sO^\times_{X,x}$,
using the fact that $\sO^\times_{X,x}$ is the set of rational functions on $X$ which are
regular without a zero or a pole in a neighborhood of $x$. It remains to show that
$\delta$ is surjective.

We let $I_x \subset \wh{\sO_{X,x}}$ be a conductor ideal for the normalization
$\wh{\sO_{X,x}} \inj \wh{\sO_{X_n, E_x}} = {\underset{y \in E_x}\prod} \wh{\sO_{X_n,y}}$.
Then $I_x \surj I_x \wh{\sO_{X_n,y}} \subset \wh{\fm_y}$ for every $y \in E_x$, where the
latter is the maximal ideal of $\wh{\sO_{X_n,y}}$. It is also clear that
$\sqrt{I_x \wh{\sO_{X_n,y}}} = \wh{\fm_y}$. As
$I_x \xrightarrow{\cong} {\underset{y \in E_x}\prod} I_x \wh{\sO_{X_n,y}}$, it follows
that $(1+I_x) \xrightarrow{\cong} {\underset{y \in E_x}\prod} (1 + I_x \wh{\sO_{X_n,y}})$.
In particular, ${\underset{y \in E_x}\prod} (1 + \wh{\fm_y}^n) \subset U^0_{x, \infty}$
for all $n \gg 0$. In other words, we have a factorization
\begin{equation}\label{eqn:LW-4}
  K^\times \to {\underset{x \in S}\prod} \frac{\wh{K}^\times_{x, \infty}}{U^n_{x,\infty}}
  \surj {\underset{x \in S}\prod} \frac{\wh{K}^\times_{x, \infty}}{U^0_{x,\infty}}
\end{equation}
for all $n \gg 0$, where we let
$U^n_{x,\infty} = {\underset{y \in E_x}\prod} (1 + \wh{\fm_y}^n)$.
We are now done because the left arrow is well known to be surjective as an application
of the approximation lemma in number theory (e.g., see \cite[Lem.~6.3]{Kerz-MRL}).
This proves the exactness of ~\eqref{eqn:LW-2}.

To conclude the proof of the lemma, it remains to show that
$k^\times = \Ker(\divf) \bigcap \Ker(\delta)$. But this is easy because we have
shown that $\Ker(\divf) \bigcap \Ker(\delta) = \Ker(\divf) \bigcap \sO^\times_{X,S}$
and the latter is $k^\times$ because $X$ is geometrically integral.
\end{proof}

We now define the Brauer-Manin pairing for $X$ as follows.
Let $w \in \Br^{\lw}(X)$. We define the associated character
$\chi_w \colon \sZ_0(X^o) \bigoplus ({\underset{x \in S}\prod} \wh{K}^\times_{x, \infty})
\to {\Q}/{\Z}$ exactly as we did to define the pairing ~\eqref{eqn:BMP-0*} in
\S~\ref{sec:Pair**}. It follows from the definition of $\Br^{\lw}(X)$ that
$\chi_w$ factors through $\sZ_0(X^o) \bigoplus {\underset{x \in S}\prod}
\frac{\wh{K}^\times_{x, \infty}}{U^0_{x,\infty}}$. The proof of the claim that
$\chi_w \circ (\divf, \delta) = 0$ is identical to the proof of \propref{prop:BMP-1},
mutatis mutandis.  We have thus shown the following.

\begin{prop}\label{prop:BMP-sing-M}
  There exists a canonical  bilinear pairing
  \begin{equation}\label{eqn:BMP-sing-0}
    \Br^{\lw}(X) \times \CH^{\lw}_0(X) \to {\Q}/{\Z}.
    \end{equation}
\end{prop}

\subsection{Perfectness of the pairing}\label{sec:Br-sing-main}
We now state and prove the main result of \S~\ref{sec:BMS}. We shall continue to work
with the set-up described in \S~\ref{sec:BMP-sing}.
Since $X$ is geometrically integral, it follows from \cite[Thm.~9.2.5, 9.4.8]{Kleiman}
that there is a group scheme $\Picc(X)$ over $k$ which is locally of finite type
and whose identity component $\Picc^0(X)$ is a smooth quasi-projective group scheme over
$k$, and $\Picc(X)(k) \cong \Pic(X)$. One also has that
$\Picc(X_{k'}) \cong \Picc(X)_{k'}$
for every field extension ${k'}/k$. It follows that $\Pic(X)$ is equipped with its
canonical adic topology and $\Pic^0(X)$ is its open subgroup. We shall let
$\Br^{\lw}(X)$ have the discrete topology.

\begin{lem}\label{lem:BMP-sing-cont}
The pairing ~\eqref{eqn:BMP-sing-0} is continuous.
\end{lem}
\begin{proof}
It suffices to show that the image of the induced homomorphism
  $\beta_X \colon \Br^{\lw}(X) \to \Pic(X)^\vee$ lies in $\Pic(X)^\star$
  (cf. \S~\ref{sec:Duality-X}).
To that end, we choose a conductor ideal sheaf
  $\sI_Z \subset \sO_X$ associated to $\pi$ so that we have an isomorphism
  $(1 + \sI_{nZ}) \xrightarrow{\cong} \pi_*((1 + \sI_{nD}))$ for every $n \ge 1$, where
  we let $Z \subset X$ be the closed subscheme defined by
  $\sI_Z$ and $D$ the scheme theoretic inverse image of $Z$ under $\pi$.
  It follows from \corref{cor:Pic-sequence-1} that the canonical map
  $\Pic(X_n|D) \surj \Pic(X_n)$ has a factorization
  $\Pic(X_n|D) \surj \Pic(X) \surj \Pic(X_n)$ in which all maps are
topological quotient maps between adic spaces.

We now let $w \in \Br^{\lw}(X)$ and let $\chi_w \colon \Pic(X) \to {\Q}/{\Z}$ be
the associated character. Since $w \in \Br(X^o)$, it lies in $\Br(X_n|D)$ for some
conductor closed subscheme $Z \subset X$ and $D = Z \times_X X_n$ as above.
It is easily seen that the composite map $\chi_w \colon \Pic(X_n|D) \surj \Pic(X) \to
{\Q}/{\Z}$ is same as the one in ~\eqref{eqn:BMP-1-1} (see Remark~\ref{remk:BMP-compln}).
\propref{prop:BM-cont} then implies that this composite map is continuous.
Since $\Pic(X_n|D) \surj \Pic(X)$ is a topological quotient, it follows that
$\chi_w$ is continuous on $\Pic(X)$. This concludes the proof.
\end{proof}

Let the notations be as in the proof of \lemref{lem:BMP-sing-cont}.
From ~\eqref{eqn:LW-4}, one also deduces the following.

\begin{lem}\label{lem:LW-5}
  For all $m \gg 0$, we have $\Br^\lw(X) \subseteq \Br(X_n|mD)$.
\end{lem}

\begin{thm}\label{thm:BMP-sing-M2}
  The pairing ~\eqref{eqn:BMP-sing-0} induces a perfect pairing of
  topological abelian groups
  \[
    \Br^{\lw}(X) \times \CH^{\lw}_0(X)^\pf \to {\Q}/{\Z}.
  \]
\end{thm}
\begin{proof}
The continuity of the pairing follows from Lemmas~\ref{lem:Prof-dual} and
  ~\ref{lem:BMP-sing-cont}. As in the proof of \thmref{thm:Main-1}, it
  suffices to show that the map
  $\beta_X \colon  \Br^{\lw}(X) \to F$ is bijective if we let
  $F = \CH^{\lw}_0(X)^\star$ when $p > 1$ and $F =  (\CH^{\lw}_0(X)^\star)_\tor$ when $p =1$.

Using \lemref{lem:LW-5}, we choose conductor
  subschemes $Z \subset X$ and $D = Z \times_X X_n \subset X_n$ such that
  $\Br^\lw(X) \subseteq \Br(X_n|D)$.
 We let $G = \CH_0(X_n|D)^\star$ (resp. $(\CH_0(X_n|D)^\star)_\tor$) if  $p >  1$ (resp.
  $p = 1$). We consider the commutative diagram
\begin{equation}\label{eqn:BMP-sing-M2-3}
  \xymatrix@C1.3pc{
    \Br^{\lw}(X) \ar[r]^-{\beta_X} \ar@{^{(}->}[d] & F \ar@{^{(}->}[d] \\
    \Br(X_n|D) \ar[r]^-{\beta_{X_n|D}} & G.}
  \end{equation}
 
The right vertical arrow is injective by \corref{cor:Pic-sequence-1}.
  Since $\beta_{X_n|D}$ is injective by \lemref{lem:BM-pf-0}, it follows that
  $\beta_X$ is injective. To prove its surjectivity, we let $\chi \in F$.
  If we consider $\chi$ as an element of $G$ , then \lemref{lem:Br-Main-00}
  says that $\chi = \beta_{X_n|D}(w)$ for some $w \in  \Br(X_n|D)$. In this case,
  it follows directly from the definition of $\Br^{\lw}(X)$ and \lemref{lem:LW-0}
  that $w \in \Br^{\lw}(X)$. This implies subsequently that
  $\beta_X(w) = \chi$. This concludes the proof of the theorem.
\end{proof}

\vskip .4cm

\noindent\emph{Acknowledgements.}
JR and SS would like to thank IISc, Bangalore for hosting them
during the work on this paper.

\end{document}